\documentclass[reqno]{amsart}

\usepackage{amsmath,amssymb,enumerate}
\usepackage{mathscinet}

\usepackage{graphicx}

\makeatletter
\@namedef{subjclassname@2020}{%
  \textup{2020} Mathematics Subject Classification}
\makeatother

\newtheorem{Theorem}{Theorem}[section]
\newtheorem{Lemma}[Theorem]{Lemma}
\newtheorem{Proposition}[Theorem]{Proposition}
\newtheorem{Corollary}[Theorem]{Corollary}

\theoremstyle{definition}
\newtheorem{Remark}[Theorem]{Remark}
\newtheorem{Example}[Theorem]{Example}
\newtheorem{Definition}[Theorem]{Definition}

\newcommand{\thref}[1]{Theorem \ref{#1}}
\newcommand{\prref}[1]{Proposition \ref{#1}}
\newcommand{\leref}[1]{Lemma \ref{#1}}
\newcommand{\reref}[1]{Remark \ref{#1}}
\newcommand{\coref}[1]{Corollary \ref{#1}}
\newcommand{\seref}[1]{Section \ref{#1}}

\numberwithin{equation}{section}

\newcommand{\Fset}{\mathbb{F}}
\newcommand{\Zset}{\mathbb{Z}}
\newcommand{\Rset}{{\mathbb{R}}}
\newcommand{\Cset}{{\mathbb{C}}}
\newcommand{\Tset}{{\mathbb{T}}}
\newcommand{\Nset}{\mathbb{N}}

\newcommand{\Dset}{{\mathbb{D}}}
\newcommand{\oDset}{\overline{{\mathbb{D}}}}

\newcommand{\gb}{\bar{g}}
\newcommand{\hb}{\bar{h}}
\newcommand{\zb}{\bar{z}}
\newcommand{\wb}{\bar{w}}

\newcommand{\cB}{{\mathcal B}}
\newcommand{\cBt}{\tilde{{\mathcal B}}}

\newcommand{\cI}{{\mathcal I}}
\newcommand{\cZ}{{\mathcal Z}}

\newcommand{\fR}{{\mathfrak R}}

\newcommand{\sP}{\mathsf{P}}
\newcommand{\sPt}{\tilde{\mathsf{P}}}

\newcommand{\fM}{{\mathfrak M}}

\newcommand{\cL}{{\mathcal L}}
\newcommand{\cM}{{\mathcal M}}
\newcommand{\cF}{{\mathcal F}}
\newcommand{\cK}{{\mathcal K}}
\newcommand{\cR}{{\mathcal R}}
\newcommand{\cS}{{\mathcal S}}
\newcommand{\cT}{{\mathcal T}}
\newcommand{\cTt}{\tilde{{\mathcal T}}}

\newcommand{\cFt}{\tilde{{\mathcal F}}}
\newcommand{\cKt}{\tilde{{\mathcal K}}}

\newcommand{\Pt}{\tilde{P}}
\newcommand{\pt}{\tilde{p}}
\newcommand{\ph}{\hat{p}}
\newcommand{\qt}{\tilde{q}}

\newcommand{\nt}{\tilde{n}}
\newcommand{\mt}{\tilde{m}}

\newcommand{\qr}{\overleftarrow{q}}

\newcommand{\prs}[1]{{\overleftarrow{p_{#1}}}}
\newcommand{\psb}[1]{{\overline{p_{#1}}}}

\newcommand{\At}{{\tilde{A}}}                                    
\newcommand{\Bt}{{\tilde{B}}}

\newcommand{\Ran}{\mathrm{range}}
\newcommand{\Sym}{\mathrm{Sym}}
\newcommand{\im}{\mathrm{Im}}

\newcommand{\Proj}{\mathrm{Proj}}

\DeclareMathOperator*{\res}{Res}

\def\deg{\mathop{\rm deg}}
\def\Span{\mathrm{span}}
\makeatletter
\def\adots{\mathinner {\mkern 1mu\raise 0\p@ \vbox {\kern 7\p@ \hbox {.}}\mkern 2mu \raise 3\p@ \hbox {.}\mkern 2mu\raise 6\p@ \hbox {.}\mkern 1mu}}
\makeatother

\title{Bernstein--Szeg\H{o} measures in the plane}

\author[J.~Geronimo]{Jeffrey~S.~Geronimo}
\address{JG, School of Mathematics, Georgia Institute of Technology, Atlanta, GA 30332--0160, USA}
\email{geronimo@math.gatech.edu}

\author[P.~Iliev]{Plamen~Iliev$^*$}
\address{PI, School of Mathematics, Georgia Institute of Technology, Atlanta, GA 30332--0160, USA}
\email{iliev@math.gatech.edu}
\thanks{$^*$PI gratefully acknowledges support by a grant from the Simons Foundation and a CRM-Simons Professorship at the Centre de Recherches Math\'ematiques, Universit\'e de Montr\'eal.}

\begin{document}

\begin{abstract}
We define a class of Bernstein--Szeg\H{o} measures on $\Rset^2$ and we establish their spectral properties, providing a natural extension of the one-dimensional theory. We also derive conditions involving finitely many moments, which are new in the two-dimensional setting, and which completely characterize these measures. A key ingredient in the theory on the real line stems from the fact that a measure $\mu$ on $\Rset$ determines a unique sequence of orthonormal polynomials which gives a simple formula for $d\mu/dx $ in the Bernstein--Szeg\H{o} family. Since there is no canonical way to introduce orthonormal polynomials in the plane, our extension is based on a new identity which connects a Fej\'er--Riesz factorization of the weight to a polynomial depending on three variables associated with $\mu$. Using recent results in the bivariate trigonometric Fej\'er--Riesz factorization problem, we define a nontrivial two-dimensional extension of the Szeg\H{o} mapping which provides explicit orthonormal bases of the spaces associated with  Bernstein--Szeg\H{o} measures on $\Rset^2$. An important part of the paper is devoted to a self-contained development of the Bernstein--Szeg\H{o} theory for matrix-valued functionals. The proofs combine techniques from real analysis, complex analysis and algebra.
\end{abstract}

\date{March 20, 2026}

\maketitle

\tableofcontents

\section{Introduction}\label{se1}
\subsection{Bernstein--Szeg\H{o} measures on $\Rset$}\label{ss1.1}
An important class of measures on $\Rset$ introduced by Bernstein and Szeg\H{o} \cite{Szego} consists of the measures of the form
\begin{equation}\label{1.1}
d\mu=\frac{2}{\pi}\frac{\sqrt{1-x^2}}{Q(x)}\chi_{(-1,1)}(x)\,dx, 
\end{equation}
where $Q(x)$ is a polynomial nonvanishing on $(-1,1)$, with at most simple zeros at $x=\pm 1$ and $\chi_{J}$ denotes the characteristic function of a set $J$. Recall that if $\{p_k(x)\}_{k=0}^{\infty}$ are orthonormal polynomials with respect to a measure $\mu$ on the real line, then  the multiplication by $x$ can be represented by a three-term operator
\begin{equation}\label{1.2}
a_{k+1}p_{k+1}(x)+b_kp_k(x)+a_kp_{k-1}(x)=xp_k(x).
\end{equation}
Suppose that $Q(x)$ is a polynomial of degree at most $2n$ for some positive integer $n$, and let $q(z)$ denote the stable Fej\'er--Riesz factor of $Q(x)$, i.e. $q(z)$ is the unique polynomial with real coefficients and no zeros in the closed unit disk, except possibly for simple zeros at $z=\pm 1$, such that 
\begin{equation}\label{1.3}
Q(x)= q(z)  q(1/z), \quad \text{ where }\quad x=\frac{1}{2}\left(z+\frac{1}{z}\right),
\end{equation}
normalized so that $q(0)>0$. We can define orthonormal polynomials with respect to $\mu$ in \eqref{1.1} by 
\begin{equation}\label{1.4}
p_k(x)=\frac{z^{k+1}q(1/z)-z^{-k-1}q(z)}{z-1/z} \quad \text{for }k\ge n.
\end{equation}
The last equation implies that 
\begin{equation}\label{1.5}
a_{k+1}=\frac{1}{2}\quad\text{ and }\quad b_k=0 \qquad\text{ for } k\geq n.
\end{equation}
Conversely, suppose now that \eqref{1.5} holds. We can restrict our attention to $p_0(x)$, \dots, $p_{n}(x)$, or more generally, we can consider a positive linear functional $\cL$ defined on the space $\Rset_{2n}[x]$ of polynomials of degree at most $2n$, with  orthonormal polynomials $p_0(x),\dots,p_{n}(x)$. Since every positive linear functional $\cL$ on $\Rset_{2n}[x]$ can be extended to a positive linear functional on the space of all polynomials using \eqref{1.5}, the natural question one might ask is what conditions on $\cL$ guarantee the existence of a polynomial $Q(x)\in \Rset_{2n}[x]$ such that
\begin{equation}\label{1.6}
\cL(f)=\frac{2}{\pi}\int_{-1}^{1}\frac{f(x)\sqrt{1-x^2}}{Q(x)}\,dx \quad \text{ for all }f\in \Rset_{2n}[x].
\end{equation}
Using \eqref{1.4} as hint, we can give a simple answer as follows: for a positive linear functional  $\cL:\Rset_{2n}[x]\to\Rset$ with orthonormal polynomials $\{p_k(x)\}_{k=0}^{n}$ equation~\eqref{1.6} holds for some  $Q(x)\in \Rset_{2n}[x]$ if and only if 
\begin{equation}\label{1.7}
q(z)=z^{n}\left(p_{n}(x)-2za_{n}p_{n-1}(x)\right)\neq 0\qquad \text{ for }\qquad z\in (-1,1),
\end{equation}
see \cite{DS,DN,GC,GeI3,GS,VanAssche} and the references therein. Moreover, in this case $Q(x)$ is unique and can be computed using the polynomial $q(z)$ in \eqref{1.7} and \eqref{1.3}. 
These types of measures have appeared in the investigations of moment sequences of measures on subsets of the real line in probability \cite{DTV},  Gaussian quadrature formulas on $[-1,1]$ \cite{BCGM,Peherstorfer}, cubature rules for unitary Jacobi ensembles and symmetric functions \cite{vDE1,vDE2}, asymptotics of extremal errors in approximation theory \cite{LL1,LL2}, and the Bethe Ansatz equations in the spectral analysis of quantum particle models \cite{vD}.

\subsection{Bernstein--Szeg\H{o} measures on $\Rset^2$}\label{ss1.2}
The goal of this article is to address the following questions:
\begin{enumerate}[(I)]
\item Are there bivariate extensions of the Bernstein--Szeg\H{o} weights on the real line \eqref{1.1} which possess spectral properties similar to  \eqref{1.5}?\label{(I)}
\item How do we characterize the measures in (\ref{(I)})?  \label{(II)}
\item Are there analogs of \eqref{1.4}? \label{(III)}
\end{enumerate}
In the remaining part of the introduction, we will go over some of the main results in the paper which answer the questions raised above. We begin by explaining how we define orthonormal polynomials and recurrence coefficients. Unlike the one-variable theory, there is no canonical way to introduce orthonormal polynomials in the plane. Starting with \cite{Jackson}, it is customary to replace the orthonormal polynomial $p_k(x)$ in the one-dimensional setting with the space $\Pi_k[x,y]$ consisting of all polynomials of total degree $k$ which are orthogonal to all polynomials of total degree at most $k-1$. Then, by picking an orthonormal basis of each $\Pi_k$, we replace the three-term recurrence relation with two matrix relations which correspond to the multiplications by $x$ and $y$, respectively. Since $\dim(\Pi_k)=k+1$ the coefficient $a_{k+1}$ in \eqref{1.2} is replaced by rectangular $(k+1)\times (k+2)$ matrices $A_{x,k+1}$ and $A_{y,k+1}$, see \cite{DX,GK,Suetin,Xu} for a detailed account. For instance, for the Chebyshev measure 
\begin{equation}\label{1.8}
d\mu=\frac{4}{\pi^2}\sqrt{1-x^2}\sqrt{1-y^2}\,\chi_{(-1,1)^2}(x,y)\, dx\,dy
\end{equation}
we can choose bases so that the matrices have the block structure $A_{x,k+1}=\frac{1}{2}[I_{k+1}|0]$, $A_{y,k+1}=\frac{1}{2}[0|I_{k+1}]$, $B_{x,k}=0=B_{y,k}$, where $I_{k+1}$ is the identity $(k+1)\times (k+1)$ matrix. However, these formulas do not hold beyond the trivial weight in \eqref{1.8} for any $k$, even if we consider simple examples like products $w_1(x)w_2(y)$ of two Bernstein--Szeg\H{o} weights on $\Rset$.

A different way to analyze two-dimensional measures and orthogonal polynomials was introduced in \cite{DGK} and used in the work \cite{GeWo,GeWo2} to solve the bivariate Fej\'er--Riesz factorization problem. Within the context of $\Rset^2$ this approach was developed in \cite{DGIM} where the spaces $\Pi_k[x,y]$ were replaced by the following 
spaces of orthogonal polynomials with respect to a measure, or more generally a positive linear functional $\cL$
\begin{subequations}\label{1.9}
\begin{align}
&\sP_{k,\ell;\cL}[x,y]=\Rset_{k,\ell}[x,y]\ominus \Rset_{k-1,\ell}[x,y], \label{1.9a} \\
&\sPt_{k,\ell;\cL}[x,y]=\Rset_{k,\ell}[x,y]\ominus \Rset_{k,\ell-1}[x,y]. \label{1.9b}
\end{align}
\end{subequations}
In the above equations $k$ and $\ell$ are nonnegative integers and $\Rset_{k,\ell}[x,y]$ denotes the space of polynomials with real coefficients in $x$ and $y$ of degrees at most $k$ in $x$ and $\ell$ in $y$, see Figure 1 for an illustration.
\begin{figure}[h! ]
\includegraphics[scale=0.35]{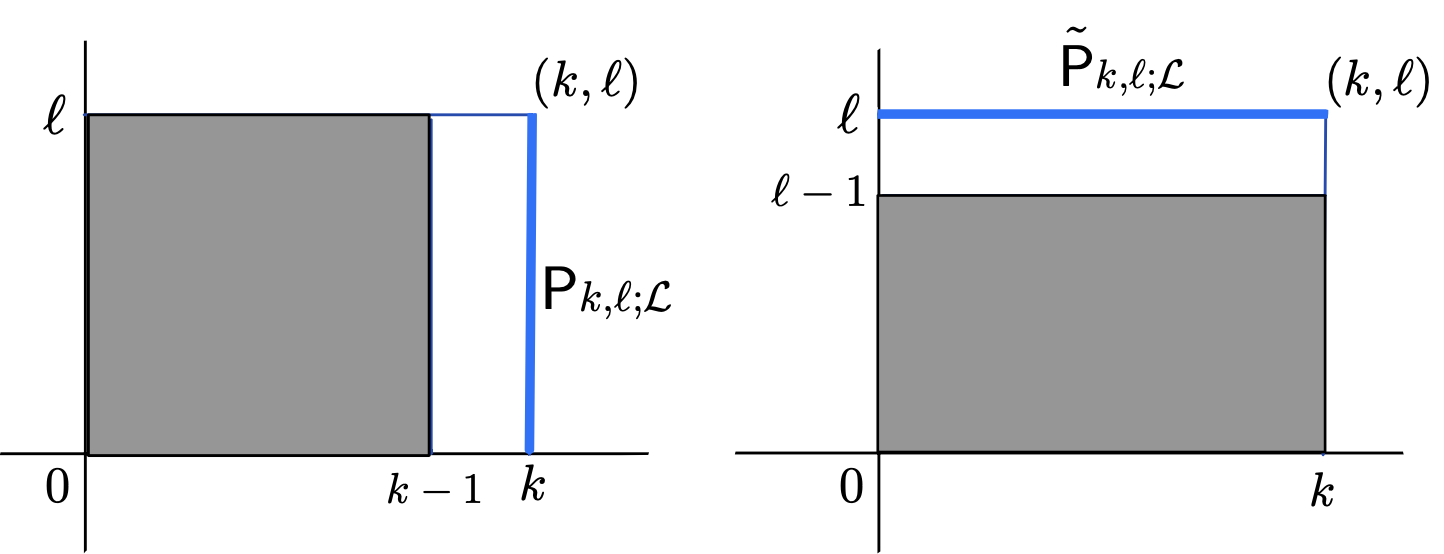}
\caption{$\sP_{k,\ell;\cL}$ and $\sPt_{k,\ell;\cL}$ are orthogonal to $x^{i}y^{j}$ for $(i,j)$ in the shaded regions above.}
\end{figure}

For every $\ell$ we fix a basis 
\begin{subequations}\label{1.10}
\begin{equation}\label{1.10a}
\cB_{\ell} =(\beta_0(y),\beta_1(y),\dots,\beta_{\ell}(y))
\end{equation}
of the space $\Rset_{\ell}[y]$ of polynomials of degree at most $\ell$ in $y$ and we define an orthonormal basis $\{p_{k,\ell}^j(x,y):0\leq j\leq \ell\}$ of the spaces $\sP_{k,\ell;\cL}[x,y]$ for all $k$ by applying the Gram-Schmidt process to the elements in the set $f_{k,\ell}^{0},\dots, f_{k,\ell}^{\ell}$ where
\begin{equation}\label{1.10b}
f_{k,\ell}^{j}=x^k \beta_j(y) -\Proj_{\Rset_{k-1,\ell}[x,y]}(x^k \beta_j(y)),
\end{equation}
\end{subequations}
and $\Proj_{V}$ denotes the orthogonal projection onto the space $V$. We set 
$$ P_{k,\ell}(x,y)=[p_{k,\ell}^0(x,y),p_{k,\ell}^1(x,y),\dots,p_{k,\ell}^{\ell}(x,y)]^t.$$  
Similarly, we fix a basis $\cBt_k$ of the space $\Rset_k[x]$ of polynomials of degree at most $k$ in $x$, which leads to an orthonormal basis $\{\pt_{k,\ell}^j(x,y): 0\leq j\leq k\}$  for $\sPt_{k,\ell;\cL}[x,y]$, and we set 
$$\Pt_{k,\ell}(x,y)=[\pt_{k,\ell}^0(x,y),\pt_{k,\ell}^1(x,y),\dots,\pt_{k,\ell}^k(x,y)]^t.$$
With these notations, it is easy to see that the above vector polynomials satisfy the following recurrence relations
\begin{subequations}\label{1.11}
\begin{align}
& x P_{k,\ell} = A_{k+1,\ell} P_{k+1,\ell} + B_{k,\ell} P_{k,\ell} +A_{k,\ell}^t P_{k-1,\ell}, \label{1.11a}\\
& y \Pt_{k,\ell} = \At_{k,\ell+1} \Pt_{k,\ell+1} + \Bt_{k,\ell} \Pt_{k,\ell} +\At_{k,\ell}^t \Pt_{k,\ell-1}, \label{1.11b}
\end{align}
\end{subequations}
where $A_{k,\ell}$ and $B_{k,\ell}$ are $(\ell+1)\times (\ell+1)$ matrices while  $\At_{k,\ell}$ and $\Bt_{k,\ell}$ are $(k+1)\times (k+1)$ matrices. The vector polynomials $P_{k,\ell}(x,y)$ can be written as 
\begin{equation*}
P_{k,\ell}(x,y)= P^{\ell}_k(x) [\beta_0(y),\beta_1(y),\dots,\beta_{\ell}(y)]^t,
\end{equation*}
where $P^{\ell}_k(x)$ is a polynomial of degree $k$ in $x$ whose coefficients are $(\ell+1)\times (\ell+1)$ matrices, and the highest coefficient is a lower-triangular matrix with positive diagonal entries. The recurrence relation \eqref{1.11a} is equivalent to the three-term relation for the matrix orthogonal polynomials $\{P^{\ell}_k(x)\}_{k\in\Nset_0}$, and we can relate \eqref{1.11b} to the theory of matrix orthogonal polynomials in a similar manner.

The vector polynomials $P_{k,\ell}$ and the matrices $A_{k,\ell}$ and $B_{k,\ell}$ in \eqref{1.11a} depend on the basis $\cB_{\ell}$ in \eqref{1.10a}, but it is easy to see that the condition 
\begin{equation*}
A_{k+1,\ell}=\frac{1}{2}I_{l+1} \quad\text{ and }\quad B_{k,\ell}=0
\end{equation*}
is independent of $\cB_{\ell}$ and the same is true for the coefficients in \eqref{1.11b}. Thus, we can formulate question (\ref{(I)}) precisely by replacing \eqref{1.5} with 
\begin{subequations}\label{1.12}
\begin{align}
&A_{k+1,\ell}=\frac{1}{2}I_{\ell+1},  && B_{k,\ell}=0, && \text{ for all }\quad k \geq n,\quad  \ell\geq m, \label{1.12a} \\
&\tilde{A}_{k,\ell+1}=\frac{1}{2}I_{k+1},  && \tilde{B}_{k,\ell}=0, && \text{ for all }\quad k \geq n,\quad  \ell\geq m, \label{1.12b}
\end{align}
\end{subequations}
for $n$ and $m$ sufficiently large. Moreover, if we consider a product measure $\mu_1(x)\times \mu_2(y)$ on $\Rset^2$, where each $\mu_j$ is absolutely continuous with respect to the Lebesgue measure on $\Rset$, then one can show that \eqref{1.12} holds for $n$ and $m$ sufficiently large if and only if both  $\mu_1(x)$ and $\mu_2(y)$ are Bernstein--Szeg\H{o} measures of the form \eqref{1.1}. It is also easy to see that if we can construct bases for the spaces $\sP_{k,\ell;\cL}[x,y]$ and $\sPt_{k,\ell;\cL}[x,y]$ such that equations \eqref{1.11} and \eqref{1.12} hold, then the bases must be obtained by the process described above.

The first nontrivial example of a measure on $\Rset^2$ for which equations  \eqref{1.12} hold was discovered in \cite[Section 3]{GeI1} as a one-parameter deformation of the Chebyshev measure \eqref{1.8}.
 It was built in a rather roundabout way from an algorithm in \cite{DGIM} which guaranteed that \eqref{1.12a}-\eqref{1.12b} hold for $n=m=1$, and inverse scattering techniques were used to derive the corresponding measure which can be written as follows
 \begin{equation}\label{1.13}
d\mu(x,y)=\frac{4}{\pi^2}\;\frac{\chi_{(-1,1)^2}(x,y)\,\sqrt{1-x^2}\sqrt{1-y^2}}
{\omega(z,w)\omega(1/z,w) \omega(z,1/w)\omega(1/z,1/w)}dx\, dy, 
\end{equation}
where $x=\frac{1}{2}\left(z+\frac{1}{z}\right)$, $y=\frac{1}{2}\left(w+\frac{1}{w}\right)$, 
$$\omega(z,w)=1-czw, \quad \text{ and }c\in (-1,1) \text{ is a free parameter.}$$
This was extended later in \cite[Theorem 4.2]{DGIX} 
where it was shown that equations \eqref{1.12} hold for  $n$ and $m$ sufficiently large if we consider measures of the form \eqref{1.13} with 
$$\omega(z,w)=\prod_{j=1}^{N}(1+c_jzw), \quad \text{ where }c_j\in (-1,1) \text{ are free parameters.}$$
While it seemed natural to believe that \eqref{1.12} must hold for arbitrary polynomials $\omega(z,w)$ nonvanishing when $|z|\leq 1$ and $|w|\leq 1$, the proof in \cite{DGIX} was rather involved, using several technical intermediate steps, see Lemmas 3.7-3.8 and Theorems 3.9-3.11 in \cite{DGIX}, and it was clear that this approach could not be easily extended. 

Our first result settles this conjecture for {\em arbitrary polynomials} $\omega(z,w)$.
\begin{Theorem}[Bernstein--Szeg\H{o} measures on $\Rset^2$]\label{th1.1}
Suppose that $n_0,n_1,m_0,m_1$ are nonnegative integers and\\ 
$\bullet$ $\omega(z,w)\in\Rset_{n_0,m_0}[z,w]$ is nonzero for $|z|<1$, $|w|<1$,\\
$\bullet$ $q_1(x)\in \Rset_{2n_1}[x]$ is positive for $x\in (-1,1)$, \\
$\bullet$ $q_2(y)\in \Rset_{2m_1}[y]$ is positive for $y\in (-1,1)$,\\
and 
\begin{equation}\label{1.14}
Q(x,y)=q_1(x)q_2(y)\omega(z,w)\omega(1/z,w)\omega(z,1/w)\omega(1/z,1/w)
\end{equation}
is such that 
\begin{equation}\label{1.15}
\iint_{(-1,1)^2}\frac{\sqrt{1-x^2}\sqrt{1-y^2}}{Q(x,y)}\,dx\,dy<\infty.
\end{equation}
Then the recurrence coefficients for the measure
\begin{equation}\label{1.16}
d\mu(x,y)=\frac{4}{\pi^2}\frac{\chi_{(-1,1)^2}(x,y)\,\sqrt{1-x^2}\sqrt{1-y^2}}{q_1(x)q_2(y)\omega(z,w)\omega(1/z,w)\omega(z,1/w)\omega(1/z,1/w)}\,dx\,dy, 
\end{equation}
satisfy equations \eqref{1.12} with $n=n_0+n_1$ and $m=m_0+m_1$.
\end{Theorem}
If $\omega(z,w)=1$ we obtain the trivial examples corresponding to product measures discussed earlier. Note that we only require $\omega(z,w)$ to be nonzero inside the bi-disk $|z|<1$, $|w|<1$ as long as the moments are finite. Thus, the theorem applies also to examples like $\omega(z,w)=2+z+w$ when $q_1(x)$ and $q_2(y)$ are positive on $[-1,1]$.

If $\qt_2(w)$ denotes the stable Fej\'er--Riesz factor of $q_2(y)$ and if we set $p(x,w)=\qt_2(w)\omega(z,w)\omega(1/z,w)$, then $Q(x,y)$ in \eqref{1.14} can be written as
\begin{equation}\label{1.17}
Q(x,y)=q_1(x)p(x,w)p(x,1/w).
\end{equation}
Our proof of \thref{th1.1} actually establishes a stronger statement, namely that the factorization in \eqref{1.17} implies \eqref{1.12a}. In order to state this result, we will also replace the measures above with positive linear functionals which will allow us to describe the characteristic properties in terms of finitely many moments. If 
$$\cL: \Rset_{2n,2m}[x,y]\to\Rset$$
is a positive linear functional with vector orthonormal polynomials $\{P_{k,\ell}(x,y)\}_{k\leq n,\ell\leq m}$ for the spaces $\sP_{k,\ell;\cL}[x,y]$  in \eqref{1.9a} we define a polynomial $\cF_{k,\ell}(\cL;\cdot)$ of three variables as follows
\begin{align}\label{1.18}
\cF_{k,\ell}(\cL;x,y,y_1)=&P_{k,\ell}(x,y)^t P_{k,\ell}(x,y_1)-4xP_{k,\ell}(x,y)^tA^t_{k,\ell}P_{k-1,\ell}(x,y_1)\nonumber\\
&\qquad\qquad\qquad+4P_{k-1,\ell}(x,y)^tA_{k,\ell}A^t_{k,\ell}P_{k-1,\ell}(x,y_1).
\end{align}
Note that the right-hand side of \eqref{1.18} is independent of the way we construct the orthonormal bases of $\sP_{k,\ell;\cL}[x,y]$ and therefore it defines a canonical polynomial associated with the positive linear functional $\cL$. 

With the polynomials $q_1(x)$ and $p(x,w)$ appearing in the factorization \eqref{1.17} we construct another polynomial $\cK_{m}(q_1,p;\cdot)$ depending on three variables by
\begin{align}\label{1.19}
&\cK_{m}(q_1,p;x,y,y_1) = \nonumber\\
&\frac{q_1(x)}{(w-1/w)(w_1-1/w_1)}\Big(\frac{(ww_1)^{-m-1}p(x,w)p(x,w_1)-(ww_1)^{m+2}p(x,1/w)p(x,1/w_1)}{1-ww_1} \nonumber\\
&\qquad+\frac{(w_1/w)^{m+1}w_1p(x,w)p(x,1/w_1)-(w/w_1)^{m+1}wp(x,1/w)p(x,w_1)}{w-w_1}\Big).
\end{align}
It is easy to see that the right-hand side of \eqref{1.19} belongs to the space $\Rset[x,y,y_1]$ hence $\cK_{m}(q_1,p;\cdot)$ is  a well-defined polynomial associated with $q_1(x)$ and $p(x,w)$.
With the above notations, we can formulate one of the main results in the paper concerning the one-sided factorization in \eqref{1.17} as follows.
\begin{Theorem}\label{th1.2}
Suppose that for some $n_1\leq n$, the polynomials $q_1(x)\in\Rset_{2n_1}[x]$ and $p(x,w)\in \Rset_{n-n_1,2m}[x,w]$ are such that
\begin{enumerate}[\rm(i)] 
\item $q_1(x)>0$ and $p(x,w)\neq 0$ when $x\in(-1,1)$, $|w|<1$,
\item $\iint_{(-1,1)^2}\frac{\sqrt{1-x^2}\sqrt{1-y^2}}{q_1(x)p(x,w)p(x,1/w)}\,dx\,dy<\infty.$
\end{enumerate}
Then  the recurrence coefficients of the positive linear functional
\begin{equation}\label{1.20}
\cL(f)=\frac{4}{\pi^2}\iint_{(-1,1)^2}f(x,y)\frac{\sqrt{1-x^2}\sqrt{1-y^2}}{q_1(x)p(x,w)p(x,1/w)}\,dx\,dy, 
\end{equation}
satisfy \eqref{1.12a} and 
\begin{align}\label{1.21}
\cF_{n,m}(\cL;\cdot)=\cK_{m}(q_1,p;\cdot).
\end{align}
\end{Theorem}
One can show that the nonvanishing conditions in \thref{th1.1} combined with the convergence \eqref{1.15} imply the conditions in \thref{th1.2} and therefore  \thref{th1.1} can be deduced as a corollary from \thref{th1.2}. Another key ingredient of \thref{th1.2} is the identity \eqref{1.21} which is crucial in the two-dimensional setting and provides a necessary condition for the opposite direction. Condition \eqref{1.21}  is missing in the one-dimensional case thanks to the  Fej\'er--Riesz lemma \cite{Fejer,Riesz}, but  the extension of this classical result to several variables leads to a series of difficult questions, related to the famous $17^{\mathrm{th}}$ problem of Hilbert which lie at the interface of analysis, algebra and algebraic geometry. We prove that the identity \eqref{1.21} combined with nonvanishing conditions similar to \eqref{1.7} {\em completely characterizes} the measures in \eqref{1.20}. The precise statement is as follows.

\begin{Theorem}\label{th1.3}
Let $\cL: \Rset_{2n,2m}[x,y]\to\Rset$ be a positive linear functional, and suppose that for some $n_1\leq n$ there exist polynomials  $q_1(x)\in\Rset_{2n_1}[x]$, $p(x,w)\in \Rset_{n-n_1,2m}[x,w]$ satisfying  \eqref{1.21}. If 
\begin{enumerate}[\rm(a)]
\item $p(x,w)\neq 0$ for $x\in (-1,1)$, $w\in(-1,1)$, and 
\item $\hat{\Psi}^m_{n}(z)=z^{n}(P^m_{n}(x)-2zA_{n,m}^tP^m_{n-1}(x))$ is invertible for $z\in(-1,1)$,
\end{enumerate}
then $q_1(x)$ and $p(x,w)$ satisfy conditions {\rm{(i)-(ii)}} in \thref{th1.2} and \eqref{1.20} holds for all $f\in\Rset_{2n,2m}[x,y]$.
\end{Theorem}
Given a functional $\cL$, there is a constructive way to check whether \eqref{1.21} holds, and thus the above theorem can be used in practice to reconstruct the polynomials  $p(x,w)$ and $q_1(x)$ from the moments. We illustrate this with several examples at the end of the paper. Moreover, one can extend \thref{th1.3} and characterize the Bernstein--Szeg\H{o} measures \eqref{1.16}. The construction of $\omega(z,w)$ from a linear functional $\cL$ combines algebraic and analytical arguments together with an extension of the well-known criterion \cite{Strintzis} for stability of bivariate polynomials on $\oDset^2$ stated in \leref{le4.13}. We refer the reader to Subsection~\ref{ss4.2} for details and \thref{th4.12} for the characterization of the Bernstein--Szeg\H{o} measures \eqref{1.16}.

The results described above provide a complete answer to questions~(\ref{(I)})--(\ref{(II)}) for the bivariate extension of the Bernstein--Szeg\H{o} theory. We now turn to the last question. Let us start with the functional in \eqref{1.20} with weight possessing one-sided factorization. Using the one-dimensional theory discussed briefly in Section~\ref{ss1.1}, one can easily see that if $\{U^{q_1}_j(x)\}_{j\in\Nset_0}$ are the orthonormal polynomials with respect to the measure $\frac{2}{\pi}\frac{\sqrt{1-x^2}}{q_1(x)}\chi_{(-1,1)}(x)dx$ on $\Rset$, and if we set 
\begin{equation*}
p_M(y;x)=\frac{w^{M+1} p (x,1/w)- w^{-M-1} p (x,w)}{w-1/w},
\end{equation*}
then the polynomials 
\begin{equation}\label{1.22}
\{p_M(y;x)U^{q_1}_j(x)\}_{j=0}^{N-n_0}, \quad \text{ where }n_0=n-n_1,
\end{equation}
are orthonormal elements in the space $\sPt_{N,M}[x,y]$ when $N\geq n$, $M\geq m$. The natural question now is whether we can extend this set to a complete orthonormal basis of the space $\sPt_{N,M}[x,y]$. Note that we have already used the stable Fej\'er--Riesz factor for the inverse of the weight, and we need $n_0$ new quantities to construct the complement. It turns out that the spaces introduced in \cite{GeI2} for the Bernstein--Szeg\H{o} measures on the bi-circle can be used to answer this question. If $p(x,w)\neq0$ for $x\in[-1,1]$ and $|w|\leq 1$, we can look at the spaces  $\sPt_{k,\ell;\Tset^2}[z,w]$ of orthogonal polynomials for the measure 
\begin{equation}\label{1.23}
\frac{1}{(2\pi)^2}\, \frac{|dz|\,|dw|}{p(x,w)p(x,1/w)} \quad \text{ on }\quad\Tset^2=\{(z,w)\in\Cset^2: |z|=|w|=1\}
\end{equation}
which are defined just like the ones for $\Rset^2$ in \eqref{1.9b}. However, in order to connect the two sets of polynomials, we need to choose $k=2n_0-1$ and $\ell=2M+1$ which means that on the bi-circle we get a space whose dimension is twice the dimension we need to complete the orthonormal set of real polynomials. A key fact discovered in  \cite{GeI2} and further analyzed in \cite{GeIK} is that the space $\sPt_{2n_0-1,2M+1;\Tset^2}[z,w]$ has a subtle decomposition as a direct sum of two subspaces $\sPt_{2n_0-1,2M+1;\Tset^2}^{1}[z,w]$ and $\sPt_{2n_0-1,2M+1;\Tset^2}^{2}[z,w]$ which possess extra orthogonality conditions and ``nice" spectral properties. Moreover, if we start with the measure \eqref{1.23}, these two subspaces are simply related to each other by a reflection in $z$, and thus we can focus on one of them, say $\sPt_{2n_0-1,2M+1;\Tset^2}^{1}[z,w]$ having exactly the dimension $n_0$ we need. The extra orthogonality and the spectral properties of this space allow us to construct a Szeg\H{o} mapping $\cS: \sPt_{2n_0-1,2M+1;\Tset^2}^{1}[z,w]  \to  \sPt_{N,M}[x,y]$ which completes the set in \eqref{1.22} to an orthonormal basis. Moreover, these complement spaces satisfy a Chebyshev relation as $N$ changes, which is a novel to 2D.
But there is an interesting twist in this story: while $\cS$ is a linear isomorphism of $\sPt_{2n_0-1,2M+1;\Tset^2}^{1}[z,w]$ onto the complement of the set in \eqref{1.22}, this map is not an isometry, and we need to modify the inner product on $\sPt_{2n_0-1,2M+1;\Tset^2}^{1}[z,w]$ in order to get an orthonormal basis of polynomials in the plane. For the Bernstein--Szeg\H{o} measures \eqref{1.16} many simplifications take place and we can construct explicit orthonormal bases on both sides in terms of fixed orthonormal polynomials on the bi-circle contained in the finite-dimensional space $\Rset_{n_0,m_0}[z,w]$ and the multiplications by $x$ and $y$ are represented by Chebyshev relations. The precise statements for the measures in \eqref{1.20} and \eqref{1.16}  are given in Theorems~\ref{th6.2} and \ref{th6.6}, respectively and the Chebyshev relations are discussed in \coref{co6.4} and \reref{re6.7}(ii).

It would be interesting to see if the absolutely continuous measures of the form given in \eqref{1.16} are completely characterized by \eqref{1.12}. A more difficult question is to drop the absolute continuity condition and to describe all possible measures on $\Rset^2$ for which equations \eqref{1.12} hold. This is well-known in the one-variable case, see for instance \cite{DS,GeI3} and the references therein. A similar spectral characterization for Bernstein--Szeg\H{o} measures on the bi-circle can be found in \cite{GeI2}. The problem in $\Rset^2$ is significantly more difficult since the weight of the absolutely continuous part of the measure is no longer simply related to a single orthonormal polynomial. Moreover, if we drop the absolute continuity, the measures can have singular continuous and discrete components. Interesting examples which illustrate this beyond the case of product measures are presented in the last section.

The paper is organized as follows. The next section provides a short introduction to the one-dimensional Bernstein--Szeg\H{o} theory for the measures in \eqref{1.1}, emphasizing the interplay between orthogonal polynomials, Hankel matrices and reproducing kernels which are important here. Most of the results in this section are known, but we have included enough details to make the presentation self-contained and to stress some subtle points which play a crucial role in the proofs. One of the new results stated in \prref{pr2.9} provides an algebraic description of the moments of  Bernstein--Szeg\H{o} measures on $\Rset$. In \seref{se3} we develop the Bernstein--Szeg\H{o} theory for matrix-valued measures (see also \cite{Kozhan,Kozhane}). This section is of independent interest and can be used as a self-contained account for the matrix theory. In \seref{se4} we prove Theorems~\ref{th1.1}-\ref{th1.3} above together with the characterization of the Bernstein--Szeg\H{o} measures \eqref{1.16}.  In \seref{se5}, we have collected some of the constructions and the results from \cite{GeI2,GeIK} for measures on the bi-circle together with several new facts needed for the measures in the plane. In \seref{se6} we define a bivariate extension of the Szeg\H{o} mapping and we use it to construct bases of orthogonal polynomials for the spaces associated with Bernstein--Szeg\H{o} measures on $\Rset^2$. \seref{se7} contains two explicit examples illustrating different constructions in the paper and the importance of some of the conditions imposed in the main theorems.

\subsection{Notations}\label{ss1.3} 
Throughout the paper, we use the following notations and conventions.\\
$\bullet$  $\Rset$ and $\Cset$ will denote the fields of real and complex numbers, respectively. \\
$\bullet$  $\Zset$ and $\Nset_0$ will denote the ring of integers and the set of nonnegative integers.\\
$\bullet$  $\Dset=\{z\in\Cset: |z|<1\}$, $\oDset=\{z\in\Cset: |z|\leq 1\}$ and $\Tset=\{z\in\Cset: |z|=1\}$ will denote the open disk, the closed disk and the unit circle in $\Cset$, respectively.\\
$\bullet$  We will use $x$ and $y$ to denote real variables, which will be related to the complex variables $z$ and $w$ via the formulas 
$$x=\frac{1}{2}\left(z+\frac{1}{z}\right)\quad \text{ and }\quad y=\frac{1}{2}\left(w+\frac{1}{w}\right).$$
$\bullet$  For a ring $\Fset$, we denote by $\Fset[a_1,\dots,a_k]$ the ring of polynomial in $a_1,\dots,a_k$ with coefficients in $\Fset$, and by $\Fset_{m_1,\dots,m_k}[a_1,\dots,a_k]$ the space of polynomials such that $a_j$ has degree at most $m_j$ for every $j=1,\dots,k$.\\
$\bullet$ $\Rset^{l\times l}$ denotes the ring of all $l\times l$ matrices with real coefficients, and $\Sym(\Rset^{l\times l})$ is the space of all symmetric matrices. In particular, for $N\in\Nset_0$ we denote by $\Rset^{l\times l}_N[x]$ the space of polynomials of degree at most $N$ in $x$ with coefficients in $\Rset^{l\times l}$. \\
$\bullet$ To unify the notation, we develop the matrix and the bivariate Bernstein--Szeg\H{o} theory working with positive linear functionals rather than measures.

\section{One-dimensional Bernstein--Szeg\H{o} weights, Hankel matrices and reproducing kernels}\label{se2}
In this section we review the connection between Bernstein--Szeg\H{o} measures on $\Rset$, inverses of Hankel matrices, reproducing kernels and the Christoffel-Darboux formula.

Recall that $y=\frac{w+1/w}{2}$ and consider measures of the form
\begin{equation}\label{2.1}
d\mu =\frac{2}{\pi}\frac{\sqrt{1-y^2}}{|q(w)|^2}\chi_{(-1,1)}(y)dy
\end{equation}
where $q(w)=q_mw^m+\cdots + q_0$ is a polynomial of degree $m$ with real coefficients 
nonvanishing for $w\in\oDset$, except for possible simple zeros at $w=\pm 1$.
Set 
\begin{equation}\label{2.2}
p_k(y)=\frac{w^{k+1}q(1/w)-w^{-k-1}q(w)}{w-1/w},
\end{equation}
and note that  
$$w^{k+1}q(1/w)\in\Span \{w^{k+1},\dots, w^{k-m+1}\} .$$  
This shows that if $k\ge m-1$ then $p_k(y)$ has degree $k$, while if $k<m-1$ the degree of  $p_k(y)$  is $\le \max(k,m-k-2)$. 
In both cases, since $q(w)$ has real coefficients, it follows that $p_k(y)$ has real coefficients which depend linearly on the coefficients of $q(w)$. Using the method in Szeg\H{o} \cite[Theorem 2.6, (2.6.3)]{Szego} we have
\begin{Proposition}\label{pr2.1}
If $2k+1\ge m$ then $p_k(y)$ is a polynomial of degree $k$ orthonormal with respect to $\mu$.
\end{Proposition}

\begin{proof}
Set $w=e^{i\varphi}$ and note that for every polynomial $f(y)$ we have
\begin{align*}
&\int f(y)p_k(y)d\mu(y)=\int_{0}^{\pi}f(y)\left(w^{k+1}q(1/w)-w^{-k-1}q(w)\right)
\frac{\sin\varphi}{\pi i |q(w)|^2}\,d\varphi\\
&\qquad=\frac{1}{\pi i}\int_{-\pi}^{\pi}\frac{f(y)w^{k+1}}{q(w)}\sin\varphi \,d\varphi 
=-\frac{1}{2\pi i}\oint_{\Tset}\frac{f(y)w^{k}}{q(w)}\left(w-\frac{1}{w}\right) \,dw.
\end{align*}
Since $q(w)$ is nonzero for $w\in\oDset$ except possibly for simple zeros at $w=\pm 1$, the Cauchy--Goursat  theorem shows the last integral is $0$ when $f(y)$ is a 
polynomial of degree less than $k$. When $f(y)=p_k(y)$ we find 
\begin{align*}
\int p_k^2(y)d\mu(y)&=-\frac{1}{2\pi i}\oint_{\Tset}
\frac{w^{k+1}q(1/w)-w^{-k-1}q(w)}{q(w)}w^{k} \,dw\\
&=-\frac{1}{2\pi i}\oint_{\Tset} \left(\frac{w^{2k+1}q(1/w)}{q(w)}-\frac{1}{w}\right)dw=\frac{1}{2\pi i} \oint_{\Tset} \frac{dw}{w}=1,
\end{align*}
completing the proof.
\end{proof}

\begin{Remark}\label{re2.2}
Note that if $m=2k+2$, the comments before \prref{pr2.1} and the proof above show that $p_k(y)$ defined in \eqref{2.2} is a polynomial of degree $k$, which is orthogonal, but not orthonormal with respect to $\mu$, since
\begin{equation}\label{2.3}
\int p_k^2(y)d\mu(y)=\frac{q_0-q_m}{q_0}.
\end{equation}
\end{Remark}

Let $h_j =\int y^jd\mu$ denote the moments with respect to the measure in \eqref{2.1}, and for $k\in\Nset_0$ let $H_k$ be the Hankel matrix
\begin{equation}\label{2.4}
H_k=
\left[
\begin{matrix}
h_{0}& h_{1} & \adots & h_{k}\\
h_{1} & h_{2}& \adots & \\[4pt]
\adots & \adots & & \adots \\[4pt]
h_{k} &  & \adots & h_{2k}
\end{matrix}
\right].
\end{equation}
Recall that if $\{p_k^o(y)\}_{k=0}^{\infty}$ are orthonormal polynomials with respect to $\mu$, then for $k\in\Nset_0$ the reproducing kernel $K_k(y,y_1)$ is defined as follows
\begin{equation}\label{2.5}
K_k(y,y_1)=p^o_0 (y)p^o_0 (y_1)+\cdots + p^o_k(y)p^o_k(y_1).
\end{equation}
On one hand, it is easy to see that
\begin{equation}\label{2.6}
K_k(y,y_1)=[1,y,\dots ,y^{k}]\,H_k^{-1}
\left[\begin{matrix} 1 , y_1 , \dots ,y_1^{k}\end{matrix}\right]^t.
\end{equation}
Indeed, the function $K_k(y,y_1)$ defined in \eqref{2.6} satisfies 
\begin{equation}\label{2.7}
\int K_k(y,y_1)y^sd\mu(y)=y_1^s,
\end{equation}
for every $s=0,1,\dots, k$, hence it must coincide with the reproducing kernel in \eqref{2.5}.
On the other hand, if we set 
$$p^o_k(y)=\sum^k_{j=0}\alpha_{k j}y^j,$$
then by the Christoffel-Darboux formula \cite[formula (3.2.3)]{Szego} we have
\begin{align}
K_k(y,y_1)=\frac{\alpha_{kk}}{\alpha_{k+1\, k+1}}
\frac{p^o_{k+1}(y)p^o_k(y_1)-p^o_k(y)p^o_{k+1}(y_1)}{y-y_1} .\label{2.8}
\end{align}
As an immediate corollary of \prref{pr2.1} and \reref{re2.2} we obtain the following explicit formula for the reproducing kernel.
\begin{Lemma}\label{le2.3} If $2k+2\ge m$, then the reproducing kernel $K_k(y,y_1)$ of the measure $\mu$ in \eqref{2.1}
is given by the formula
\begin{equation}\label{2.9}
K_k(y,y_1)=\frac{p_{k+1}(y)p_k(y_1)-p_k(y)p_{k+1}(y_1)}{2(y-y_1)},
\end{equation}
where $p_j(y)$ are the polynomials defined in \eqref{2.2}.
\end{Lemma}

\begin{proof}
If $2k+1\geq m$, then we can take $p^o_j(y)=p_j(y)$ for $j=k$ and $j=k+1$ by \prref{pr2.1}. Equation~\eqref{2.2} implies that
$\alpha_{jj}=2^{j}q_0$, which combined with \eqref{2.8} establishes \eqref{2.9}.

If $2k+2=m$, then we can take $p^o_k(y)=\sqrt{q_0/(q_0-q_m)}p_k(y)$ by \reref{re2.2} and $p^o_{k+1}(y)=p_{k+1}(y)$ by \prref{pr2.1}. This shows that $\alpha_{kk}=2^{k}\sqrt{q_0(q_0-q_m)}$, $\alpha_{k+1\, k+1}=2^{k+1}q_0$ and substituting these formulas in \eqref{2.8} we obtain \eqref{2.9}.
\end{proof}

The next lemma tells that the entries of the matrix $H^{-1}_k$ coincide with the coefficients in the expansion of the reproducing kernel $K_k(y,y_1)$ in powers of $y$ and $y_1$.

\begin{Lemma}\label{le2.4} The entries $(H^{-1}_k)_{j \ell}$ of the Hankel matrix \eqref{2.4} can be determined from the reproducing kernel \eqref{2.5} as follows
\begin{equation}\label{2.10}
K_k(y,y_1)=\sum_{j,\ell=0}^k(H^{-1}_k)_{j \ell}\, y^j y_1^{\ell}.
\end{equation}
\end{Lemma}
\begin{proof}
If we expand $K_k(y,y_1)$ in powers of $y$ and $y_1$
$$K_k(y,y_1)= \sum^k_{j,\ell=0}\beta_{j\ell} y^j y_1^{\ell},$$
then for every $s=0,\dots,k$, equation \eqref{2.7} holds and therefore 
$$ \sum^k_{j,\ell=0} \beta_{j\ell} h_{j+s}y_1^\ell=y_1^s.$$
Comparing the coefficients of $y_1^\ell$  in both sides of the last formula we see that
\begin{equation}\label{2.11}
\sum^k_{j=0} h_{j+s}\beta_{j\ell}=\delta_{s,\ell}.
\end{equation}
Since the $\ell$-th column of $H^{-1}_k$ is the unique solution $(z_0,\dots,z_{k})^t$ 
of the linear system of equations
\begin{equation*}
\sum^k_{j=0} h_{s+j}z_j=\delta_{s,\ell}
\end{equation*}
where $s=0,1,\dots,k$, we conclude from  \eqref{2.11} that $z_j=\beta_{j\ell}$.
\end{proof}

From \leref{le2.3} and \leref{le2.4}, we obtain the following important corollary.
\begin{Corollary}\label{co2.5} If $2k+2\ge m$, then the entries of $H^{-1}_k$ are homogeneous quadratic polynomials of the coefficients of $q(w)$.
\end{Corollary}

The next proposition provides a characterization of Hankel matrices in terms of reproducing kernels, or the stable Fej\'er--Riesz factors of Bernstein--Szeg\H{o} weights.

\begin{Proposition}\label{pr2.6}
For a real positive definite $(k+1)\times (k+1)$ matrix $H$ the following conditions are equivalent.
\begin{enumerate}[\rm(i)] 
\item $H$ is a Hankel matrix.
\item There exist $p_{k}(y)\in\Rset_k[y]$ and $p_{k-1}(y)\in\Rset_{k-1}[y]$ such that
\begin{align}\label{2.12}
&\frac{p_k(y)p_{k-1}(y_1)-p_k(y_1)p_{k-1}(y)}{2(y-y_1)}
+p_k(y) p_k(y_1) 
= [1,y,\dots ,y^{k}]\,H^{-1}
\left[\begin{matrix} 1\\ y_1\\ \vdots\\ y_1^{k}\end{matrix}\right].
\end{align}
\item There exists $q(w)\in\Rset_{2k}[w]$ such that 
\begin{align}
&\frac{(ww_1)^{-k-1}q(w)q(w_1)-(ww_1)^{k+2}q(1/w)q(1/w_1)}{1-ww_1} \nonumber\\
&\qquad+\frac{(w_1/w)^{k+1}w_1q(w)q(1/w_1)-(w/w_1)^{k+1}wq(1/w)q(w_1)}{w-w_1} \nonumber\\
&\quad\quad\qquad\quad\qquad \quad\qquad\quad\qquad
= [1,y,\dots ,y^{k}]\,H^{-1}
\left[\begin{matrix} 1, y_1 , \dots , y_1^{k}\end{matrix}\right]^t.\label{2.13}
\end{align}
\end{enumerate}
\end{Proposition}

\begin{proof}
The equivalence  (i)$\Leftrightarrow$(ii) follows from the work of Lander \cite{Lander} on Bezoutians, but we outline a direct proof in our setting, which allows us to reconstruct the polynomials  $p_k(y)$, $p_{k-1}(y)$ and $q(w)$ uniquely from $H$, up to an overall sign.

Suppose first that (i) holds, i.e. $H=H_k$ has the Hankel structure in \eqref{2.4}. Then $\cL(y^j)=h_{j}$ for $j=0,1,\dots,2k$ extends to a positive linear functional $\cL:\Rset_{2k}[y]\to\Rset$. If we denote by 
$\{p^o_j(y)\}_{0\leq j\leq k}$ the orthonormal polynomials with respect to $\cL$, then using formulas \eqref{2.5}, \eqref{2.6} and \eqref{2.8} we see that (ii) holds if we set $p_k(y)=p^o_k(y)$ and $p_{k-1}(y)=(2\alpha_{k-1\,k-1}/\alpha_{kk})p^o_{k-1}(y)$.

Conversely, suppose now that (ii) holds. Let $H=LL^t$ be the Cholesky factorization of $H$, and let $p^o_0(y),\dots,p^o_k(y)$ be the components of the vector $L^{-1}[1,y,\dots ,y^{k}]^t$. Then \eqref{2.12} can be rewritten as 
$$\frac{p_k(y)p_{k-1}(y_1)-p_k(y_1)p_{k-1}(y)}{2(y-y_1)}
+p_k(y) p_k(y_1) =p^o_0 (y)p^o_0 (y_1)+\cdots + p^o_k(y)p^o_k(y_1).$$
One can show now by induction on $k$ that if the last equation holds, then we can define a positive linear functional $\cL:\Rset_{2k}[y]\to\Rset$ with orthonormal polynomials $p^o_0(y),\dots,p^o_k(y)$. Then the right-hand side of \eqref{2.12} will be the reproducing kernel for this functional. Therefore $H$ will be a Hankel matrix of the form given in \eqref{2.4}, where $h_j=\cL(y^j)$. This completes the proof of the equivalence (i)$\Leftrightarrow$(ii).

To see that  (ii)$\Leftrightarrow$(iii), note that the linear mapping 
\begin{align}
\Rset_{k}[y]\times \Rset_{k-1}[y]&\to \Rset_{2k}[w]\nonumber \\
(p_{k}(y),p_{k-1}(y)) &\to q(w)=w^{k}(p_{k}(y)-wp_{k-1}(y)) \label{2.14}
\end{align}
is an isomorphism with inverse $p_j(y)=\frac{w^{j+1}q(1/w)-w^{-j-1}q(w)}{w-1/w}$, where $j=k$ and $j=k-1$. It is straightforward to check that $(p_k(y),p_{k-1}(y))$ satisfy \eqref{2.12} if and only if $q(w)$ satisfies  \eqref{2.13}, completing the proof.
\end{proof}

\begin{Remark}\label{re2.7}
Suppose that the equivalent conditions in \prref{pr2.6} hold.
If we set $p_j(y)=\sum_{l=0}^{j}\beta_{jl}y^l$ for $j=k$ and $j=k-1$, and if we compare the coefficients of $y_1^k$ on both sides of \eqref{2.12}, it follows that 
\begin{equation}\label{2.15}
\beta_{kk} p_k(y)=[1,y,\dots ,y^{k}]\,H^{-1} [0,0,\dots ,0,1] ^t.
\end{equation}
Comparing now the coefficients of $y^k$ we see that $\beta_{kk}^2=(H^{-1})_{k+1,k+1}$. In particular, this shows that $\beta_{kk}\neq 0$ and is uniquely determined from the right-hand side of \eqref{2.12}, up to a sign. We can use now  \eqref{2.15} to compute $p_k(y)$ and it is determined uniquely from $H$, up to an overall sign. Without a restriction, we can normalize $\beta_{kk}$ to be positive, and then $p_k(y)$ will coincide with the orthonormal polynomial $p^o_k(y)$  of degree $k$ with respect to the positive linear functional $\cL$ whose moment matrix is $H$. Once we compute $p_k(y)$, we can multiply both sides of  \eqref{2.12} by $2(y-y_1)$, and comparing the coefficients of $y_1^k$ we can compute $p_{k-1}(y)$. It will be a polynomial of degree $k-1$, with positive highest coefficient, which up to a positive factor is equal to the orthonormal polynomial $p^o_{k-1}(y)$ of degree $k-1$  with respect to $\cL$. Moreover, if $yp^o_{k-1}(y)=a_kp^o_k(y)+b_kp^o_k(y)+a_{k-1}p^o_{k-2}(y)$ is the three-term recurrence relation for the orthonormal polynomials, then 
\begin{equation}\label{2.16}
p^o_k(y)-2a_{k}w p^o_{k-1}(y)=p_{k}(y) -wp_{k-1}(y).
\end{equation}
This shows that $q(w)=w^k(p_{k}(y) -wp_{k-1}(y))$ in \eqref{2.13} is uniquely determined from \eqref{2.13} up to an overall sign and $q(0)\neq 0$. In particular, if we normalize the highest coefficient of $p_k(y)$ to be positive as above, then the linear mapping \eqref{2.14} will fix  the unique solution of \eqref{2.13} for which $q(0)>0$.
\end{Remark}

\begin{Remark}\label{re2.8}
One can use the computation in the proof of \prref{pr2.1} to show that if $q(w)=q_mw^m+\cdots + q_0$ is a polynomial of degree $m$ which is nonzero for $w\in\oDset$, then the moments 
\begin{equation}\label{2.17}
h_j(q)=\frac{2}{\pi} \int_{-1}^{1}y^j \frac{\sqrt{1-y^2}}{|q(w)|^2}\,dy, \qquad j\in\Nset_0
\end{equation}
for the measure in \eqref{2.1} are rational functions of the coefficients $q_0,\dots,q_m$. In particular, this shows that they can be computed explicitly using algebraic manipulations. Indeed, rewriting the integral above as an integral over the unit circle we see that
\begin{equation}\label{2.18}
h_j(q)=-\frac{1}{2^{j+2}\pi i} \oint_{\Tset}\frac{\left(w+\frac{1}{w}\right)^{j}\left(w-\frac{1}{w}\right)^2}{q(w) q(1/w)}\,\frac{dw}{w}, \qquad \text{for }\quad j\in\Nset_0,
\end{equation}
and using Cauchy's residue theorem we can evaluate the contour integral by computing the residues of the integrand at $w=0$ and at the zeros of the polynomial $\qr(w)=w^mq(1/w)$ which all lie in $\Dset$. This will give a rational symmetric function of the zeros $\qr(w)$, which can be expressed in terms of the coefficients of $q(w)$ by Vieta's formulas. The next proposition provides a more detailed information about these formulas. 
\end{Remark}

We denote by $\fR(f,g)$ the resultant of polynomials $f(w)$ and $g(w)$.

\begin{Proposition} \label{pr2.9}
Suppose that  $q(w)=q_mw^m+\cdots + q_0\in\Rset[w]$ is a polynomial of degree $m$ which is nonzero for $w\in\oDset$ and let $\qr(w)=w^mq(1/w)$. If $w_1,\dots,w_m$ denote the zeros of $q(w)$, then
\begin{equation}\label{2.19}
\fR(q,\qr)=(-1)^mq(1)q(-1)\Delta_m(q)^2, 
\end{equation}
where
\begin{equation}\label{2.20}
\Delta_m(q)=(-q_m)^{m-1}\prod_{1\leq k<l\leq m}(w_kw_l-1)\in\Zset[q_0,\dots,q_m]
\end{equation}
is a homogeneous polynomial of degree $m-1$ in the coefficients $q_0,\dots,q_m$. Moreover, for the moments $h_j(q)$ in \eqref{2.17} we have 
\begin{equation}\label{2.21}
2^{j}\Delta_m(q)h_j(q) \in \Zset[q_0^{\pm 1},q_1,q_2,\dots,q_m].
\end{equation}
\end{Proposition}

\begin{proof}
Since $\qr(w)$ has roots $w_1^{-1},\dots,w_m^{-1}$ and the highest coefficient is $q_0$, the formula for the resultant yields
\begin{align*}
\fR(q,\qr)&=q_m^mq_0^m\prod_{k,l=1}^{m}(w_k-1/w_l)=\frac{q_m^mq_0^m}{\underbrace{(w_1\cdots w_m)^m}_{(-1)^mq_0^m/q_m^m}}\prod_{k,l=1}^{m}(w_kw_l-1)\\
&=(-1)^m\left(q_m^2\prod_{k=1}^{m}(w_k^2-1)\right) q_m^{2m-2}\prod_{1\leq k\neq l\leq m}(w_kw_l-1)\\
&=(-1)^m q(1)q(-1) \left(q_m^{m-1}\prod_{1\leq k<l\leq m}(w_kw_l-1)\right)^2,
\end{align*}
establishing the factorization in \eqref{2.19}. The sign in \eqref{2.20} is chosen so that $q_0^{m-1}$ has coefficient 1 in the expansion of $\Delta_m(q)$.

Let $\varphi_j(w)\in\Zset[w]$ be such that 
$$\left(w+\frac{1}{w}\right)^{j}\left(w-\frac{1}{w}\right)^2=\varphi_j(w)+\varphi_j(1/w).$$
Clearly, $\varphi_j(w)$ has zeros at $w=\pm 1$ and from \eqref{2.18} we see that 
$$2^{j}h_j(q)=-\frac{1}{2\pi i} \oint_{\Tset}\frac{\varphi_j(w)}{q(w) q(1/w)}\,\frac{dw}{w}.$$
Without any restriction, we can assume that the zeros $w_1,\dots,w_m$ of $q(w)$ are simple. Applying Cauchy's residue theorem it follows that 
$$2^{j}h_j(q)= \sum_{k=1}^{m}\frac{\varphi_j(w_k)w_k}{q(1/w_k) q'(w_k)}-\res_{w=0}\frac{\varphi_j(w)}{wq(w) q(1/w)}.$$
The right-hand side of the last equation is a rational symmetric function of $w_1,\dots,w_m$. Our next goal is to see what common denominator we can choose so that the numerator becomes a Laurent polynomial from the space $\Zset[w_1^{\pm 1},\dots,w_m^{\pm 1}]$.

Note that 
$$q(1/w_k)=q_m\prod_{j=1}^{m}(1/w_k-w_j)=\frac{q_m}{w_k^m}(1-w_k^2)\prod_{j\neq k}(1-w_kw_j),$$
and the term $(1-w_k^2)$ will cancel with a factor in $\varphi_j(w_k)$ since $\varphi_j(\pm 1)=0$. The term $\prod_{j\neq k}(1-w_kw_j)$ will cancel if we multiply $h_j(q)$ by $\Delta_m(q)$. Next, note that 
$q'(w_k)=q_m\prod_{j\neq k}(w_k-w_j)$, so all these terms will be canceled if we multiply $h_j(q)$ by the Vandermonde determinant $V(w_1,\dots,w_m)=\prod_{1\leq j < k\leq m}(w_k-w_j)$. Finally, note that the residue at $w=0$ belongs to $\Zset[w_1^{\pm 1},\dots,w_m^{\pm 1}]$. Thus, we conclude that 
$$2^{j}h_j(q)\Delta_m(q)=\frac{F(w_1,\dots,w_m)}{V(w_1,\dots,w_m)}\text{ for some } F(w_1,\dots,w_m) \in \Zset[w_1^{\pm 1},\dots,w_m^{\pm 1}].$$
The left-hand side $2^{j}h_j(q)\Delta_m(q)$ of the last equation is a symmetric function of $w_1,\dots,w_m$ and therefore $\frac{F(w_1,\dots,w_m)}{V(w_1,\dots,w_m)}$ must also be symmetric. Since $V(w_1,\dots,w_m)$ is skew-symmetric, it follows that $F(w_1,\dots,w_m)$ is also skew-symmetric and therefore divisible by each $(w_k-w_j)$ for $j<k$. This shows that $2^{j}h_j(q)\Delta_m(q)\in  \Zset[w_1^{\pm 1},\dots,w_m^{\pm 1}]$, and applying Vieta's formulas we see that 
$$2^{j}h_j(q)\Delta_m(q)\in \Zset[q_0^{\pm 1},q_1,q_2,\dots,q_{m-1}, q_m^{\pm 1}].$$ 
Finally, note that for arbitrary $q_0,q_1,\dots,q_{m-1}$ such that $|q_0|>|q_1|+\cdots +|q_{m-1}|$, the function $2^{j}h_j(q)\Delta_m(q)$ must be analytic in a neighborhood of $q_m=0$, which shows that \eqref{2.21} holds. 
\end{proof}

The formulas obtained in \prref{pr2.9} extend to case when $q(w)$ has simple roots at $w=\pm 1$.

\begin{Example}\label{Ex2.10}
Suppose $q(w)=q_4w^4+q_3w^3+q_2w^2+q_1w+q_0$ is a polynomial of degree $4$, which is nonzero for $w\in\oDset$. A straightforward computation shows that \eqref{2.19}-\eqref{2.20} hold with 
$$\Delta_4(q)=q_0^3 - q_0^2 q_2 + q_0q_1q_3 - q_0 q_3^2 - q_0^2 q_4 - 
 q_1^2 q_4 + 2 q_0 q_2 q_4 + q_1 q_3 q_4 - q_0 q_4^2 - q_2 q_4^2 + q_4^3,$$
and the first few moments \eqref{2.17} can be computed from the coefficients by the following formulas
\begin{align*}
h_0(q)&=\frac{q_0-q_4}{\Delta_4(q)}, \quad h_1(q)=\frac{q_3-q_1}{2\Delta_4(q)}\\
h_2(q)&=\frac{q_0^2 + q_1^2 - q_0q_2 - q_1q_3 + q_2q_4-q_4^2}{4q_0\Delta_4(q)}.
\end{align*}
From this computation, we can also obtain the first few moments for any polynomial of degree $j\leq 4$ simply by setting $q_{j+1}=\cdots=q_4=0$ above. For instance, if 
$$\hat{q}(w)=q_2w^2+q_1w+q_0$$ 
is a polynomial of degree $2$, we set $q_3=q_4=0$ and we see that $\Delta_4(q)=q_0^2(q_0-q_2)=q_0^2\Delta_2(\hat{q})$, where $\Delta_2(\hat{q})=q_0-q_2$. The first few moments for $\hat{q}(w)$ follow easily from the formulas above
\begin{align*}
h_0(\hat{q})=\frac{1}{q_0(q_0-q_2)}, \quad h_1(\hat{q})=-\frac{q_1}{2q_0^2(q_0-q_2)}, \quad 
h_2(\hat{q})=\frac{q_0^2 + q_1^2 - q_0q_2}{4q_0^3(q_0-q_2)},
\end{align*}
and illustrate \prref{pr2.9} in the case of a quadratic polynomial.
\end{Example}

\section{Matrix orthogonal polynomials}\label{se3}

The theory of matrix orthogonal polynomials has an extensive literature \cite{Ber, DPS} and the works \cite{Ger,Kozhan,Kozhane} will be of the most use for us. We will review these results and make more precise some of the statements found in \cite{Kozhan,Kozhane}.  The connection between the two variable problem and the theory of matrix orthogonal polynomials has been developed  in  \cite{DGIM} and the connection with scattering theory was sketched out in \cite{GeI1}. We include most of the results in this section to make the paper self contained and for the convenience of the reader.

Recall that $\Rset^{l\times l}$ denotes the space of all $l\times l$ matrices with real coefficients and $\Sym(\Rset^{l\times l})$ is the subspace of all symmetric matrices. For $N\in\Nset_0$ we denote by $\Rset^{l\times l}_N[x]$ the space of polynomials of degree at most $N$ in $x$ with coefficients in $\Rset^{l\times l}$.

A linear transformation $\cL_l: \Rset_{2N}[x]\to \Sym(\Rset^{l\times l})$ will be called a {\em matrix-valued  functional} and we define a matrix-valued inner product  on   $\Rset^{l\times l}_N[x]$ as follows: if $Q(x)=\sum_{j=0}^{N}Q_jx^j$ and $R(x)=\sum_{j=0}^{N}R_jx^j$ are two elements in $\Rset^{l\times l}_N[x]$, then  
$$\langle Q(x),R(x)\rangle_{\cL_l}= \sum_{j,k=0}^{N} Q_j \cL_{l} (x^{j+k}) R_k^{t}.$$
We say that $\cL_l$ is a {\em positive matrix-valued functional} if $\langle Q(x),Q(x)\rangle_{\cL_l}$ is a positive definite matrix for every matrix-valued polynomial $Q(x)=x^k+\sum_{j=0}^{k-1}Q_jx^j$ with highest coefficient $I_l$. With this in hand, we can apply the Gram-Schmidt process to $I_l,xI_l,\dots,x^NI_l$ to construct a sequence of matrix orthonormal polynomials $\{P_n(x)\}_{n=0}^{N}$, such that $P_n(x)\in  \Rset^{l\times l}_{n}[x]$ whose coefficient of $x^n$  is a lower triangular matrix with positive diagonal entries, and
$$\langle P_n(x),P_m(x)\rangle_{\cL_l}=\delta_{n,m}I_l. $$
It is not difficult to see that the above conditions uniquely define $P_n(x)$. 

Standard arguments show that these polynomials satisfy the following recurrence relation
\begin{equation}\label{3.1}
  x P_{n}(x) = A_{n+1} P_{n+1}(x) + B_{n}P_{n} (x)+A_{n}^t P_{n-1}(x),
\end{equation}
where $A_{n} = \langle xP_{n-1}(x),P_n(x)\rangle_{\cL_l} \in \Rset^{l\times l}$ is a lower triangular with positive diagonal entries, and $B_n= \langle xP_{n}(x),P_n(x)\rangle_{\cL_l} \in \Sym(\Rset^{l\times l})$. If $\cL_l$ is defined on  $\Rset_{2N}[x]$, then \eqref{3.1} holds for $n=0,1,\dots,N-1$ with the convention $P_{-1}(x)=0$. If $\cL_l$ extends to a positive matrix-valued linear functional on $\Rset[x]$, then  \eqref{3.1} holds for all $n\in\Nset_0$.

Our first goal is to characterize the linear functionals on $\Rset[x]$ for which the recurrence 
coefficients satisfy 
\begin{equation}\label{3.2}
A_{n+1}=\frac{1}{2}I_l \quad\text{ and }\quad B_n=0 \quad \text{ for all }\quad n\geq N,
\end{equation}
for some $N\in\Nset_0$.
Following \cite{Ger}, we introduce the matrix-valued function
\begin{equation}\label{3.3}
\Psi_{n}(z)=P_{n}(x)-2zA_{n}^tP_{n-1}(x),
\end{equation}
where $z=x-\sqrt{x^2-1}$ and the branch of the square root is chosen so that 
$z\rightarrow 0$ as $x\rightarrow +\infty$.
Using \eqref{3.1} and \eqref{3.3} one can deduce that for $n\geq 1$ we have
\begin{equation}\label{3.4}
\Psi_{n}(z)=\frac{1}{2z}A_{n}^{-1}\Psi_{n-1}(z)+
\frac{1}{2}A_{n}^{-1}\left[(I_l-4A_{n}A_{n}^t)z-2B_{n-1}
\right]P_{n-1}(x).
\end{equation}
Combing the last equation with  \eqref{3.2} we see that 
$$\hat\Psi_{n}(z)=z^n\Psi_{n}(z)=\hat\Psi_{N}(z)\ \text{ for all } n\ge N,$$
and then using \eqref{3.3} we find
\begin{equation}\label{3.5}
P_n(x)=\frac{z^{n+1}\hat\Psi_{N}(1/z)-z^{-n-1}\hat\Psi_{N}(z)}{z-1/z}\qquad \text{for } n\ge N.
\end{equation}
If $X_n(x)$ and $Y_n(x)$ are two solutions of equation~\eqref{3.1} with boundary conditions $X_{-1}(x)=Y_{-1}(x)=0$, then
routine manipulations yield
\begin{equation}\label{3.6}
X_{n-1}(x)^{\dagger}A_{n}Y_{n}(y)-X_{n}(x)^{\dagger} A^t_{n}Y_{n-1}(y)=(y-\bar x)\sum_{j=0}^{n-1} X_j(x)^{\dagger} Y_j(y),
\end{equation}
where $M^\dagger=\overline{M^t}$ denotes the Hermitian conjugate of $M$.  We set $X_n(x)=Y_n(x)=P_n(x)$ and $x\in\Rset$ in the last equation, and by taking the
limit $y\to x$ we find
\begin{equation}\label{3.7}
P_{n-1}(x)^t A_{n}P'_{n}(x)-P_{n}(x)^tA^t_{n}P'_{n-1}(x)=\sum_{j=0}^{n-1} P_j(x)^t P_j(x).
\end{equation}
Since the right-hand side of the above equation is positive definite for all $x_0\in\Rset$, it follows that there is no nonzero vector $a\in\Cset^l$ so that both $P_{n-1}(x_0)a=0$ and  $P_{n}(x_0)a=0$. This implies that there is no $x_0\in \Rset$ for which there is a vector $a\in\Cset^l\setminus\{0\}$ so that both $P_n(x_0)a=0$ and $\Psi_n(z_0)a=0$.

Taking $X_n(x)=Y_n(x)=P_n(x)$ and $y=x\in\Rset$ in \eqref{3.6} we see that 
 \begin{equation}\label{3.8}
P_{n-1}(x)^t A_{n} P_{n}(x)=P_{n}(x)^t A^t_{n} P_{n-1}(x),
\end{equation}
hence using \eqref{3.3} we find
\begin{equation}\label{3.9}
\Psi_{n}(z)^tP_{n}(x)-P_{n}(x)^t\Psi_{n}(z)=0.
\end{equation}
Finally equations \eqref{3.8} and \eqref{3.3} yield
\begin{equation}\label{3.10}
\Psi_n(z)^t\Psi_n(1/z)=\Psi_n(1/z)^t\Psi_n(z).
\end{equation}
With $X_n(x)=Y_n(x)=P_n(x)$ and $y=x$ in equation~\eqref{3.6} we find after using equation~\eqref{3.3}
\begin{align}
\frac{1}{2i}\left(P_{n}(x)^{\dagger}\frac{\Psi_{n}(z)}{z}-\frac{\Psi_{n}(z)^{\dagger}}{\bar z}P_{n}(x)\right)&=\left(\im(z)+\im\left(\frac{1}{z}\right)\right)\sum_{j=0}^{n-1} P_j(x)^{\dagger}P_j(x)
\nonumber\\
&+ \im\left(\frac{1}{z}\right)P_{n}(x)^{\dagger}P_{n}(x).\label{3.11}
\end{align}
This shows that there is no vector $a\neq0 $ so that $\Psi_n(z)a=0$ for $z\in\oDset\setminus [-1,1]$. 

If $\Phi(z)$ is an $l\times l$ matrix-valued meromorphic function defined in a neighborhood of $z_0$ such that $\det \Phi (z)$ is not identically 0, then 
\begin{itemize}
\item the {\em order of a pole} of $\Phi$ at $z_0$ is defined to be the minimal $k>0$ such that $\lim_{z\to z_0}(z-z_0)^k\Phi(z)$ is a finite nonzero matrix, and a {\em simple pole} is a pole of order 1.
\item $\Phi$ has a {\em simple zero} at $z_0$ if $\Phi(z)^{-1}$ has a simple pole at $z_0$.
\end{itemize}

A lemma of Newton and Jost \cite{NJ} that will be of use is,
\begin{Lemma}\label{le3.1}
Let $\Phi(z)$ be an $l\times l$ matrix which is analytic in the open disk\, $\Dset$, such that $\det \Phi(0)=0$ but $\det \Phi(z)\ne0$ for $z\in\Dset\setminus\{0\}$. Then the matrix $\Phi(z)$ has a simple zero at $z=0$ if and only if the relations
\begin{equation}\label{3.12}
\Phi(0)a=0, \quad \Phi(0)b+\Phi'(0)a=0,
\end{equation}
where $a$ and $b$ are constant vectors imply $a=0$.
\end{Lemma}
Suppose that $\Phi(z)$ has a simple zero at $z=0$ and write
\begin{equation}\label{3.13}
\Phi(z)=\Phi(0)+\Phi'(0)z+\cdots,
\end{equation}
and
\begin{equation}\label{3.14}
\Phi(z)^{-1}=\frac{\Phi_{-1}}{z}+\Phi_0 + \Phi_1 z+\cdots.
\end{equation}
Since $\Phi(z) \Phi(z)^{-1}=I_l=\Phi(z)^{-1} \Phi(z)$ the above equations give
\begin{equation}\label{3.15}
\Phi(0)\Phi_{-1}=0=\Phi_{-1} \Phi(0),
\end{equation}
and
\begin{equation}\label{3.16}
\Phi(0)\Phi_0+\Phi'(0)\Phi_{-1}=I_l=\Phi_0 \Phi(0)+\Phi_{-1} \Phi'(0).
\end{equation}
If equation~\eqref{3.3} is used in \eqref{3.7} we find
\begin{equation}\label{3.17}
\begin{aligned}
P_{n}(x)^t\left(\frac{\Psi_{n}(z)}{z}\right)'-\frac{\Psi_{n}(z)^t}{z}\frac{d P_{n}(x)}{d z}=&\left(1-\frac{1}{z^2}\right)\sum_{j=0}^{n-1} P_j(x)^tP_j(x)\\& -\frac{1}{z^2}P_{n}(x)^tP_{n}(x).
\end{aligned}
\end{equation}
Thus at a zero $z_0\in\oDset\setminus\{0\}$ of $\det \Psi_{n}(z)$, which must be real,
there is a nonzero real vector $a$ so that  $\Psi_{n}(z_0)a=0$. Equation~\eqref{3.17}
implies that 
\begin{equation}\label{3.18}
a^tP_{n}(x_0)^t\frac{\Psi_{n}'(z_0)}{z_0}a=\left(1-\frac{1}{z_0^2}\right)\sum_{j=0}^{n-1} a^tP_j(x_0)^tP_j(x_0)a -\frac{1}{z_0^2}a^tP_{n}(x_0)^tP_{n}(x_0)a<0.
\end{equation}
Now suppose that besides $\Psi_{n}(z_0)a=0$ there is  vector $b$ so
that $\Psi_{n}(z_0)b+\Psi'_{n}(z_0)a=0$. Hence,
$$a^tP_{n}(x_0)^t\Psi_{n}(z_0)b+a^tP_{n}(x_0)^t\Psi'_{n}(z_0)a=0.$$
An application of \eqref{3.9} yields that
$a^tP_{n}(x_0)^t\Psi_{n}'(z_0)a=0$ which contradicts equation~\eqref{3.18}. From \leref{le3.1}
we find that $\Psi_{n}(z)$ has a simple zero at $z_0$.

We summarize the above discussion with the following
\begin{Lemma}\label{le3.2}
Suppose that  $\cL_l: \Rset[x]\to \Sym(\Rset^{l\times l})$ is a positive matrix-valued functional and let $\hat\Psi_{n}(z)=z^n(P_{n}(x)-2zA_{n}^tP_{n-1}(x))$. Then $\hat\Psi_{n}(z)$ is nonsingular for $z\in\oDset\setminus ([-1,0)\cup(0,1])$, and  can have only simple zeros when $z\in  [-1,0)\cup(0,1]$.
\end{Lemma}

Note that $\hat\Psi_{N}(z)\in  \Rset^{l\times l}_{2N}[z]$ for every $N\in\Nset_0$ and from \eqref{3.10} it follows that 
\begin{equation}\label{3.19}
W_N(x)=\hat\Psi_{N}(1/z)^t\hat\Psi_{N}(z)\in \Sym(\Rset^{l\times l}_{2N}[x]).
\end{equation}
Moreover, since $\hat\Psi_{N}(0)$ is an invertible matrix, it is easy to see that $\hat\Psi_{N}(z)$ and $W_N(x)$ have the same degrees. The degree of a matrix-valued polynomial $\Phi(z)$ is the minimal $k$ for which $\Phi(z)\in  \Rset^{l\times l}_{k}[z]$ (however, the coefficient of $z^k$ need not be an invertible matrix).

Let $z_0\in\Dset\setminus\{0\}$ be a zero of $\hat\Psi_{N}(z)=z^{N}\Psi_{N}(z)$, which must be real and simple. Let $E_0$ be the orthogonal projection onto $\ker(\hat\Psi_{N}(z_0))$, and let 
$$\hat{\Psi}_N(z)^{-1}=\frac{M_{-1}}{z-z_0}+M_{0}+\cdots,$$
be the Laurent expansion of $\hat{\Psi}_N(z)^{-1}$ in a neighborhood of $z_0$. 
From equations~\eqref{3.15} and \eqref{3.16} we have
\begin{equation}\label{3.20}
\hat\Psi_{N}(z_0)M_{-1}=0=M_{-1}\hat\Psi_{N}(z_0)  
\end{equation}
and
\begin{equation}\label{3.21}
\hat\Psi'_{N}(z_0)M_{-1}+\hat\Psi_{N}(z_0)M_{0}=I_l=M_{-1}\hat\Psi'_{N}(z_0)+M_{0} \hat\Psi_{N}(z_0).
\end{equation}
We now use an argument of Newton and Jost \cite{NJ} to show that 
\begin{equation}\label{3.22}
\ker(\hat\Psi_{N}(z_0))=\Ran(M_{-1}).
\end{equation}
The inclusion $\Ran(M_{-1})\subset\ker(\hat\Psi_{N}(z_0))$ follows immediately from \eqref{3.20}. Conversely, if $a\in\ker(\hat\Psi_{N}(z_0))$, then \eqref{3.21} shows that $a=M_{-1}\hat\Psi_{N}'(z_0)a\in \Ran(M_{-1})$ which gives \eqref{3.22}. This implies that $E_0M_{-1}=M_{-1}$. If \eqref{3.2} holds we see from equation~\eqref{3.5} that
\begin{equation}\label{3.23}
P_n(x_0)A=\frac{z_0^{n+1}\hat\Psi_{N}(1/z_0)}{z_0-1/z_0}A, \qquad \text{ when }n\geq N,
\end{equation}
for $A=E_0$ or $A=M_{-1}$. Thus taking the limit $n\to \infty$ in equation~\eqref{3.17} and using the fact that $E_0$ is symmetric gives,
$$
\frac{1}{z_0-1/z_0}E_0\hat\Psi_{N}(1/z_0)^t\hat\Psi'_{N}(z_0)A=(1-1/z_0^2)\sum_{j=0}^{\infty}E_0 P_j(x_0)^t P_j(x_0) A,
$$
since $P_n(x_0)A\to0$ as $n\to\infty$ by \eqref{3.23}. If $A=E_0$ then the left hand side of the above equation is negative semi-definite. If $A=M_{-1}$ then equation~\eqref{3.21} shows that
$\hat\Psi'_{N}(z_0)M_{-1}=I_l-\hat\Psi_{N}(z_0)M_{0}$, 
so that
\begin{align*}
\frac{1}{z_0-1/z_0}E_0\hat\Psi_{N}(1/z_0)^t\hat\Psi'_{N}(z_0)M_{-1}&=\frac{1}{z_0-1/z_0}E_0\hat\Psi_{N}(1/z_0)^t (I_l-\hat\Psi_{N}(z_0)M_{0})\\
&=\frac{1}{z_0-1/z_0}E_0\hat\Psi_{N}(1/z_0)^t,
\end{align*}
where equation~\eqref{3.10} has been used to obtain the last equality.
This implies that
\begin{align}
&\frac{1}{z_0-1/z_0}E_0\hat\Psi_{N}(1/z_0)^t=(1-1/z_0^2)\sum_{j=0}^{\infty}E_0 P_j(x_0)^t P_j(x_0) M_{-1} \nonumber \\
&\quad=(1-1/z_0^2)\sum_{j=0}^{\infty}E_0 P_j(x_0)^t P_j(x_0)E_0 M_{-1}=(1-1/z_0^2)L_0 M_{-1}, \label{3.24}
\end{align}
where
\begin{align}
L_0&=\sum_{j=0}^{\infty}E_0 P_j(x_0)^t P_j(x_0)E_0+(I_l-E_0) \nonumber  \\
&=\frac{z_0}{(z_0-1/z_0)^2}E_0\hat\Psi_{N}(1/z_0)^t\hat\Psi'_{N}(z_0)E_0+(I_l-E_0)\label{3.25}
\end{align}
is an invertible positive definite matrix that commutes with $E_0$. Note that if we use the polynomial $W_N(x)$ in \eqref{3.19}, we have 
$$E_0\hat\Psi_{N}(1/z_0)^t\hat\Psi'_{N}(z_0)E_0=\frac{z_0-1/z_0}{2z_0}E_0W_N'(x_0)E_0.$$
and therefore $L_0=\frac{1}{2(z_0-1/z_0)}E_0W_N'(x_0)E_0+(I_l-E_0)$.
\begin{Definition}\label{de3.3}
If $z_0\in (-1,0)\cup(0,1)$ is zero of $\hat\Psi_{N}(z)$ and if $E_0$ denotes the orthogonal projection onto $\ker(\hat\Psi_{N}(z_0))$, then the $l\times l$ matrix
\begin{equation}\label{3.26}
\rho_0=E_0 \left(\frac{1}{2(z_0-1/z_0)}E_0W_N'(x_0)E_0+(I_l-E_0)\right)^{-1}
\end{equation}
is called a {\em canonical weight} at $x_0=\frac{1}{2}(z_0+1/z_0)$.
\end{Definition}

A definition for canonical weights can be found in \cite{Kozhan} which has been modified in \cite{Kozhane}, due to the incorrect Lemma 2.19, to be similar to the one given above. From equations \eqref{3.25} and \eqref{3.26} we see that $\rho_0=E_0L_0^{-1}$ which combined with \eqref{3.24} shows that
\begin{equation}\label{3.27}
\frac{\hat\Psi_{N}(1/z_0)}{z_0-1/z_0}\rho_0=\frac{(z_0-1/z_0)}{z_0}M_{-1}^t.
\end{equation}

We have the following,
\begin{Theorem}\label{th3.4}
Let  $\cL_l: \Rset[x]\to \Sym(\Rset^{l\times l})$ be a positive matrix-valued  functional with orthonormal 
polynomials $\{P_n(x)\}_{n=0}^{\infty}$, normalized so that the highest coefficients are  lower triangular matrices with positive diagonal entries. Let $N\in\Nset_0$ and let $\hat\Psi_{N}(z)=z^N(P_{N}(x)-2zA_{N}^tP_{N-1}(x))$. 
Then $\hat\Psi_{N}(z)$ is nonsingular for $z\in\oDset\setminus ([-1,0)\cup(0,1])$, and  can have only simple zeros when $z\in  [-1,1]\setminus\{0\}$.
Moreover, $A_{n+1}=\frac{1}{2}I_l$ and $B_n=0$ for all $n\ge N$ if and only if
\begin{equation}\label{3.28}
\cL_{l}(f)=\frac{2}{\pi}\int_{-1}^{1}f(x)\sqrt{1-x^2}\,W_N(x)^{-1}\,dx+\sum_{j=1}^{k}f(x_j) \rho_j , \quad\text{ for }f\in \Rset[x]
\end{equation}
where $W_N(x) = \hat\Psi_{N}(1/z)^t\hat\Psi_{N}(z)$,  $\rho_j$ are the canonical weights at the zeros $\{z_j\}_{j=1}^{k}$ of $ \hat\Psi_{N}(z)$ in $(-1,1)$, and $x_j=\frac{1}{2}(z_j+1/z_j)$.
If
\begin{itemize}
\item $A_{N}\ne \frac{1}{2}I_l$, then $\hat\Psi_{N}(z)$ is a polynomial of degree $2N$;
\item $A_{N}=\frac{1}{2}I_l$, but $B_{N-1}\ne 0$, then $\hat\Psi _{N}(z)$ is a polynomial of degree $2N-1$.  
\end{itemize}
\end{Theorem}

\begin{proof}
We have already established the properties of $\hat\Psi_{N}(z)$, and the statements about its degree follow from \eqref{3.4}. In particular, these properties show that the integral in \eqref{3.28} is well defined. 
We show next that the assumptions on the recurrence coefficients imply that the polynomials $P_n$ are orthonormal with respect to the functional given in \eqref{3.28}. To this end, take $s<n$ and consider 
\begin{equation}\label{3.29}
\cL_l(x^sP_n(x))=\frac{2}{\pi}\int_{-1}^{1}x^sP_n(x)\sqrt{1-x^2}\,W_N(x)^{-1}\,dx+\sum_{j=1}^{k}x_j^s P_n(x_j) \rho_j. 
\end{equation} 
If we set $z=e^{i\theta}$ and $x=\cos\theta=\frac{1}{2}\left(z+\frac{1}{z}\right)$, then for $n\ge N$ the integral in the above expression can be rewritten using \eqref{3.5}
as
\begin{align*}
\frac{2}{\pi}\int_{-1}^{1}x^sP_n(x)\sqrt{1-x^2}\,W_N(x)^{-1}\,dx
&=\frac{1}{\pi i}\int^\pi_{-\pi} z^{n+1}x^s(\sin\theta) (\hat\Psi_{N} (z)^{-1})^{t} d\theta\\
&=-\frac{1}{2\pi i}\oint_{\Tset} z^{n}x^s(z-1/z) (\hat\Psi_{N} (z)^{-1})^{t} dz.
\end{align*}
We see that the  integrand is analytic except at the zeros of $\hat\Psi_{N}(z)$ so 
the residue theorem shows,
\begin{align*}
\frac{2}{\pi}\int_{-1}^{1}x^sP_n(x)\sqrt{1-x^2}\,W_N(x)^{-1}\,dx
& = -\sum_{j=1}^{k} z_j^{n}x_j^s(z_j-1/z_j) \mathrm{res}_{z=z_j}( \hat\Psi_{N}(z)^{-1})^{t}\\
& =-\sum_{j=1}^{k}x_j^sP_n (x_j)\rho_j,
\end{align*}
where we have used the definition of  $\rho_j$, equation~\eqref{3.27}, and 
equation~\eqref{3.5}. The above equation proves  \eqref{3.29}. The fact that $\langle P_n(x),P_n(x)\rangle_{\cL_l}=I_l$ follows by the above argument after 
using equation~\eqref{3.5} to eliminate $P_n$ and utilizing the fact that the residue at $z=0$ is equal to $I_l$. For $n<N$ the result
follows by induction with the aid of the three term recurrence formula. 

To show the other direction assume that $\cL_l$ has the form in \eqref{3.28}. Construct the polynomials $P_n$ for $n\geq N$ from equation~\eqref{3.5}. Then the argument above shows that $P_n$ is an orthonormal polynomial with respect to $\cL_l$. The polynomials satisfy the recurrence formula 
$$
\frac{1}{2}P_{n+2}(x)+\frac{1}{2}P_{n}(x)=xP_{n+1}(x)
$$
for $n\ge N$ so that $A_{n+1}=\frac{1}{2} I_l$ and $B_{n+1}=0$ for $n\ge N$. 

We also have using \eqref{3.5}
$$xP_{N}(x)-\frac{1}{2}P_{N+1}(x)=\frac{1}{2}\frac{z^{N}\hat\Psi_{N}(1/z)-(1/z)^{N} \hat\Psi_{N}(z)}{z-1/z},$$
which is  polynomial of degree $N-1$ implying that $B_{N}=0$. This completes the proof.
\end{proof}

We will also need the following theorem, see \cite[Theorem 1]{YK}.
\begin{Theorem}[Matrix Fej\'er--Riesz factorization]\label{th3.5}
Let $W(x) \in\Sym(\Rset^{l\times l}_{N}[x])$ be a positive definite matrix for a.e. $x\in(-1,1)$.
Then there exists $\Psi(z) \in\Rset^{l\times l}_{N}[z]$ which is invertible for $z\in\Dset$ such that $W(x)=\Psi(1/z)^t\Psi(z)$. Moreover, $\Psi$ can be normalized so that $\Psi(0)$ is a lower triangular matrix with positive diagonal entries.
\end{Theorem}
\noindent
Note that since $\Psi(0)$ is invertible, the matrix-valued  polynomials $W(x)$ and $\Psi(z)$ in the Fej\'er--Riesz factorization theorem above will have the same degree. 
Combining the last theorem with \thref{th3.4} we can establish the following useful theorem.

\begin{Theorem}\label{th3.6}
For a positive matrix-valued  functional $\cL_l: \Rset_{2N}[x]\to \Sym(\Rset^{l\times l})$ the following conditions are equivalent:
\begin{enumerate}[\rm(i)]
\item \label{th3.6.1}
There exists $W(x) \in\Sym(\Rset^{l\times l}_{2N}[x])$ which is positive definite for $x\in(-1,1)$ and has at most simple zeros at $x=\pm 1$ such that for $f\in \Rset_{2N}[x]$ we have
\begin{equation} \label{3.30}
\cL_{l}(f)=\frac{2}{\pi}\int_{-1}^{1}f(x)\sqrt{1-x^2}\,W(x)^{-1}\,dx.
\end{equation}
\item $\hat\Psi_{N}(z)=z^N(P_{N}(x)-2zA_{N}^tP_{N-1}(x))$ is invertible for $z\in(-1,1)$.
\end{enumerate}
Moreover, if the equivalent conditions above hold, then $W(x)$ in {\rm{(\ref{th3.6.1})}} is uniquely determined from $\cL_l$ by $W(x) = \hat\Psi_{N}(1/z)^t\hat\Psi_{N}(z)$, and \eqref{3.30} defines the unique extension of $\cL_l$ to a positive matrix-valued functional on 
$\Rset[x]$ for which $A_{n+1}=\frac{1}{2}I_l$ and $B_n=0$ for all $n\ge N$.
\end{Theorem}

\begin{proof}
Note first that if $\cL_l$ is a positive matrix-valued  functional defined on $\Rset_{2N}[x]$, then there exists a unique extension of  $\cL_l$ to a positive matrix-valued functional on the space 
$\Rset[x]$ of all polynomials such that \eqref{3.2} holds. Indeed, we can use \eqref{3.1} to define $P_{n}(x)$ for $n>N$, and it follows that the highest coefficient of $P_n(x)$ will be $2^{n-N}$ times the highest coefficient of $P_N(x)$. It is easy to see that there is a unique extension of $\cL_l$ so that these polynomials are orthonormal: if $\cL_l$ is extended to $\Rset_{2n}[x]$, then the equation $ \langle xP_{n}(x),P_n(x)\rangle_{\cL_l} =0$ defines uniquely $\cL(x^{2n+1})\in \Sym(\Rset^{l\times l})$, then $\langle xP_{n}(x),P_{n+1}(x)\rangle_{\cL_l}=\frac{1}{2}I_l$ defines uniquely $\cL_l(x^{2n+2})$, and the resulting functional is positive on $\Rset_{2n+2}[x]$.

Suppose first that (i) holds. By the Fej\'er--Riesz factorization theorem, the matrix-valued polynomial  $W(x)$ can be factored as $W(x)=\Psi(1/z)^t\Psi(z)$, where $\Psi(z)\in \Rset^{l\times l}_{2N}[z]$ is nonzero for $z\in\Dset$ and $\Psi(0)$ is a lower triangular matrix with positive diagonal entries. Since $W(x)$ is invertible for $x\in(-1,1)$ and has at most simple zeros at $x=\pm 1$, it follows that $\Psi(z)$ is also invertible for $z\in\Tset\setminus\{\pm1\}$ and has at most simple zeros at $z=\pm 1$. 
We now set $P_n(x)= \frac{z^{n+1}\Psi(1/z)-z^{-n-1}\Psi(z)}{z-1/z}$ for $n\geq N$ and the contour integral given in \thref{th3.4} shows that $P_n(x)$ is an orthonormal matrix-valued polynomial of degree $n$ with respect to the functional $\cL$ in \eqref{3.30}.
The explicit formula for $P_n$ also shows that $A_{n+1}=\frac{1}{2}I_l$ and $B_n=0$ for $n\ge N$ and therefore $\Psi(z)=\hat\Psi_n(z)$ for $n\geq N$. From this, we deduce (ii) and the formula for $W(x)$. 
The implication  (ii)$\Rightarrow$(i) follows immediately from \thref{th3.4}. 
\end{proof}

\section{Two variable Bernstein--Szeg\H{o} measures}\label{se4}

Recall that $\Rset_{n,m}[x,y]$ denotes the space of polynomials with real coefficients in $x$ and $y$ of degrees at most $n$ in $x$ and $m$ in $y$. A linear functional $\cL: \Rset_{2n,2m}[x,y]\to\Rset$ is said to be {\em positive} if $\cL(p^2)>0$ for every nonzero polynomial $p\in \Rset_{n,m}[x,y]$. Using $\cL$, we define an inner product on $\Rset_{n,m}[x,y]$ by 
$$\langle p,q\rangle_{\cL} =\cL(pq),\qquad p,q\in \Rset_{n,m}[x,y],$$
and we consider the spaces of orthogonal polynomials in \eqref{1.9} with respect to this inner product.
For every $l\leq m$, we fix a basis $\cB_l$ of  $\Rset_l[y]$ which leads to an orthonormal basis $\{p_{k,l}^j(x,y):0\leq j\leq l\}$ of $\sP_{k,l;\cL}[x,y]$ as 
explained briefly in \seref{ss1.2} and we set 
$ P_{k,l}(x,y)=[p_{k,l}^0(x,y),p_{k,l}^1(x,y),\dots,p_{k,l}^l(x,y)]^t.$
Explicitly, we apply the Gram-Schmidt process consecutively to the elements in the sets $\cB_{l}\cup x \cB_{l}\cup x^{2}\cB_{l}\cup\cdots =\{\beta_{0}(y),\beta_{1}(y),\dots, \beta_{l}(y)\}\cup \{x\beta_{0}(y),x\beta_{1}(y),\dots, x\beta_{l}(y)\}\cup \{x^{2}\beta_{0}(y),x^{2}\beta_{1}(y),\dots, x^{2}\beta_{l}(y)\}\cup\cdots$ and we put the first $l+1$ elements in $P_{0,l}(x,y)$, the next $l+1$ in $P_{1,l}(x,y)$, etc.
\begin{figure}[h! ] 
\includegraphics[scale=0.4]{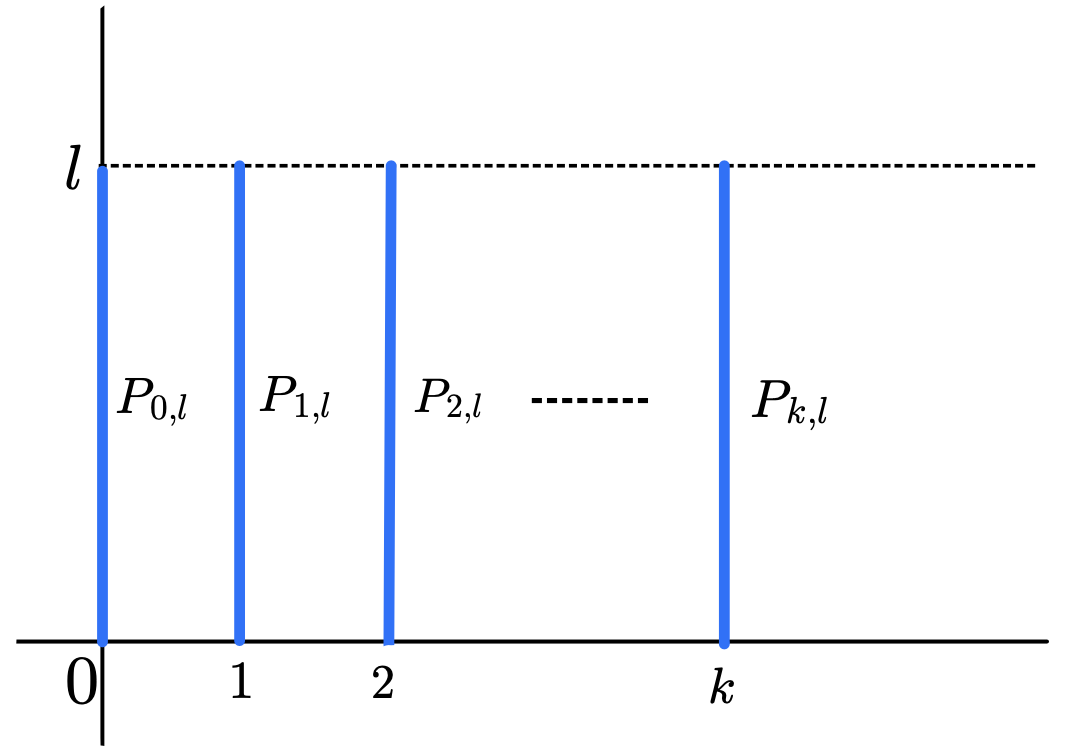}
\caption{Each $P_{k,l}$ can be represented by a line segment corresponding to the span of new monomials in the $(k+1)$st group.}
\end{figure}

Similarly,  for $k\leq n$, we fix a basis $\cBt_k$ of the space $\Rset_k[x]$ of polynomials of degree at most $k$ in $x$, which leads to an orthonormal basis $\{\pt_{k,l}^j(x,y): 0\leq j\leq k\}$  for $\sPt_{k,l;\cL}[x,y]$, and we set 
$$\Pt_{k,l}(x,y)=[\pt_{k,l}^0(x,y),\pt_{k,l}^1(x,y),\dots,\pt_{k,l}^k(x,y)]^t.$$
With these notations, the vector polynomials satisfy the recurrence relations in \eqref{1.11}, 
where $A_{k,l} = \cL (x P_{k-1,l} P^t_{k,l})$ and $B_{k,l} = \cL ( x P_{k,l} P^t_{k,l})$ are $(l+1)\times (l+1)$ matrices, and $\At_{k,l}$ and $\Bt_{k,l}$ are $(k+1)\times (k+1)$ matrices defined in a similar manner. 
The vector polynomials $P_{k,l}(x,y)$, $\Pt_{k,l}(x,y)$ and the recurrence relations \eqref{1.11} can be naturally related to the theory of matrix orthogonal polynomials which we explore in this section. Indeed, the vector polynomials $P_{k,l}(x,y)$ can be represented as 
\begin{equation}\label{4.1}
P_{k,l}(x,y)= P^l_k(x) [\beta_0(y),\beta_1(y),\dots,\beta_l(y)]^t,
\end{equation}
where $P^l_k(x)$ is a polynomial of degree $k$ in $x$, whose coefficients are $(l+1)\times (l+1)$ matrices, and the highest coefficient is a lower-triangular matrix with positive diagonal entries. Moreover, these polynomials are orthonormal with respect to the positive matrix-valued functional $\cL_{l+1}$ defined as follows
\begin{equation}\label{4.2}
\cL_{l+1}(f)=\cL\left( f(x) H_{l,\cB_l} (y)\right),
\end{equation}
where $ H_{l,\cB_l} (y)$ is an $(l+1)\times (l+1)$ matrix with entries $(H_{l,\cB_l} (y))_{i,j}=\beta_i(y)\beta_j(y)$, for $0\leq i,j\leq l$. The recurrence relation \eqref{1.11a} is equivalent to the three-term recurrence relation for these matrix polynomials. Equation \eqref{1.11b} has a similar interpretation obtained by exchanging the roles of $x$ and $y$ in the above construction. In particular, if $\cB_l=(1,y,\dots,y^l)$ and  $\cBt_k=(1,x,\dots,x^k)$ are the standard bases of $\Rset_l[y]$ and $\Rset_k[x]$, respectively, we obtain the vector polynomials defined and studied in \cite{DGIM} which can be described explicitly as follows. We order the monomials $x^i y^j$, $j\leq l$, lexicographically 
i.e. $\{1,y,\dots, y^{l},x,xy,\dots,xy^{l},x^2,\dots\}$ and the orthonormal polynomial $p_{k,l}^j(x,y)$, $0\leq j\leq l$ corresponds to the monomial $x^ky^j$ in the above ordering. Identifying the monomial $x^ky^j$ with the point $(k,j)$ in the plane, the polynomials can be constructed by applying the Gram-Schmidt process by moving on the integer points in the blue line segments in Figure~2 from bottom to top, and from left to right, see \eqref{1.10}. The orthogonality conditions characterizing the polynomial $p_{k,l}^j(x,y)$ are shown in Figure~3.
\begin{figure}[h! ]
\includegraphics[scale=0.25]{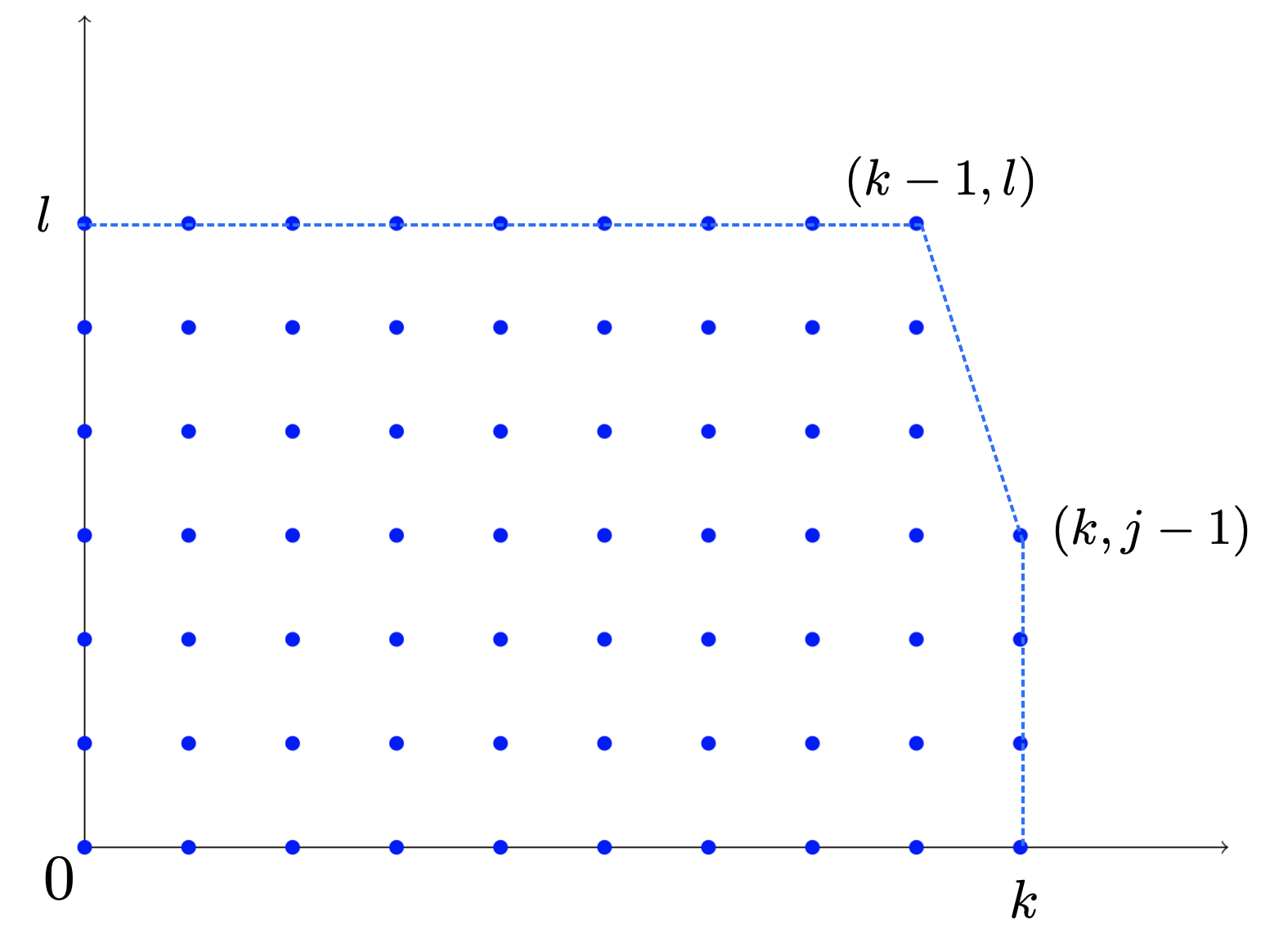}
\caption{$p_{k,l}^j(x,y) \perp x^{i}y^{s}$ where $(i,s)$ are the lattice points indicated above.}
\end{figure}

In this case $H_{l,\cB_l} (y)$ in \eqref{4.2} is a Hankel matrix.
Similarly, we can define the polynomials $\tilde p^j_{k,l}$ using reverse lexicographical ordering on the monomials in a vertical strip. Note that the construction and the orthogonality properties of these polynomials are very different from the ones obtained by the classical approach \cite{DX} based on total degree. Another important difference is that the vector polynomials $P_{k,l}(x,y)$ and $\Pt_{k,l}(x,y)$  defined above satisfy recurrence relations with square matrices, which can be used to develop scattering and inverse scattering techniques by exploring the connection with matrix-valued orthogonal polynomials. This is how the weight in \eqref{1.13} was discovered in \cite{GeI1} by applying inverse scattering techniques to a functional defined implicitly in terms of its moments. We will use a similar strategy in this section to characterize the Bernstein--Szeg\H{o} measures.

\subsection{One-sided factorization: spectral properties and characterization}\label{ss4.1}
We fix below $n,m\in\Nset_0$. We can formulate the first main result of the paper as follows.
\begin{Theorem}\label{th4.1}
\noindent
{\rm{(I)}} Suppose that for some $n_1\leq n$, the polynomials $q(x)\in\Rset_{2n_1}[x]$ and $p(x,w)\in \Rset_{n-n_1,2m}[x,w]$ are such that
\begin{enumerate}[\rm(i)] 
\item $q(x)>0$ and $p(x,w)\neq 0$ when $x\in(-1,1)$, $w\in\Dset$,
\item $\iint_{(-1,1)^2}\frac{\sqrt{1-x^2}\sqrt{1-y^2}}{q(x)p(x,w)p(x,1/w)}\,dx\,dy<\infty.$
\end{enumerate}
Then the polynomial $\cF_{n,m}(\cL;\cdot)$ defined in \eqref{1.18} for the linear functional
\begin{equation}\label{4.3}
\cL(f)=\frac{4}{\pi^2}\iint_{(-1,1)^2}f(x,y)\frac{\sqrt{1-x^2}\sqrt{1-y^2}}{q(x)p(x,w)p(x,1/w)}\,dx\,dy, 
\end{equation}
coincides with the polynomial $\cK_m(q,p;\cdot)$ constructed from $q(x)$ and $p(x,w)$ in \eqref{1.19}
\begin{align}\label{4.4}
\cF_{n,m}(\cL;\cdot)=\cK_m(q,p;\cdot),
\end{align}
and the recurrence coefficients of $\cL$  satisfy 
\begin{equation}\label{4.5}
A_{k+1,l}=\frac{1}{2}I_{l+1} \quad\text{ and }\quad B_{k,l}=0 \quad \text{ for all }\quad k \geq n,\quad  l\geq m.
\end{equation}

\noindent
{\rm{(II)}} Conversely, let $\cL: \Rset_{2n,2m}[x,y]\to\Rset$ be a positive linear functional, and suppose that
there exist polynomials  $q(x)\in\Rset_{2n_1}[x]$, $p(x,w)\in \Rset_{n-n_1,2m}[x,w]$ for some $n_1\leq n$ satisfying  \eqref{4.4}. If 
\begin{enumerate}[\rm(a)]
\item $p(x,w)\neq 0$ for $x\in (-1,1)$, $w\in(-1,1)$, and 
\item $\hat{\Psi}^m_{n}(z)=z^{n}(P^m_{n}(x)-2zA_{n,m}^tP^m_{n-1}(x))$ is invertible for $z\in(-1,1)$,
\end{enumerate}
then $q(x)$ and $p(x,w)$ satisfy conditions {\rm{(i)-(ii)}} in part {\rm{(I)}}, and \eqref{4.3} holds for all $f\in\Rset_{2n,2m}[x,y]$. In particular,  \eqref{4.3} extends $\cL$ to a positive linear functional on $\Rset[x,y]$ and \eqref{4.5} holds.
\end{Theorem}

\begin{Remark}\label{re4.2}
Recall that the construction of the polynomials $P_{k,l}(x,y)$, and the recurrence coefficients $A_{k,l}$ depend on the basis $\cB_l$ of the space $\Rset_l[y]$. However,  it is easy to see that if \eqref{4.5} holds for one basis $\cB_l$ of $\Rset_l[y]$, then \eqref{4.5} holds for any basis $\hat{\cB}_l$ of $\Rset_l[y]$. Moreover, if \eqref{4.5} holds  and if $P_{k,l}(x,y)$ and $\hat{P}_{k,l}(x,y)$ denote the vectors polynomials constructed using the bases $\cB_l$ and  $\hat{\cB}_l$, respectively, then there exists an $(l+1)\times (l+1)$ orthogonal matrix $U_l$ such that $\hat{P}_{k,l}(x,y)=U_l P_{k,l}(x,y)$ for all $k\geq n$.
\end{Remark}

\begin{Remark}\label{re4.3}
We can naturally extend the isomorphism \eqref{2.14} to an isomorphism 
\begin{align}
\Rset_{n-n_1,m}[x,y]\times \Rset_{n-n_1,m-1}[x,y]&\to \Rset_{n-n_1,2m}[x,w]\nonumber \\
(p_{m}(y;x),p_{m-1}(y;x)) &\to p(x,w)=w^{m}(p_{m}(y;x)-wp_{m-1}(y;x)), \label{4.6}
\end{align}
with inverse given by
\begin{equation}\label{4.7}
p_j(y;x) = \frac{w^{j+1}p(x,1/w) - w^{-(j+1)}p(x,w)}{w-\frac 1w}, \text{ for } j=m \text{ and }j=m-1.
\end{equation}
Using this correspondence and the arguments in \prref{pr2.6}, we see that there exists  $p(x,w)\in \Rset_{n-n_1,2m}[x,w]$ satisfying \eqref{4.4} if and only if there exist $p_{m}(y;x)\in \Rset_{n-n_1,m}[x,y]$ and $p_{m-1}(y;x)\in \Rset_{n-n_1,m-1}[x,y]$, satisfying
\begin{align}\label{4.8}
&\cF_{n,m}(\cL;x,y,y_1)=\nonumber\\
&\quad q(x)\left(\frac{p_m(y_1;x)p_{m-1}(y;x)-p_m(y;x)p_{m-1}(y_1;x)}{2(y_1-y)}
+p_m(y;x) p_m(y_1;x)\right).
\end{align}
Moreover, if the last equation holds, then $\sqrt{q(x)}p_m(y;x)$ and $\sqrt{q(x)}p_{m-1}(y;x)$ are uniquely determined from the left-hand side of \eqref{4.8}, up to a simultaneous sign change, and can be computed from the left-hand side following the steps in \reref{re2.7}. This shows that if $q(x)$ and $p(x,w)$ satisfy \eqref{4.4}, then $\sqrt{q(x)}p(x,w)$ is uniquely determined from the left-hand side, up to a sign. Clearly, if $p(x,w)$ has a factor $r(x)$ depending only on $x$ and $p(x,w)=r(x)p_1(x,w)$, then we can replace the pair $(q(x),p(x,w))$ with $(q(x)r^2(x),p_1(x,w))$, and up to such transformations $q(x)$ and $p(x,w)$ are uniquely determined from the left-hand side of \eqref{4.4}. If we want to have uniqueness, we can normalize $p(x,w)$ so that it does not have nonconstant factors depending only on $x$, then to fix the constant multiples, we can add $q(0)=1$, and finally, we can fix the sign of $p(x,w)$ by asking that $p(0,0)>0$. With this normalization, $q(x)$ and $p(x,w)$ will be uniquely determined from \eqref{4.4}. 
\end{Remark}

\begin{proof}[Proof of \thref{th4.1}, Part (I)]
We show first that if (i) and (ii) hold, then 
\begin{equation}\label{4.9}
\int_{(-1,1)}\frac{\sqrt{1-y^2}}{p(x,w)p(x,1/w)}\,dy<\infty  
\end{equation}
for all $x\in(-1,1)$. Indeed, from Tonelli's theorem we know that \eqref{4.9} holds for a.e. $x\in(-1,1)$. Note that if we fix such $x$, then $p(x,w)\neq 0$ for $w\in \oDset\setminus\{\pm 1\}$, and $p(x,w)$ can have only simple zeros at $w=\pm 1$. This means that $R(x)= \int_{(-1,1)}\frac{\sqrt{1-y^2}}{p(x,w)p(x,1/w)}\,dy$ is a rational function for a.e. $x\in(-1,1)$ for which \eqref{4.9} holds, see \reref{re2.8}. If $x_0\in(-1,1)$ is such that $\int_{(-1,1)}\frac{\sqrt{1-y^2}}{p(x_0,w)p(x_0,1/w)}\,dy=\infty $, then  by Fatou's lemma we conclude that $\lim_{x\to x_0}R(x)=\infty$. This implies that $R(x)$ has a pole of even order at $x_0$, hence  $\int_{(-1,1)}\frac{\sqrt{1-x^2}}{q(x)} R(x)\,dx=\infty$, which contradicts (ii).

We fix now $x\in(-1,1)$, and applying \leref{le2.3} we see that the right-hand side of \eqref{4.8} is equal to the reproducing kernel $K^x_m(y,y_1)$ for the measure $d\mu_x(y)=\frac{2\sqrt{1-y^2}}{\pi | \sqrt{q(x)} p(x,w)|^2}\chi_{(-1,1)}(y)dy$. We can also compute $K^x_m(y,y_1)$ using formula \eqref{2.6}, where 
\begin{equation}\label{4.10}
H_m(x)=
\frac{2}{\pi} \int_{(-1,1)}
\left[
\begin{matrix}
 1& y & \adots & y^{m}\\
y & y^{2}& \adots & \\[4pt]
\adots & \adots & & \adots \\[4pt]
y^{m} &  & \adots & y^{2m}
\end{matrix}
\right]\frac{\sqrt{1-y^2}} {|\sqrt{q(x)} p(x,w)|^2} dy.
\end{equation}

Clearly, $H_m(x)$ is a positive definite matrix for every $x\in(-1,1)$ and from \coref{co2.5} we conclude that $H_m(x)^{-1}$ is a matrix polynomial in $x$ of degree at most $2n$. If we use the standard basis $\cB_m=(1,y,\dots,y^m)$ of $\Rset_m[y]$, then the vector polynomials $P_{j,m}(x,y)$ can be represented as 
\begin{equation}\label{4.11}
P_{j,m}(x,y)= P^m_j(x) [1,y,\dots,y^m]^t,
\end{equation}
where $\{P^m_j(x)\}$ are the matrix orthogonal polynomials associated with the positive matrix functional $\cL_{m+1}: \Rset[x]\to \Sym(\Rset^{(m+1)\times (m+1)})$ defined as follows
\begin{equation}\label{4.12}
\cL_{m+1}(f)=\cL\left( f(x) \left[
\begin{matrix}
 1& y & \adots & y^{m}\\
y & y^{2}& \adots & \\[4pt]
\adots & \adots & & \adots \\[4pt]
y^{m} &  & \adots & y^{2m}
\end{matrix}
\right]  \right).
\end{equation}
This combined with \eqref{4.3} and \eqref{4.10} shows that 
$$\cL_{m+1}(f)= \frac{2}{\pi} \int_{(-1,1)} f(x)  \sqrt{1-x^2} H_m(x) dx.$$
From \thref{th3.6} it follows that \eqref{4.5} holds and that $H_m(x)^{-1} = \Psi^m_n(1/z)^t \Psi^m_n(z)$. The latter combined with \eqref{4.11} and \eqref{2.6} shows that the  reproducing kernel $K^x_m(y,y_1)$ is equal to the left-hand side of \eqref{4.8}, thus establishing \eqref{4.4}, and completing the proof of Part (I). \qed\\

\noindent
{\it Proof of  \thref{th4.1}, Part (II).} Starting with $\cL$, we define the matrix orthogonal polynomials $\{P^m_j(x)\}_{j=0,\dots,n}$ associated with the positive matrix functional $\cL_{m+1}: \Rset_{2n}[x]\to \Sym(\Rset^{(m+1)\times (m+1)})$ in \eqref{4.12} and related to the vector polynomials by \eqref{4.11}. If we set $H_m(x) = [\hat{\Psi}^m_n(1/z)^t \hat{\Psi}^m_n(z)]^{-1}$, then from condition (b) and  \thref{th3.6} we see  that $H_m(x)$ is positive definite for $x\in(-1,1)$ and
\begin{equation} \label{4.13}
\cL_{m+1}(f)=\frac{2}{\pi}\int_{(-1,1)}f(x)\sqrt{1-x^2}\,H_m(x)\,dx, \text{ for all }f\in\Rset_{2n}[x].
\end{equation}
Moreover, 
$$\cF_{n,m}(\cL;x,y,y_1)= [1,y,\dots ,y^{m}]\,H_m(x)^{-1} [1,y_1,\dots ,y_1^{m}]^t.$$
Therefore, for fixed $x\in(-1,1)$, equation \eqref{4.4} and \prref{pr2.6} tell us that $H_m(x)$ is a Hankel matrix, i.e. 
\begin{equation}\label{4.14}
H_m(x)=
\left[
\begin{matrix}
h_{0}(x)& h_{1}(x) & \adots & h_{m}(x)\\
h_{1}(x) & h_{2}(x)& \adots & \\[4pt]
\adots & \adots & & \adots \\[4pt]
h_{m}(x) &  & \adots & h_{2m}(x)
\end{matrix}
\right].
\end{equation}
Thus, for fixed $x\in(-1,1)$, we can think of $H_m(x)$ as the moment matrix of a positive linear functional $\cL^x:\Rset_{2m}[y]\to \Rset$, and \eqref{4.8} is its reproducing kernel $K_m^{x}(y,y_1)$. If $p_{m,m}(x)$  denotes the coefficient of $y^m$ in $p_m(y;x)$, then comparing the coefficients of $y^my_1^m$  on both sides of equation \eqref{4.8} we obtain
$$q(x)p_{m,m}^2(x) =[0,0,\dots,1] H_m(x)^{-1} [0,0,\dots, 1]^t>0\quad \text{ for }x\in(-1,1),$$
establishing the first condition in (i).
Using \thref{th3.6} for the functional $\cL^x$ (note that $l=1$ here, and $z$ and $\hat{\Psi}_N(z)$ there are replaced by $w$ and $\sqrt{q(x)}p(x,w)$, respectively, in view of \reref{re2.7} and \eqref{2.16}) and condition (a)
 we conclude that 
 \begin{equation} \label{4.15}
\cL^x(g)=\frac{2}{\pi}\int_{(-1,1)}\frac{\sqrt{1-y^2}}{q(x)p(x,w)p(x,1/w)}g(y)\,dy, \text{ for all }g\in\Rset_{2m}[y].
\end{equation}
Moreover, for every fixed $x\in (-1,1)$, \leref{le3.2} tells us that $p(x,w)\neq 0$ when $w\in \oDset\setminus\{\pm1\}$, thus establishing condition (i) in Part (I).
If we take $g(y)=y^j$ in \eqref{4.15} we see that 
 \begin{equation} \label{4.16}
h_j(x)=\cL^x(y^j)=\frac{2}{\pi}\int_{(-1,1)}\frac{\sqrt{1-y^2}}{q(x)p(x,w)p(x,1/w)} y^j\,dy, \text{ for }j=0,\dots,2m.
\end{equation}
Using now equations \eqref{4.12}-\eqref{4.13} with $f(x)=x^k$, $k=0,1,\dots,2n$, and then \eqref{4.14} and \eqref{4.16}, we see that 
\begin{align}
\cL(x^ky^j)&= \frac{2}{\pi}\int_{(-1,1)}x^k \sqrt{1-x^2}\,h_j(x)\,dx  \nonumber \\
&  =\frac{4}{\pi^2}\int_{(-1,1)}\left( \int_{(-1,1)}x^ky^j\,\frac{\sqrt{1-x^2} \sqrt{1-y^2}}{q(x)p(x,w)p(x,1/w)}\,dy \right)\,dx,\label{4.17}
\end{align}
for all $k=0,1,\dots, 2n$ and $j=0,1,\dots,2m$. In particular, if we take $k=j=0$ in the last formula and by Tonelli's theorem we deduce that  $\iint_{(-1,1)^2}\frac{\sqrt{1-x^2}\sqrt{1-y^2}}{q(x)p(x,w)p(x,1/w)}\,dx\,dy=\frac{\pi^2}{4}\cL(1)<\infty$, which yields condition (ii) in Part (I). Using now equation \eqref{4.17}, Fubini's theorem and the linearity of $\cL$ we see that \eqref{4.3} holds.
\end{proof}

\begin{Remark}\label{re4.4}
Note that if either conditions (i)-(ii) in Part (I) of \thref{th4.1} hold, or conditions (a)-(b) in Part (II) hold, then the proof shows that:
\begin{itemize}
\item $p(x,w)\neq 0$ for $x\in(-1,1)\times \oDset\setminus\{\pm1\}$, and  
\item $\hat{\Psi}^m_{n}(z)$ is invertible for all $z\in \oDset\setminus\{\pm1\}$. 
\end{itemize}
In the next lemma, we show that these conditions also imply that $p(x,w)\neq 0$ when $(x,w)\in \{\pm1 \}\times \Dset$. This means that when we look at the closure $[-1,1]\times \oDset$, the polynomial $p(x,w)$ can only vanish on  $\{\pm1 \}\times \Tset$ and $(-1,1)\times \{\pm 1\}$. There are examples where $p(x,w)$ vanishes on these parts of the boundary and  \thref{th4.1} can be applied with appropriate polynomials $q(x)$. For instance:\\
$\bullet$ $p(x,w)=2+x+w^2$ vanishes when $(x,w)=(-1,\pm i)$, and 
$$\frac{2}{\pi}\int_{-1}^{1}\frac{\sqrt{1-y^2}}{p(x,w)p(x,1/w)}\,dy=\frac{1}{(2+x)(1+x)},$$
by Example~\ref{Ex2.10}, hence \thref{th4.1} can be applied if $q(x)>0$ for $x\in[-1,1)$, possibly having a simple zero at $x=1$.\\
$\bullet$ $p(x,w)=(x-x_0)^2+1+w$ vanishes when $(x,w)=(x_0,-1)$, and 
$$\frac{2}{\pi}\int_{-1}^{1}\frac{\sqrt{1-y^2}}{p(x,w)p(x,1/w)}\,dy=\frac{1}{\big((x-x_0)^2+1\big)^2},$$
by Example~\ref{Ex2.10}, hence \thref{th4.1} can be applied if $q(x)>0$ for $x\in(-1,1)$, possibly having simple zeros at $x=\pm1$.
\end{Remark}

\begin{Lemma}\label{le4.5}
Suppose that $p(x,w)\in \Rset[x,w]$ is such that
\begin{enumerate} [\rm(1)] 
\item $p(x,w)\neq 0$ when $(x,w)\in(-1,1)\times\Dset$,
\item $\iint_{(-1,1)^2}\frac{\sqrt{1-x^2}\sqrt{1-y^2}}{p(x,w)p(x,1/w)}\,dx\,dy<\infty$.
\end{enumerate}
Then 
$$p(x,w)\neq 0 \quad \text{for }\quad (x,w) \in \big((-1,1)\times\oDset\setminus\{\pm1\}\big)\cup \big(\{\pm1 \}\times \Dset\big).$$
\end{Lemma}

\begin{proof}
The proof of  \thref{th4.1} shows that $p(x,w)\neq 0$ when $x\in (-1,1)$ and $w\in \oDset\setminus\{\pm1\}$, so we need to show that this is also true when $(x,w)\in \{\pm1 \}\times \Dset$. Suppose that $p(x_0,w_0)=0$ for some $x_0\in\{\pm 1\}$ and $w_0\in \Dset$. For $k\in\Nset$, set $x_k=x_0-x_0/k\in(-1,1)$ and consider the sequence of polynomials $f_k(w)=p(x_k,w)$ which do not vanish on $\Dset$. Since $f_k(w)$ converges to $f(w)=p(x_0,w)$ uniformly on $\Dset$ and $f(w_0)=0$, Hurwitz's theorem tells us that $f(w)=p(x_0,w)$ must be identically equal to $0$.  This means that $(x-x_0)$ divides $p(x,w)$ and therefore condition (2) cannot hold.
\end{proof}

As we explained in \reref{re4.3}, the polynomial $q(x)p(x,w)p(x,1/w)$ in the denominator of \eqref{4.3} is uniquely determined from $\cL$. More generally, the lemma below shows that there exists at most one polynomial $Q(x,y)\in\Rset_{2n,2m}[x,y]$ such that $\cL(f)=\frac{4}{\pi^2}\iint_{(-1,1)^2}f(x,y)\frac{\sqrt{1-x^2}\sqrt{1-y^2}}{Q(x,y)}\,dx\,dy$.

\begin{Lemma}\label{le4.6}
Let $Q_1(x,y), Q_2(x,y)\in\Rset_{2n,2m}[x,y]$ be positive for $(x,y)\in(-1,1)^2$ and such that $\iint_{(-1,1)^2}\frac{\sqrt{1-x^2}\sqrt{1-y^2}}{Q_j(x,y)}\,dx\,dy<\infty$ for $j=1$ and $j=2$. If 
\begin{equation}\label{4.18}
\iint_{(-1,1)^2}f(x,y)\frac{\sqrt{1-x^2}\sqrt{1-y^2}}{Q_1(x,y)}dx\,dy=
\iint_{(-1,1)^2}f(x,y)\frac{\sqrt{1-x^2}\sqrt{1-y^2}}{Q_2(x,y)}dx\,dy
\end{equation}
for all $f(x,y)\in\Rset_{2n,2m}[x,y]$, then  $Q_1(x,y)= Q_2(x,y)$.
\end{Lemma}

\begin{proof}
Set  $d\nu(x,y)=\frac{4}{\pi^2}\sqrt{1-x^2}\sqrt{1-y^2}\chi_{(-1,1)^2}\,dx\,dy,$
and note that \eqref{4.18} implies that 
\begin{equation*}
\iint_{(-1,1)^2}\frac{Q_2(x,y)}{Q_1(x,y)}d\nu(x,y) =
\iint_{(-1,1)^2}\frac{Q_1(x,y)}{Q_2(x,y)}d\nu(x,y)=1.
\end{equation*}
Applying the Cauchy-Schwarz inequality in the space $L^2(\nu)$,
we see that
\begin{align*}
1&=\left(\iint_{(-1,1)^2}\sqrt{\frac{Q_2(x,y)}{Q_1(x,y)}}\sqrt{\frac{Q_1(x,y)}{Q_2(x,y)}}d\nu(x,y)\right)^2\\
&\leq \iint_{(-1,1)^2}\frac{Q_2(x,y)}{Q_1(x,y)}d\nu(x,y)
\iint_{(-1,1)^2}\frac{Q_1(x,y)}{Q_2(x,y)}d\nu(x,y)=1.
\end{align*}
Thus, the polynomials $Q_1(x,y)$ and $Q_2(x,y)$ must be equal a.e. on $(-1,1)^2$, and therefore they coincide. 
\end{proof}

\begin{Remark}\label{re4.7}
Equation \eqref{4.5} in \thref{th4.1} tells us that for every fixed $M\geq m$, we can construct orthonormal bases of the spaces $\sP_{N,M;\cL}[x,y]$ for the functional $\cL$ in \eqref{4.3} satisfying a ``Chebyshev relation" in $N$ when $N\geq n$. With the notations in the theorem and using the one-dimensional theory outlined in \seref{se2}, it follows that if $\{U^q_j(x)\}_{j\in\Nset_0}$ are the orthonormal polynomials with respect to the measure $\frac{2}{\pi}\frac{\sqrt{1-x^2}}{q(x)}\chi_{(-1,1)}(x)dx$ on $\Rset$ and if we set 
\begin{equation}\label{4.19}
p_M(y;x)=\frac{w^{M+1} p (x,1/w)- w^{-M-1} p (x,w)}{w-1/w},
\end{equation}
then $p_M(y;x)\in\Rset_{n-n_1,M}[x,y]$ and the polynomials $\{p_M(y;x)U^q_j(x)\}_{j=0}^{N-(n-n_1)}$ are orthonormal elements of $\sPt_{N,M;\cL}[x,y]$ when $N\geq n$, $M\geq m$. \\
(i) We explain in \seref{se6} how this orthonormal set can be extended to an orthonormal basis of $\sPt_{N,M;\cL}[x,y]$ which also satisfies a Chebyshev relation in $N$ in the case when $p(x,w)$ does not vanish on $[-1,1]\times \oDset$ using an appropriate basis of spaces associated with Bernstein--Szeg\H{o} measures on the torus $\Tset^2$, see \thref{th6.2}, \reref{re6.3}(ii) and \coref{co6.4}.\\
(ii) Note that the coefficient of $y^M$ in $p_M(y;x)$ is $2^Mp(x,0)$. This means that the basis constructed in (i) above can be characterized as follows. We start with the set $(p(x,0),xp(x,0),\dots, x^{N-(n-n_1)}p(x,0))$ in $\Rset_N[x]$ and we complete it to a basis $\cBt_N$ of $\Rset_N[x]$. Then the basis of $\sPt_{N,M;\cL}[x,y]$ in (i) corresponds to the orthonormal polynomials built using $\cBt_N$ as explained in Section~\ref{ss1.2}. In particular, if $p(x,0)$ is a positive constant, we can take $\cBt_N$ to be the standard basis $(1,x,\dots,x^N)$ of $\Rset_N[x]$.
\end{Remark}

\begin{Remark}\label{re4.8}
The explicit examples in \seref{se7} show that equation \eqref{4.4} does not imply conditions (a) and (b) in \thref{th4.1}(II).
\end{Remark}

\begin{Remark}\label{re4.9}
Clearly, we can reverse the roles of $x$ and $y$ in \thref{th4.1} and consider functionals of the form
\begin{equation}\label{4.20}
\cL(f)=\frac{4}{\pi^2}\iint_{(-1,1)^2}f(x,y)\frac{\sqrt{1-x^2}\sqrt{1-y^2}}{q(y)\ph(z,y)\ph(1/z,y)}\,dx\,dy, 
\end{equation}
where $q(y)\in\Rset_{2m_1}[y]$ and $\ph(z,y)\in \Rset_{2n,m-m_1}[z,y]$ are such that
\begin{enumerate}[\rm(i)] 
\item $q(y)>0$ and $\ph(z,y)\neq 0$ when $(z,y)\in \Dset \times (-1,1)$,
\item $\iint_{(-1,1)^2}\frac{\sqrt{1-x^2}\sqrt{1-y^2}}{q(y)\ph(z,y)\ph(1/z,y)}\,dx\,dy<\infty.$
\end{enumerate}
If we define polynomials in three variables analogous to \eqref{1.18}-\eqref{1.19} by reversing the roles of $x$ and $y$
\begin{align}\label{4.21}
\cFt_{k,l}(\cL;x,x_1,y)=& \Pt_{k,l}(x,y)^t \Pt_{k,l}(x_1,y)-4y\Pt_{k,l}(x,y)^t\At^t_{k,l}\Pt_{k,l-1}(x_1,y)\nonumber\\
&\qquad\qquad\qquad+4\Pt_{k,l-1}(x,y)^t\At_{k,l}\At^t_{k,l}\Pt_{k,l-1}(x_1,y),
\end{align}
and
\begin{align}\label{4.22}
&\cKt_{n}(q,\ph;x,x_1,y) = \nonumber\\
&\quad\frac{q(y)}{(z-1/z)(z_1-1/z_1)}\Big(\frac{(zz_1)^{-n-1}\ph(z,y)\ph(z_1,y)-(zz_1)^{n+2}\ph(1/z,y)\ph(1/z_1,y)}{1-zz_1} \nonumber\\
&\qquad+\frac{(z_1/z)^{n+1}z_1\ph(z,y)\ph(1/z_1,y)-(z/z_1)^{n+1}z\ph(1/z,y)\ph(z_1,y)}{z-z_1}\Big),
\end{align}
then we can formulate an analog of \thref{th4.1} by replacing \eqref{4.4} with 
\begin{align}\label{4.23}
\cFt_{n,m}(\cL;\cdot)=\cKt_n(q,\ph;\cdot),
\end{align}
and \eqref{4.5} by its tilde analog 
\begin{equation}\label{4.24}
\At_{k,l+1}=\frac{1}{2}I_{k+1} \quad\text{ and }\quad \Bt_{k,l}=0 \quad \text{ for all }\quad k \geq n,\quad  l\geq m.
\end{equation}
\end{Remark}
\subsection{Bernstein--Szeg\H{o} measures: spectral properties and characterization}\label{ss4.2}
For functionals which admit both representations \eqref{4.3} and \eqref{4.20}, the recurrence coefficients on both sides will satisfy equations \eqref{4.5} and \eqref{4.24}. Before we formulate the precise statement, we prove a lemma which tells that the inverse of the corresponding weights can only vanish at the four points $(\pm 1,\pm 1)$ on the boundary of $\Dset^2$.

\begin{Lemma}\label{le4.10}
Let $\omega(z,w)\in\Rset[z,w]$ be such that $\omega(z,w)\neq 0$ for $(z,w)\in\Dset^2$ and 
\begin{equation}\label{4.25}
\iint_{(-1,1)^2}\frac{\sqrt{1-x^2}\sqrt{1-y^2}}{\omega(z,w)\omega(1/z,w)\omega(z,1/w)\omega(1/z,1/w)}\,dx\,dy<\infty.
\end{equation}
Then 
\begin{equation}\label{4.26}
\omega(z,w)\neq 0\quad \text{ for }\quad 
 (z,w)\in\oDset^2\setminus \{(-1,-1),(-1,1),(1,-1),(1,1)\}.
 \end{equation}
\end{Lemma}

\begin{proof}
Similarly to the proof of \leref{le4.5}, we can use Hurwitz's theorem to deduce that $\omega(z,w)$ cannot vanish on $(\Dset\times \Tset)\cup (\Tset\times\Dset)$, unless it has a factor depending on just one of the variables vanishing on $\Tset$, which would contradict \eqref{4.25}. Thus we can restrict our attention to $\Tset^2$ and we need to show that $\omega(z,w)$ can vanish there only at $(\pm 1,\pm1)$. This follows by applying arguments similar to the ones we used at the beginning of the proof of \thref{th4.1}(I). We set 
\begin{equation}\label{4.27}
p(x,w)=\omega(z,w)\omega(1/z,w)
\end{equation}
and thus the denominator in the integrand in \eqref{4.25} is equal to $p(x,w)p(x,1/w)$. By Tonelli’s theorem, equation~\eqref{4.9} holds for a.e. $x\in(-1,1)$. Moreover, if $x=x_0$ is such that \eqref{4.9} holds then $p(x_0,w)\neq 0$ for $w\in \oDset\setminus\{\pm 1\}$ and $p(x_0,w)$ can have only simple zeros at $w=\pm 1$. But if $z_0\in\Tset$ is such that $x_0=\frac{1}{2}\left(z_0+\frac{1}{z_0}\right)$, then \eqref{4.27} shows that the real zeros of $p(x_0,w)$ have even multiplicities. Therefore, if $x=x_0$ is such that \eqref{4.9} holds then $p(x_0,w)\neq 0$ for $w\in \oDset$. Applying Fatou's lemma, it follows that \eqref{4.9} holds for all $x\in (-1,1)$ and therefore 
\begin{subequations}
\begin{align}
&\omega(z,w)\neq 0\qquad\text{ when }(z,w)\in (\Tset\setminus\{\pm 1\})\times \oDset.\label{4.28a}\\
\intertext{Using the same argument but reversing the roles of $z$ and $w$, we see that }
&\omega(z,w)\neq 0\qquad\text{ when }(z,w)\in \oDset\times  (\Tset\setminus\{\pm 1\}).\label{4.28b}
\end{align}
\end{subequations}
The proof follows from equations \eqref{4.28a}-\eqref{4.28b}.
\end{proof}

Note that $\omega_{\epsilon_1,\epsilon_2}(z,w) =2+\epsilon_1 z+\epsilon_2 w$ where $\epsilon_1,\epsilon_2\in\{-1,1\}$ provide examples of polynomials satisfying \eqref{4.25} and vanishing at the four points $(\pm 1,\pm 1)$.

The next theorem describes the characteristic properties of the Bernstein--Szeg\H{o} measures on $\Rset^2$ for which both \thref{th4.1} and \reref{re4.9} apply.

\begin{Theorem}[Bernstein--Szeg\H{o} measures on $\Rset^2$]\label{th4.11}
Suppose that \\ 
$\bullet$ $\omega(z,w)\in\Rset_{n_0,m_0}[z,w]$ is nonzero for $(z,w)\in \Dset^2$, \\
$\bullet$ $q_1(x)\in \Rset_{2n_1}[x]$ is positive for $x\in (-1,1)$, \\
$\bullet$ $q_2(y)\in \Rset_{2m_1}[y]$ is positive for $y\in (-1,1)$,\\
and 
\begin{equation}\label{4.29}
Q(x,y)=q_1(x)q_2(y)\omega(z,w)\omega(1/z,w)\omega(z,1/w)\omega(1/z,1/w)
\end{equation}
is such that 
\begin{equation}\label{4.30}
\iint_{(-1,1)^2}\frac{\sqrt{1-x^2}\sqrt{1-y^2}}{Q(x,y)}\,dx\,dy<\infty.
\end{equation}
Then the recurrence coefficients of the linear functional $\cL:\Rset[x,y]\to\Rset$ defined by
\begin{equation}\label{4.31}
\cL(f)=\frac{4}{\pi^2}\iint_{(-1,1)^2}\frac{f(x,y)\sqrt{1-x^2}\sqrt{1-y^2}}{Q(x,y)}\,dx\,dy, 
\end{equation}
satisfy 
\begin{subequations}\label{4.32}
\begin{align}
&A_{k+1,l}=\frac{1}{2}I_{l+1},  && B_{k,l}=0, && \text{ for all }\quad k \geq n,\quad  l\geq m, \label{4.32a} \\
&\tilde{A}_{k,l+1}=\frac{1}{2}I_{k+1},  && \tilde{B}_{k,l}=0, && \text{ for all }\quad k \geq n,\quad  l\geq m, \label{4.32b}
\end{align}
\end{subequations}
where $n=n_0+n_1$, $m=m_0+m_1$. Moreover, if $\qt_1(z)\in\Rset_{2n_1}[z]$, $\qt_2(w)\in  \Rset_{2m_1}[w]$ denote the stable Fej\'er--Riesz factors of $q_1(x)$ and $q_2(y)$, respectively,  and if we set 
\begin{subequations}\label{4.33}
\begin{align}
&p(x,w)=\qt_2(w)\omega(z,w)\omega(1/z,w),\label{4.33a}\\
&\ph(z,y)=\qt_1(z)\omega(z,w)\omega(z,1/w),\label{4.33b}
\end{align}
\end{subequations}
then 
\begin{equation}\label{4.34}
\cF_{n,m}(\cL;\cdot)=\cK_m(q_1,p;\cdot) \qquad\text{and}\qquad \cFt_{n,m}(\cL;\cdot)=\cKt_n(q_2,\ph;\cdot).
\end{equation}
\end{Theorem}

\begin{proof}
From \leref{le4.10} we know that $\omega(z,w)$ is nonzero on $\oDset^2$, except possibly at the four points $(\pm 1,\pm 1)$. Combining this with \eqref{4.33a} we see that $p(x,w)\neq 0$ for $(x,w)\in(-1,1)\times  (\oDset\setminus  \{-1,1\})$. 
Note also that $p(x,w)\in \Rset_{n_0,2m}[x,w]$ and the denominator on the right-hand side of \eqref{4.31} can be rewritten as 
$$q_1(x) p(x,w)p(x,1/w).$$
Thus \eqref{4.32a} and the first identity in \eqref{4.34} follow from \thref{th4.1}. The proof of \eqref{4.32b} and the second identity in \eqref{4.34} can be obtained by applying the same arguments, exchanging the roles of $z$ and $w$.
\end{proof}

The next theorem goes in the opposite direction and characterizes the Bernstein--Szeg\H{o} measures in \eqref{4.29}-\eqref{4.31}. 

\begin{Theorem}\label{th4.12}
Let $\cL: \Rset_{2n,2m}[x,y]\to\Rset$ be a positive linear functional such that $\hat{\Psi}^m_{n}(z)=z^{n}(P^m_{n}(x)-2zA_{n,m}^tP^m_{n-1}(x))$ is invertible for $z\in(-1,1)$, while $\hat{\tilde{\Psi}}^n_{m}(w)=w^{m}(P^n_{m}(y)-2w\At_{n,m}^t\Pt^n_{m-1}(y))$ is invertible for $w\in(-1,1)$. Suppose that for some $n_1\leq n$ and $m_1\leq m$ there exist polynomials
\begin{enumerate}[\rm(a)]
\item $p(x,w)\in \Rset_{n-n_1,2m}[x,w]$ which is nonzero for $(x,w)\in(-1,1)^2$, 
\item $\ph(z,y)\in \Rset_{2n,m-m_1}[z,y]$ which is nonzero for $(z,y)\in(-1,1)^2$, and
\item $q_1(x)\in\Rset_{2n_1}[x]$, $q_2(y)\in\Rset_{2m_1}[y]$,
\end{enumerate}
such that \eqref{4.34} holds. 
Then $q_1(x)>0$, $q_2(y)>0$ for $x,y\in (-1,1)$, and there exists $\omega(z,w)\in \Rset_{n-n_1,m-m_1}[z,w]$ which is nonzero for $(z,w)\in\Dset^2$ and such that if we set 
\begin{equation}\label{4.35}
Q(x,y)=q_1(x)q_2(y)\omega(z,w)\omega(1/z,w)\omega(z,1/w)\omega(1/z,1/w),
\end{equation}
then \eqref{4.31} holds for all $f\in\Rset_{2n,2m}[x,y]$.
\end{Theorem}

Besides \thref{th4.1}(II) and \reref{re4.9}, for the proof of the above theorem we need a few lemmas. The first one provides a natural extension of a well-known criterion \cite[Theorem 1]{Strintzis} for stability of bivariate polynomials on $\oDset^2$, by replacing $\oDset^2$ with $\Dset^2$ and by allowing vanishing at some points on the boundary of $\Dset^2$.

\begin{Lemma}\label{le4.13}
Let $\omega(z,w)\in\Cset[z,w]$. \\
{\rm (a)} If $z_0\in \oDset$ and if $\hat{T}_2$ is a dense subset of $\Tset$ such that
\begin{subequations}\label{4.36}
\begin{align}
&\omega(z_0,w)\neq 0, && \text {for all }w\in\oDset, \label{4.36a}\\
&\omega(z,w)\neq 0, &&\text {for }(z,w)\in \Dset \times \hat{T}_2  ,\label{4.36b}
\end{align}
\end{subequations}
then 
\begin{equation}\label{4.37}
\omega(z,w)\neq 0\qquad  \text {for }(z,w)\in \Dset^2. 
\end{equation}
{\rm (b)}  In particular, if $T_1$ and $T_2$ are countable subsets of $\Tset$ and
\begin{subequations}\label{4.38}
\begin{align}
&\omega(z,w)\neq 0, && \text {for }(z,w)\in (\Tset\setminus T_1)\times (\oDset\setminus T_2), \label{4.38a}\\
&\omega(z,w)\neq 0, &&\text {for }(z,w)\in \Dset \times (\Tset\setminus T_2),\label{4.38b}
\end{align}
\end{subequations}
then \eqref{4.37} holds.
\end{Lemma}

\begin{proof}
The statements in both parts are trivial if $\omega(z,w)$ depends only on one of the variables. Moreover, if $\omega(z,w)$ factors in $\Cset[z,w]$ then equations \eqref{4.36} in (a) or \eqref{4.38} in (b) will hold for each of the factors and \eqref{4.37} will be true for $\omega(z,w)$ if and only if it is true for each of the factors. Thus, without any restriction, we can assume in the proof that $\omega(z,w)$ has no factors depending only on $z$ or $w$.

We start by proving (a). It is easy to see that $\omega(z,w)\neq 0$ for all $(z,w)\in \Dset\times \Tset$. By \eqref{4.36b}, we just need to show that this is true if $w\notin \hat{T}_2$. We fix $w_0\in\Tset\setminus \hat{T}_2 $ and we consider a sequence $\hat{w}_k$ of points in $\hat{T}_2$ such that $\hat{w}_k\to w_0$ as $k\to\infty$. Since $\omega(z,\hat{w}_k)\to \omega(z,w_0)$ uniformly on $\Dset$, the proof follows from Hurwitz's theorem.

Let $\{K_n\}_{n=1}^{\infty}$ be a sequence of connected, compact subsets of $\Dset\cup\{z_0\}$ such that $z_0\in K_n$ and 
\begin{equation}\label{4.39}
\cup_{n=1}^{\infty}K_n= \Dset\cup\{z_0\}. 
\end{equation}
For instance, we can take $K_n$ to be the convex hull of the closed disk $\oDset_{n/(n+1)}$ centered at the origin with radius $n/(n+1)$ and $\{z_0\}$. Then for every fixed $z\in K_n$, $\omega(z,w)\neq 0$ when $w\in \Tset$ and the number of zeros of $\omega(z,w)$ for $w\in\Dset$ is 
$$\cZ(z)=\frac{1}{2\pi i}\oint_{\Tset}\frac{\partial_w \omega(z,w)}{\omega(z,w)}\,dw.$$
By the dominated converge theorem, $\cZ:K_n\to \Nset_0$ is continuous and since $K_n$ is connected, $\cZ$ must be a constant. Since $\cZ(z_0)=0$, we see that $\cZ(z)=0$ for all $z\in K_n$, i.e.
$\omega(z,w)\neq 0$ when $(z,w)\in K_n\times \Dset$. Equation \eqref{4.37} now follows from \eqref{4.39}, completing the proof of (a)

If we set $\hat{T}_2=\Tset\setminus T_2$, then clearly \eqref{4.38b} implies \eqref{4.36b}, hence (b) will follow from (a) if we can show that equation \eqref{4.38a} implies \eqref{4.36a}. For each $w_0\in T_2$ the polynomial $\omega(z,w_0)$ is not identically $0$ and has finitely many roots in $\Cset$. Let $z_0\in \Tset\setminus T_1$ be such that $\omega(z_0,w_0)\neq 0$ for all $w_0\in T_2$. Then by construction and using \eqref{4.38a} we see that  \eqref{4.36a} holds.  
\end{proof}

\begin{Lemma}\label{le4.14}
Let $\xi(z,w)\in\Rset[z,w]$ be an irreducible polynomial which depends on both $z$ and $w$, such that
\begin{equation}\label{4.40}
\xi(z,w)=cz^k\xi(1/z,w),
\end{equation}
for some $k\in\Nset$ and $c\in\Rset$. Then $c=1$, $k$ is even and there exists $(z_0,w_0)\in\Dset^2$ such that $\xi(z_0,w_0)=0$.
\end{Lemma}

\begin{proof}
Iterating \eqref{4.40} we see that $\xi(z,w)=cz^k\xi(1/z,w)=c^2\xi(z,w)$ which shows that $c=\pm1$. If we assume that $c=-1$ then setting $z=1$ in \eqref{4.40} gives $\xi(1,w)=0$ which is impossible because this would imply that $\xi(z,w)$ is divisible by $z-1$. Thus $c=1$ and it is easy to see that $k$ must be even since otherwise \eqref{4.40} with $z=-1$ would imply that $\xi(z,w)$ is divisible by $z+1$.

Equation \eqref{4.40} and the irreducibility of $\xi(z,w)$ show that  $\xi(z,w)$ has the expansion 
$$\xi(z,w)=\sum_{j=0}^{k}\xi_j(w)z^j,$$
where $\xi_j(w)=\xi_{k-j}(w)$ for $j=0,\dots,k$ and the highest coefficient $\xi_k(w)$ is not identically equal to $0$. 
If we take any $w_1\in\Dset$ such that $\xi_k(w_1)\neq 0$, the product of the roots of $\xi(z,w_1)=0$ will be $1$ and therefore, if none of them is in $\Dset$, they all must lie on $\Tset$. The proof will follow if we can show that this cannot be true for all such $w_1\in\Dset$. Fix $w_1\in\Dset$ which is neither a zero of the highest coefficient $\xi_k(w)$ nor a zero of the discriminant of $\xi(z,w)$ viewed as a polynomial of $z$. If $z_1$ is any root of the equation $\xi(z,w_1)=0$, there exists an open disk $\Dset_{\epsilon}(w_1)\subset\Dset$ containing $w_1$ and a nonconstant holomorphic function $f:\Dset_{\epsilon}(w_1)\to\Cset$ such that $f(w_1)=z_1$ and $\xi(f(w),w)=0$. Since the set $f(\Dset_{\epsilon}(w_1))$  
is open, it cannot be a subset of $\Tset$.
\end{proof}

\begin{Lemma}\label{le4.15}
Let $Q(x,y)\in \Rset[x,y]$ be a polynomial which is positive for $(x,y)\in(-1,1)^2$. If there exist polynomials $q_1(x)$, $q_2(y)$, $p(x,w)$ and $\ph(z,y)$ with real coefficients such that 
\begin{enumerate}[\rm(i)]
\item $p(x,w)\neq 0$ for $(x,w)\in(-1,1)\times  (\oDset\setminus  \{-1,1\})$, 
\item $\ph(z,y)\neq 0$ for $(z,y)\in (\oDset\setminus  \{-1,1\}) \times (-1,1)$,
\end{enumerate}
and
\begin{equation}\label{4.41}
Q(x,y)= q_1(x)p(x,w)p(x,1/w)=q_2(y)\ph(z,y)\ph(1/z,y),
\end{equation}
then there exists a polynomial $\omega(z,w)\in\Rset[z,w]$ which is nonzero for $(z,w)\in \Dset^2$ and equation \eqref{4.29} holds.
\end{Lemma}

\begin{proof}
Equation~\eqref{4.41} shows that $q_1(x)>0$ for $x\in(-1,1)$ and $q_1(x)$ divides $Q(x,y)$. Moreover, if $\qt_1(z)\in\Rset[z]$ is the stable Fej\'er--Riesz factor of $q_1(x)$, then it is easy to see that $\qt_1(z)$ divides $\ph(z,y)$. Thus, we can cancel $q_1(x)$ in \eqref{4.41} and the remaining polynomials will satisfy the conditions in the lemma. This means that we can assume that $q_1(x)=1$ and a similar argument shows that we can take $q_2(y)=1$. In view of this, we assume below that $q_1(x)=q_2(y)=1$ and therefore \eqref{4.41} reduces to 
\begin{equation}\label{4.42}
Q(x,y)= p(x,w)p(x,1/w)=\ph(z,y)\ph(1/z,y).
\end{equation}
We can further assume that $p(x,w)$ and $\ph(z,y)$ have no factors depending on just one of the variables. Indeed, $p(x,w)$ has a factor $q_0(x)$ if and only if $q_0^2(x)$ divides $Q(x,y)$ and $\ph(z,y)$ has a factor $\hat{q}_0(z)$ such that $q_0^2(x)=\hat{q}_0(z)\hat{q}_0(1/z)$. This means that if we know how to construct $\omega(z,w)$ satisfying \eqref{4.35} with $Q$ replaced by $Q(x,y)/q_0^2(x)$, we can multiply it by $\hat{q}_0(z)$ to obtain $\omega(z,w)$ for $Q(x,y)$. A similar argument can be applied for factors depending only on $y$ and $w$. We can also assume that $p$ and $\ph$ are normalized so that $p(0,0)>0$ and $\ph(0,0)>0$.

Next, we replace $x$ by $\frac{1}{2}\left(z+\frac{1}{z}\right)$ and we factor $p(x,w)$ into a product of irreducible factors in the unique factorization domain $\Rset[z^{\pm 1},w]$. With the normalization above, we know that in this factorization, $p(x,w)$ cannot have irreducible factors depending only on $z$ or $w$. 
We show next that this factorization cannot contain an irreducible factor $\xi(z,w)$ which is associate to $\xi(1/z,w)$ in $\Rset[z^{\pm 1},w]$. Indeed, assume that $p(x,w)$ has such irreducible factor, i.e. \eqref{4.40} holds and by multiplying by a unit element of the form $az^j$ where $a\in\Rset\setminus\{0\}$ and $j\in\Zset$ we can assume that  $\xi(z,w)$ is an irreducible element in $\Rset[z,w]$.
Condition (i) shows that 
\begin{subequations}
\begin{equation}\label{4.43a}
\xi(z,w)\neq 0\quad\text{ when }(z,w)\in(\Tset\setminus\{-1,1\})\times  (\oDset\setminus  \{-1,1\}). 
\end{equation}
We replace $y$ by $\frac{1}{2}\left(w+\frac{1}{w}\right)$ and we consider \eqref{4.42} in the ring $\Rset[z^{\pm 1},w^{\pm 1}]$ of Laurent polynomials of $z$ and $w$ which is also a unique factorization domain. This shows that $\xi(z,w)$ must also be an irreducible factor in both $\ph(z,y)$ and $\ph(1/z,y)$ in the ring $\Rset[z^{\pm 1},w^{\pm 1}]$. Therefore, condition (ii) tells us that
\begin{equation}\label{4.43b}
\xi(z,w)\neq 0\quad\text{ when }(z,w)\in (\oDset\setminus  \{-1,1\})\times (\Tset\setminus\{-1,1\}). 
\end{equation}
\end{subequations}
Using equations \eqref{4.43a}-\eqref{4.43b} and \leref{le4.13}(b) with $T_1=T_2=\{-1,1\}$ it follows that $\xi(z,w)\neq 0$ for $(z,w)\in\Dset^2$ which contradicts \leref{le4.14}.

Since $p(x,w)$ is invariant under the involution $z\to 1/z$ and has no irreducible factor $\xi(z,w)$ which is associated to $\xi(1/z,w)$, we can factor it as follows  
\begin{equation}\label{4.44}
p(x,w)=\prod_{j=1}^{l}\omega_j(z,w)\omega_j(1/z,w), 
\text{ where }\omega_j(z,w)\in\Rset[z,w] \text{ are irreducible.}
\end{equation}
From condition (i) it follows that for all $j=1,\dots,l$ we have
\begin{equation}\label{4.45}
\omega_j(z,w)\neq 0\quad \text{ for } (z,w)\in(\Tset\setminus\{-1,1\})\times  (\oDset\setminus  \{-1,1\}).
\end{equation}
Note that we still have some freedom in choosing the polynomials $\omega_j(z,w)$: if $n_j\in\Nset$ is the minimal positive integer such that $z^{n_j}\omega_j(1/z,w)\in\Rset[z,w]$ (i.e. $n_j$ is the degree of $\omega_j(z,w)$ viewed as a polynomial in $z$ for generic $w$), then we can replace $\omega_j(z,w)$ by  $z^{n_j}\omega_j(1/z,w)$ and both the representation in \eqref{4.44} and the nonvanishing condition in \eqref{4.45} will still hold. For each $j$, we will pick one of these two terms below and include it in $\omega(z,w)$.

Similarly to the constructions above, we replace $y$ by $\frac{1}{2}\left(w+\frac{1}{w}\right)$ and we can factor $\ph(z,y)$ as follows
\begin{equation}\label{4.46}
\ph(z,y)=\prod_{k=1}^{l}\hat{\omega}_k(z,w) \hat{\omega}_k(z,1/w), \text{ where }\hat{\omega}_k(z,w)\in\Rset[z,w] \text{ are irreducible.}
\end{equation}
Using condition (ii) we see that for $k=1,\dots,l$ we have
\begin{equation}\label{4.47}
\hat{\omega}_k(z,w)\neq 0\quad \text{ for } (z,w)\in (\oDset\setminus  \{-1,1\})\times (\Tset\setminus\{-1,1\}).
\end{equation}
If $\hat{m}_k\in\Nset$ is the minimal positive integer such that $w^{\hat{m}_k} \hat{\omega}_k(z,1/w)\in\Rset[z,w]$, then we can replace $\hat{\omega}_k(z,w)$ by $w^{\hat{m}_k} \hat{\omega}_k(z,1/w)$ and both the representation in \eqref{4.46} and the nonvanishing condition in \eqref{4.47} will still hold.

We substitute \eqref{4.44} and \eqref{4.46} into \eqref{4.42} and consider this equation in the unique factorization domain $\Rset[z^{\pm 1},w^{\pm 1}]$. It follows that for every $j$, there exists a unique $k$ such that exactly one of the following holds:
\begin{enumerate}[\rm(a)]
\item $\omega_j(z,w)$ is associate to $\hat{\omega}_k(z,w)$ or $\hat{\omega}_k(z,1/w)$ in $\Rset[z^{\pm 1},w^{\pm 1}]$;
\item $\omega_j(z,w)$ is associate to $\hat{\omega}_k(1/z,w)$ or $\hat{\omega}_k(1/z,1/w)$ in $\Rset[z^{\pm 1},w^{\pm 1}]$.
\end{enumerate}
If (b) holds for some $j$, then we will use the freedom to replace  $\omega_j(z,w)$ by  $z^{n_j}\omega_j(1/z,w)$, and thus we can assume that (a) holds for all $j=1,\dots,l$. Next, if $\omega_j(z,w)$ is associate to $\hat{\omega}_k(z,1/w)$ then we replace $\hat{\omega}_k(z,w)$ by $w^{\hat{m}_k} \hat{\omega}_k(z,1/w)$. With this normalization, for every $j\in\{1,\dots,l\}$, there exists a unique $k\in\{1,\dots,l\}$ such that $\omega_j(z,w)$ is associate to $\hat{\omega}_k(z,w)$ in  $\Rset[z^{\pm 1},w^{\pm 1}]$. Note that the units in $\Rset[z^{\pm1},w^{\pm1}]$ are of the form $cz^jw^s$, where $c\in\Rset\setminus\{0\}$, $j,r\in\Zset$ and the polynomials $\omega_j(z,w)$, $\hat{\omega}_k(z,w)$ are not divisible by $z$ or $w$, so we must have 
$$\omega_j(z,w)=c_k\hat{\omega}_k(z,w),\quad\text{ where }c_k \text{ is a nonzero constant}.$$ 
Combining the last equation with equations \eqref{4.45} and \eqref{4.47}, and applying \leref{le4.13}(b) with $T_1=T_2=\{-1,1\}$, we see that $\omega_j(z,w)\neq 0$ for $(z,w)\in\Dset^2$. Since this is true for every $j$, the proof follows by setting $\omega(z,w)=\prod_{j=1}^{l}\omega_j(z,w)$.
\end{proof}
 
\begin{proof}[Proof of \thref{th4.12}]
\thref{th4.1}(II) and \reref{re4.4} tell us that $q_1(x)>0$ for $x\in(-1,1)$,  
$$p(x,w)\neq 0\quad\text{ for }(x,w)\in(-1,1)\times  (\oDset\setminus  \{-1,1\})$$ 
and if we set 
$$Q(x,y)=q_1(x)p(x,w)p(x,1/w),$$
then \eqref{4.31} holds for all $f\in\Rset_{2n,2m}[x,y]$. Similarly, reversing the roles of $x$ and $y$ using \reref{re4.9} we see that 
$q_2(y)>0$ for $y\in(-1,1)$,  
$$\ph(z,y)\neq 0\quad\text{ for }(z,y)\in (\oDset\setminus  \{-1,1\})\times (-1,1)$$ 
and if we set 
$$Q_1(x,y)=q_2(y)\ph(z,y)\ph(1/z,y),$$
then \eqref{4.31} holds for all $f\in\Rset_{2n,2m}[x,y]$ with $Q$ replaced by $Q_1$. By \leref{le4.6}, $Q$ and $Q_1$ coincide, i.e.
$$Q(x,y)=q_1(x)p(x,w)p(x,1/w)=q_2(y)\ph(z,y)\ph(1/z,y).$$
The proof now follows from \leref{le4.15}.
\end{proof}

\section{Orthogonal decompositions and extra orthogonality for Bernstein--Szeg\H{o} measures on the torus}\label{se5}
We denote by $\Cset[z^{\pm 1},w^{\pm 1}]=\Cset[z,z^{-1},w,w^{-1}]$ the space of Laurent polynomials with complex coefficients in $z$ and $w$. For every polynomial $g(z,w)$ we set
$$\gb(z,w)=\overline{g(\zb,\wb)},$$
and we extend this notation to subspaces of $\Cset[z^{\pm 1},w^{\pm 1}]$, i.e. if $V[z,w]$ is a subspace of $\Cset[z^{\pm 1},w^{\pm 1}]$, then $\overline{V}[z,w]$ is the subspace 
$$\overline{V}[z,w]=\{\gb(z,w): g(z,w)\in V[z,w] \}.$$ 

Recall that $\Cset_{\nt,\mt}[z,w]$ denotes the space of polynomials with complex coefficients in $z$ and $w$ of degrees at most $\nt$ in $z$ and $\mt$ in $w$. Throughout this section, we fix 
\begin{subequations}
\begin{equation}\label{5.1a}
p_c(z,w)\in \Cset_{\nt,\mt}[z,w]\qquad \text{ such that }\quad p_{c}(z,w)\neq 0\text{ for }(z,w)\in\Tset\times \oDset,
\end{equation}
and we define a positive linear functional on $\Cset[z^{\pm 1},w^{\pm 1}]$ by
\begin{equation}\label{5.1b}
\cL_{p_c}(z^kw^l)=\frac{1}{(2\pi)^2}\iint_{\Tset^2}z^kw^l\, \frac{|dz|\,|dw|}{|p_c(z,w)|^2}.
\end{equation}
\end{subequations}
The functional $\cL_{p_c}$ induces an inner product on $\Cset[z,w]$ by 
$$\langle f,g\rangle_{\cL_{p_c}} =\cL_{p_c}(f(z,w)\gb(1/z,1/w)),\quad\text{where } f,g\in \Cset[z,w],$$
and we consider the spaces 
\begin{subequations}\label{5.2}
\begin{align}
&\sP_{k,l;\cL_{p_c}}[z,w]=\Cset_{k,l}[z,w]\ominus \Cset_{k-1,l}[z,w], \label{5.2a} \\
&\sPt_{k,l;\cL_{p_c}}[z,w]=\Cset_{k,l}[z,w]\ominus \Cset_{k,l-1}[z,w], \label{5.2b}
\end{align}
\end{subequations}
of orthogonal polynomials with respect to this inner product. 
One of the interesting facts discovered in \cite{GeI2} is that for $k$ and $l$ sufficiently large, the stability condition \eqref{5.1a} allows us to decompose the space $\sPt_{k,l;\cL_{p_c}}[z,w]$ into the direct sum of two orthogonal subspaces $\sPt_{k,l}^{1}$ and $\sPt_{k,l}^{2}$ which have a simple spectral characterization called the {\em split-shift orthogonality condition} in \cite{GeIK}, namely $\sPt_{k,l}^{1}\perp z\sPt_{k,l}^{2}$ and $\sPt_{k,l}^{1}, z\sPt_{k,l}^{2} \subset \sPt_{k+1,l;\cL_{p_c}}[z,w]$. Moreover, these spaces possess a lot of ``extra orthogonality properties'' which will play a crucial role in the construction of the Szeg\H{o} mapping later, replacing the Cauchy--Goursat  theorem in the one-dimensional setting in the proof of \prref{pr2.1}. More precisely, with the above notations, we can summarize some of the results in \cite{GeI2,GeIK} as follows.
\begin{Theorem}\label{th5.1}
For $M\geq \mt$, there exist unique sub-spaces $\sPt_{\nt-1,M;\cL_{p_c}}^{1},\sPt_{\nt-1,M;\cL_{p_c}}^{2}$ of $\sPt_{\nt-1,M;\cL_{p_c}}[z,w]$ such that
\begin{subequations}\label{5.3}
\begin{align}
&\sPt_{\nt-1,M;\cL_{p_c}}[z,w] =\sPt_{\nt-1,M;\cL_{p_c}}^{1}\oplus \sPt_{\nt-1,M;\cL_{p_c}}^{2} , \label{5.3a} \\
& \sPt_{\nt,M;\cL_{p_c}}[z,w] =\sPt_{\nt-1,M;\cL_{p_c}}^{1}\oplus z\sPt_{\nt-1,M;\cL_{p_c}}^{2} \oplus \Span_{\Cset}\{ w^{M-\mt}\prs{c}(z,w)\} , \label{5.3b} 
\end{align}
\end{subequations}
where $\prs{c}(z,w)=z^{\nt}w^{\mt}\psb{c}(1/z,1/w)$. Moreover, these spaces satisfy the following orthogonality conditions
\begin{subequations}\label{5.4}
\begin{align}
&\sPt_{\nt-1,M;\cL_{p_c}}^{1}\perp z^{k}w^l, &&\text{ for all }\quad  0\leq k, \quad 0\leq l\leq M-1, \label{5.4a} \\
&\sPt_{\nt-1,M;\cL_{p_c}}^{1}\perp z^{k}w^{M-\mt}\prs{c}(z,w), &&\text{ for all }\quad  0\leq k,  \label{5.4b}\\
&\sPt_{\nt-1,M;\cL_{p_c}}^{2}\perp z^{k}w^l, &&\text{ for all }\quad   k\leq \nt-1, \quad 0\leq l\leq M-1, \label{5.4c} \\
&z^{k}\sPt_{\nt-1,M;\cL_{p_c}}^{2}\perp w^{M-\mt}\prs{c}(z,w), &&\text{ for all }\quad  1\leq k,  \label{5.4d}\\
&\sPt_{\nt-1,M;\cL_{p_c}}^{1}\perp z^k\sPt_{\nt-1,M;\cL_{p_c}}^{2}, &&\text{ for all }\quad  0\leq k. \label{5.4e}
\end{align}
\end{subequations}
\end{Theorem}

Note that the orthogonality conditions \eqref{5.4} tell us that for all $N\geq \nt$ and $M\geq \mt$ we have
\begin{equation*}
\sPt_{N,M;\cL_{p_c}}[z,w] =\sPt_{\nt-1,M;\cL_{p_c}}^{1}\oplus z^{N-\nt+1}\sPt_{\nt-1,M;\cL_{p_c}}^{2} \bigoplus_{k=0}^{N-\nt} \Span_{\Cset}\{ z^k w^{M-\mt}\prs{c}(z,w)\},
\end{equation*}
i.e. orthonormal bases of the spaces  $\sPt_{\nt-1,M;\cL_{p_c}}^{1} $ and $\sPt_{\nt-1,M;\cL_{p_c}}^{2}$ in \eqref{5.3a} can be used to construct a simple orthonormal basis of  $\sPt_{N,M;\cL_{p_c}}[z,w] $ for every $N\geq \nt$.
In the lemma below we prove another orthogonality relation which will be useful later.
\begin{Lemma}\label{le5.2}
For $k\geq 0$ we have
\begin{equation*}
z^{k+\nt-1}w^{M-1}\overline{\sPt_{\nt-1,M;\cL_{p_c}}^{1}}[1/z,1/w] \perp  \sPt_{\nt-1,M;\cL_{p_c}}^{1}[z,w].
\end{equation*}
\end{Lemma}

\begin{proof}
In the proof, it will be convenient to use the fact that the inner product $\langle  ,\rangle_{\cL_{p_c}} $ extends to the Hilbert space $L^2(\Tset^2, \mu_c)$ where $d\mu_c=\frac{|dz|\,|dw|}{(2\pi)^2|p_c(z,w)|^2}$. Suppose that
$$p_c(z,0)=p_s(z)p_u(z),$$
where 
\begin{itemize}
\item $p_s(z)$ is nonzero for $|z|\leq 1$;
\item $p_u(z)$ is nonzero for $|z|\geq 1$, 
\end{itemize}
and let $\nt_0=\deg_z p_u(z)$. If $f\in \sPt_{\nt-1,M;\cL_{p_c}}^{1}$, then by \cite[Theorem~2.3]{GeIK} we have
\begin{align}
f(z,w)=&h(z) z^{\nt_0}\psb{u}(1/z)w^M + g(z,w),\label{5.5}\\
 \text{where }\quad &h(z)\in\Cset_{\nt-\nt_0-1}[z]\quad \text{ and }\quad g(z,w)\in \Cset_{\nt-1,M-1}[z,w]. \nonumber
\end{align}
Therefore, 
\begin{equation*}
z^{k+\nt-1}w^{M-1}\overline{f}(1/z,1/w)=z^{k+\nt-\nt_0-1} \hb(1/z) p_u(z) \frac{1}{w} + z^{k+\nt-1}w^{M-1}\gb(1/z,1/w).
\end{equation*}
Note that $z^{k+\nt-1}w^{M-1}\gb(1/z,1/w) \perp \sPt_{\nt-1,M;\cL_{p_c}}^{1}$ by \eqref{5.4a}, and therefore it is enough to show that $w^{-1} p_u(z) \Cset[z]\perp \sPt_{\nt-1,M;\cL_{p_c}}^{1} $. Moreover, 
$$\frac{p_u(z)}{w}=\frac{p(z,0)}{wp_s(z)}=\frac{p(z,w)}{wp_s(z)}- \underbrace{\frac{p(z,w)-p(z,0)}{wp_s(z)}}_{r(z,w)},$$
where $r(z,w)$ is a polynomial in $w$ of degree at most $\mt-1$, whose coefficients are holomorphic on the closed disk $|z|\leq 1$, and therefore $r(z,w)\Cset[z]\perp  \sPt_{\nt-1,M;\cL_{p_c}}^{1}$ by \eqref{5.4a}. Finally, it is easy to $\frac{p(z,w)}{wp_s(z)} \Cset[z]\perp \sPt_{\nt-1,M;\cL_{p_c}}^{1}$ by using the definition of the inner product and by computing the $w$-integral.
\end{proof}

We show below that the spaces $\sPt_{\nt-1,M;\cL_{p_c}}^{1}$ and  $\sPt_{\nt-1,M;\cL_{p_c}}^{2}$ are simply related to each other if $p_c(z,w)$ is invariant under a reflection in $z$. This fact will play a key role in the definition of the bivariate Szeg\H{o} mapping in the plane.

\begin{Proposition}\label{pr5.3}
If $\nt$ is even and 
\begin{equation}\label{5.6}
z^{\nt}p_c(1/z,w)=p_c(z,w),
\end{equation}
then 
\begin{equation}\label{5.7}
\sPt_{\nt-1,M;\cL_{p_c}}^{2}[z,w]=z^{\nt-1}\sPt_{\nt-1,M;\cL_{p_c}}^{1}[1/z,w].
\end{equation}
\end{Proposition}

\begin{proof}
For $k\in\Nset_0$, consider the involution $\cR_z^{k}$ on $\Cset[z^{\pm 1},w^{\pm 1}]$ defined by 
$$\cR_z^{k}(g(z,w))=z^{k}g(1/z,w).$$
Equation \eqref{5.6} shows that $\cR_z^{k}$ preserves the inner product $\langle , \rangle_{p_c}$, and since 
$$\cR_z^{k}(\Cset_{k,\ell}[z,w])=\Cset_{k,\ell}[z,w], \text{ for all }k,\ell \in\Nset_0$$
we see that  $\sPt_{k,\ell;\cL_{p_c}}[z,w]$ is an invariant subspace of $\cR_z^{k}$. In particular, applying $\cR_z^{\nt-1}$ to \eqref{5.3a} yields 
\begin{subequations}\label{5.8}
\begin{equation} \label{5.8a}
\sPt_{\nt-1,M;\cL_{p_c}} =\cR_z^{\nt-1}(\sPt_{\nt-1,M;\cL_{p_c}}^{1})\oplus \cR_z^{\nt-1}(\sPt_{\nt-1,M;\cL_{p_c}}^{2}).
\end{equation}
Equation \eqref{5.6} implies that $\cR_z^{\nt}( \prs{c}(z,w))= \prs{c}(z,w)$ and since 
$$\cR_z^{\nt}(\sPt_{\nt-1,M;\cL_{p_c}}^{1}) = z \cR_z^{\nt-1}(\sPt_{\nt-1,M;\cL_{p_c}}^{1}), \quad   \cR_z^{\nt}(z\sPt_{\nt-1,M;\cL_{p_c}}^{2} )=\cR_z^{\nt-1}(\sPt_{\nt-1,M;\cL_{p_c}}^{2} ),$$
applying $\cR_z^{\nt}$ to \eqref{5.3b} shows that
\begin{equation} \label{5.8b}
\sPt_{\nt,M;\cL_{p_c}} =z\cR_z^{\nt-1}(\sPt_{\nt-1,M;\cL_{p_c}}^{1})\oplus  \cR_z^{\nt-1}(\sPt_{\nt-1,M;\cL_{p_c}}^{2} )\oplus \Span_{\Cset}\{ w^{M-\mt}\prs{c}(z,w)\}.
\end{equation}
\end{subequations}
From equations \eqref{5.8a}-\eqref{5.8b} and the uniqueness of the subspaces $\sPt_{\nt-1,M;\cL_{p_c}}^{j}$ satisfying   \eqref{5.3a}-\eqref{5.3b} it follows that \eqref{5.7} holds.
\end{proof}

Note that \eqref{5.6} cannot hold if $\nt$ is odd because $p_c(z,w)$ will vanish when $z=-1$.

\begin{Remark}\label{re5.4}
A subclass of measures on $\Tset^2$ associated with polynomials 
\begin{equation}\label{5.9}
p_c(z,w)  \quad \text{ which are nonzero for }\quad (z,w)\in\oDset^2 
\end{equation}
was introduced and studied in \cite{GeWo}. 
In this case, there are several simplifications: 
\begin{align*}
\sPt_{\nt-1,M;\cL_{p_c}}^{2}=\{0\},\quad \sPt_{\nt-1,M;\cL_{p_c}}[z,w] =\sPt_{\nt-1,M;\cL_{p_c}}^{1}[z,w]=w^{M-\mt}\sPt_{\nt-1,\mt;\cL_{p_c}}[z,w]
\end{align*}
and
\begin{equation} \label{5.10}
\sPt_{\nt-1,\mt;\cL_{p_c}}\perp z^{k}w^l,\qquad\text{ for all }\quad  0\leq k, \quad l \leq \mt-1.
\end{equation}
\end{Remark}
There is an important class of polynomials which do not satisfy the stability condition \eqref{5.9}, but can be built from two such polynomials $p_{1}$ and $p_{2}$  by reversing one of the variables in one of them and then considering their product. These polynomials appear naturally in the classification of the Bernstein--Szeg\H{o} measures on $\Tset^{2}$ in \cite{GeI2}, and in that case, one can construct simple explicit bases of $\sPt_{\nt-1,M;\cL_{p_c}}^{1}$ and $\sPt_{\nt-1,M;\cL_{p_c}}^{2}$ for arbitrary $M$ sufficiently large, by using fixed orthonormal bases of the spaces associated with $p_{1}$ and $p_{2}$. The precise statement can be deduced from \cite[Theorem 2.7]{GeI2} and is given below.

\begin{Theorem}\label{th5.5}
Suppose that $p_1(z,w)\in \Cset_{\nt_1,\mt_1}[z,w]$, $p_2(z,w)\in \Cset_{\nt_2,\mt_2}[z,w]$ are polynomials nonvanishing for $(z,w)\in\oDset^2$, and 
$$p_c(z,w)=p_1(z,w)z^{\nt_2} p_2(1/z,w), \text{ where }\nt=\nt_1+\nt_2 \text{ and }\mt=\mt_1+\mt_2.$$
Then we have
\begin{subequations}\label{5.11}
\begin{align}
&\sPt_{\nt-1,\mt;\cL_{p_c}}^{1}[z,w] =(w^{\mt_2}\psb{2}(z,1/w) )\,\sPt_{\nt_1-1,\mt_1;\cL_{p_1}}[z,w] , \label{5.11a} \\
&\sPt_{\nt-1,\mt;\cL_{p_c}}^{2}[z,w] =(z^{\nt_1}w^{\mt_1}\psb{1}(1/z,1/w))  z^{\nt_2-1}\overline{\sPt_{\nt_2-1,\mt_2;\cL_{p_2}}}[1/z,w]. \label{5.11b} 
\end{align}
\end{subequations}
Moreover, for $M\geq \mt$ we have
\begin{subequations}\label{5.12}
\begin{align}
&\sPt_{\nt-1,M;\cL_{p_c}}^{1}[z,w] = w^{M-\mt}\sPt_{\nt-1,\mt;\cL_{p_c}}^{1}[z,w]  , \label{5.12a} \\
&\sPt_{\nt-1,M;\cL_{p_c}}^{2}[z,w] = w^{M-\mt}\sPt_{\nt-1,\mt;\cL_{p_c}}^{2}[z,w]. \label{5.12b} 
\end{align}
\end{subequations}
\end{Theorem}

Finally, if $p_c(z,w)$ has a factor depending only on $z$ we have the following decomposition.

\begin{Theorem}\label{th5.6}
Suppose that $p_1(z,w)\in \Cset_{\nt_1,\mt}[z,w]$ is nonzero for $(z,w)\in\Tset\times\oDset$, $p_2(z)\in \Cset_{\nt_2}[z]$ is nonzero for $z\in\oDset$, and let
$$p_c(z,w)=p_1(z,w) p_2(z), \text{ where }\nt=\nt_1+\nt_2.$$
If we set $\prs{1}(z,w)=z^{\nt_1}w^{\mt}\psb{1}(1/z,1/w)$ and $\prs{2}(z)=z^{\nt_2}\psb{2}(1/z)$, then for $M\geq \mt$ we have 
\begin{subequations}\label{5.13}
\begin{align}
&\sPt_{\nt-1,M;\cL_{p_c}}^{1}[z,w] =p_2(z) \sPt_{\nt_1-1,M;\cL_{p_1}}^{1}[z,w] \oplus w^{M-\mt} \prs{1}(z,w) \Cset_{\nt_2-1}[z], \label{5.13a} \\
&\sPt_{\nt-1,M;\cL_{p_c}}^{2}[z,w] = \prs{2}(z) \sPt_{\nt_1-1,M;\cL_{p_1}}^{2}[z,w]  . \label{5.13b} 
\end{align}
\end{subequations}
\end{Theorem}

\begin{proof}
We can check that with respect to $\cL_{p_c}$ we have:
\begin{align*}
&p_2(z) \sPt_{\nt_1-1,M;\cL_{p_1}}^{1}[z,w]\perp z^{k}w^l, \quad  \text{ for all } \quad  0\leq k, \quad 0\leq l\leq M-1, \\
& w^{M-\mt} \prs{1}(z,w) \Cset_{\nt_2-1}[z] \perp z^{k}w^l,  \quad  \text{ for all } \quad  0\leq k, \quad 0\leq l\leq M-1,  \\
&p_2(z) \sPt_{\nt_1-1,M;\cL_{p_1}}^{1}[z,w]\perp w^{M-\mt} \prs{1}(z,w) \Cset_{\nt_2-1}[z] ,   \\
& \prs{2}(z) \sPt_{\nt_1-1,M;\cL_{p_1}}^{2}[z,w]  \perp z^{k}w^l, \quad  \text{ for all } \quad  k\leq \nt-1, \quad 0\leq l\leq M-1, \\
&p_2(z) \sPt_{\nt_1-1,M;\cL_{p_1}}^{1}[z,w]\perp z^{k} \prs{2}(z) \sPt_{\nt_1-1,M;\cL_{p_1}}^{2}[z,w],  \quad  \text{ for all } \quad  0\leq k,  \\
&w^{M-\mt} \prs{1}(z,w) \Cset_{\nt_2-1}[z] \perp z^{k} \prs{2}(z) \sPt_{\nt_1-1,M;\cL_{p_1}}^{2}[z,w],  \quad  \text{ for all } \quad  0\leq k, \\
&p_2(z) \sPt_{\nt_1-1,M;\cL_{p_1}}^{1}[z,w]\perp w^{M-\mt} \prs{c}(z,w),  \\
&w^{M-\mt} \prs{1}(z,w) \Cset_{\nt_2-1}[z] \perp w^{M-\mt} \prs{c}(z,w), \\
&z  \prs{2}(z) \sPt_{\nt_1-1,M;\cL_{p_1}}^{2}[z,w]  \perp w^{M-\mt} \prs{c}(z,w), 
\end{align*}
and apply \thref{th5.1}.
\end{proof}

\begin{Remark}\label{re5.7}
If $p_c(z,w)$ has real coefficients then the spaces $\sPt_{\nt-1,M;\cL_{p_c}}^{1},\sPt_{\nt-1,M;\cL_{p_c}}^{2}$ in Theorem~\ref{th5.1} have bases consisting of polynomials with real coefficients, and we can replace the complex field with $\Rset$ in all statements above. 
More precisely, if 
$$p_c(z,w) \in \Rset_{\nt,\mt}[z,w]$$
and if we set 
\begin{align*}
&\sP_{k,l;\cL_{p_c};\Rset}[z,w]= \sP_{k,l;\cL_{p_c}}[z,w] \cap \Rset[z,w], \quad \sPt_{k,l;\cL_{p_c};\Rset}[z,w]= \sPt_{k,l;\cL_{p_c}}[z,w] \cap \Rset[z,w], \\
&\sPt_{\nt-1,M;\cL_{p_c};\Rset}^{j}[z,w] =\sPt_{\nt-1,M;\cL_{p_c}}^{j} [z,w]  \cap \Rset[z,w] , \quad \text{ for }j=1,2,
\end{align*}
then we can replace equations \eqref{5.3} with
\begin{subequations}\label{5.14}
\begin{align}
&\sPt_{\nt-1,M;\cL_{p_c};\Rset} =\sPt_{\nt-1,M;\cL_{p_c};\Rset}^{1}\oplus \sPt_{\nt-1,M;\cL_{p_c};\Rset}^{2} , \label{5.14a} \\
& \sPt_{\nt,M;\cL_{p_c};\Rset} =\sPt_{\nt-1,M;\cL_{p_c};\Rset}^{1}\oplus z\sPt_{\nt-1,M;\cL_{p_c};\Rset}^{2} \oplus \Span_{\Rset}\{ w^{M-\mt}\prs{c}(z,w)\}. \label{5.14b} 
\end{align}
\end{subequations}
\end{Remark}

\section{Szeg\H{o} mapping}\label{se6}
In this section, we explain how the decomposition of the spaces of polynomials on $\Tset^2$ in the previous section and their orthogonality properties can be used to construct bases of orthonormal polynomials for the Bernstein--Szeg\H{o} measures in the plane. The section is divided in two parts. In the first one, we consider in detail the functional $\cL$ in \thref{th4.1} and the construction of the orthonormal complement on the tilde side of the space spanned by product polynomials obtained in analogy with the one-dimensional theory, cf. \reref{re4.7}. The spaces associated with the  Bernstein--Szeg\H{o} measures in \thref{th4.11} are analyzed in the second subsection. In this case, the constructions simplify significantly and provide explicit orthonormal bases on both sides in terms of fixed orthonormal polynomials on the bi-circle contained in the finite-dimensional space $\Rset_{n_0,m_0}[z,w]$. 

\subsection{Szeg\H{o} mapping for the one-sided factorization case}
In this subsection, we consider the functional $\cL$ in \eqref{4.3} where $q(x)\in\Rset_{2n_1}[x]$, $p(x,w)\in \Rset_{n_0,2m}[x,w]$ are fixed polynomials such that $q(x)>0$ and $p(x,w)\neq 0$ when $(x,w)\in[-1,1]\times \oDset$. 
Let $\qt(z)$ be the stable Fej\'er--Riesz factor of $q(x)$ and let 
\begin{equation}
p_c(z,w)=\qt(z)\pt(z,w)\in\Rset_{2n,2m}[z,w],  \text{ where }n=n_1+n_0 \text{ and } \pt(z,w)=z^{n_0}p(x,w).
\end{equation}
With these notations, the denominator of the weight in \eqref{4.3} can be written simply as $p_c(z,w)p_c(1/z,1/w)$. Note that $p_c(z,w)$ is nonzero for $(z,w)\in\Tset\times \oDset$, and since it has real coefficients, we can use \reref{re5.7} and work with the spaces introduced in previous section over $\Rset$ where $\nt=2n$, $\mt=2m$, but we will omit the explicit $\Rset$ dependence in order to simplify the notation. Using  \thref{th5.6} we see that for $M\geq m$ we have 
\begin{subequations}\label{6.2}
\begin{align}
&\sPt_{2n-1,2M+1;\cL_{p_c}}^{1}[z,w] =\qt(z) \sPt_{2n_0-1,2M+1;\cL_{\pt}}^{1}[z,w] \oplus w^{2M+1} \pt(z,1/w) \Rset_{2n_1-1}[z], \label{6.2a} \\
&\sPt_{2n-1,2M+1;\cL_{p_c}}^{2}[z,w] = z^{2n_1}\qt(1/z) \sPt_{2n_0-1,2M+1;\cL_{\pt}}^{2}[z,w]. \label{6.2b} 
\end{align}
\end{subequations}
Since $z^{2n_0}\pt(1/z,w)=\pt(z,w)$, \prref{pr5.3} tells us that the spaces $\sPt_{2n_0-1,2M+1;\cL_{\pt}}^{1}[z,w] $ and $\sPt_{2n_0-1,2M+1;\cL_{\pt}}^{2}[z,w]$ in \eqref{6.2a} and \eqref{6.2b}, respectively, are  related by an appropriate reflection in $z$. In particular, this implies that 
$$\dim(\sPt_{2n_0-1,2M+1;\cL_{\pt}}^{1}[z,w])=\dim(\sPt_{2n_0-1,2M+1;\cL_{\pt}}^{2}[z,w])=n_0.$$

Let $S_{z,N}:\Rset[z]\to \Rset[x]$ denote the (linear) Szeg\H{o} mapping 
\begin{equation*}
S_{z,N}(f(z))=\frac{z^{N+1}f(1/z)-z^{-N-1}f(z)}{z-1/z}.
\end{equation*}
It is easy to see that $S_{z,N}:\Rset_{2n-1}[z]\to \Rset_{N}[x]$ for all $N\geq n$. Similarly, we have the  Szeg\H{o} mapping $S_{w,M-1}:\Rset[w]\to \Rset[y]$ defined by 
\begin{equation*}
S_{w,M-1}(g(w))=\frac{w^{M}g(1/w)-w^{-M}g(w)}{w-1/w},
\end{equation*}
and  $S_{w,M-1}:\Rset_{2M+1}[w]\to \Rset_{M}[y]$. Finally, we define the {\em bivariate Szeg\H{o} mapping}
$$\cS=\cS_{N,M-1}=S_{z,N}\circ S_{w,M-1}:\Rset[z,w]\to \Rset[x,y],$$ 
and therefore
$$\cS:\Rset_{2n-1,2M+1}[z,w]\to \Rset_{N,M}[x,y].$$ 
Explicitly, we have
\begin{equation*}
\begin{split}
\cS(f(z,w))=\frac{1}{(z-1/z)(w-1/w)}\Big[z^{N+1}w^{M}f(1/z,1/w)- z^{-N-1}w^{M}f(z,1/w)&\\
\qquad -z^{N+1}w^{-M}f(1/z,w)+z^{-N-1}w^{-M}f(z,w)\Big].&
\end{split}
\end{equation*}
To simplify the notation, we fix in this subsection $N\geq n$, $M\geq m$, and we will write $\cS$ instead of $\cS_{N,M-1}$, unless the $(N,M)$ dependence is important. Repeating the computation in \prref{pr2.1} for the $z$-integral and the $w$-integral it follows that for $g(x,y)\in\Rset[x,y]$ and $f(z,w)\in\Rset[z,w]$ we have
\begin{equation}\label{6.3}
\langle g(x,y), \cS(f(z,w))\rangle_{\cL}=\langle z^{N}w^{M-1}(z^2-1)(w^2-1)g(x,y) ,f(z,w) \rangle_{\cL_{p_c}}.
\end{equation}
Using this equation, it is easy to see that 
\begin{equation}\label{6.4}
\cS: \sPt_{2n-1,2M+1;\cL_{p_c}}^{1}[z,w] \to \sPt_{N,M;\cL}[x,y].
\end{equation}
Indeed, if $g(x,y)\in\Rset_{N,M-1}[x,y]$, then $ z^{N}w^{M-1}(z^2-1)(w^2-1)g(x,y)\in\Rset[z,w]$ is of degree at most $2M$ in $w$, hence \eqref{5.4a} shows that $z^{N}w^{M-1}(z^2-1)(w^2-1)g(x,y) \perp \sPt_{2n-1,2M+1;\cL_{p_c}}^{1}[z,w] $ with respect to $\cL_{p_c}$. Thus, 
$\cS(\sPt_{2n-1,2M+1;\cL_{p_c}}^{1}[z,w])$ is orthogonal to $\Rset_{N,M-1}[x,y]$ with respect to $\cL$, establishing \eqref{6.4}.

Note that if $f(z,w)=w^{2M+1}z^{n_0}p(x,1/w)g(z)\in w^{2M+1} \pt(z,1/w) \Rset_{2n_1-1}[z]$ is an element of the second space in \eqref{6.2a}, then 
\begin{align}
\cS(f(z,w))&=S_{w,M-1}(w^{2M+1}p(x,1/w))S_{z,N}(z^{n_0}g(z))\nonumber\\
&=-S_{w,M}(p(x,w))S_{z,N-n_0}(g(z))\in S_{w,M}(p(x,w)) \Rset_{N-n_0}[x],\label{6.5}
\end{align}
where $S_{w,M}(p(x,w))=p_M(y;x)$ is the polynomial in \reref{re4.7}. We show next that the restriction of $\cS$ to the first space in \eqref{6.2a} builds the orthogonal complement of the orthonormal set $\{p_M(y;x)U^q_j(x)\}_{j=0}^{N-n_0}$ in $ \sPt_{N,M;\cL}[x,y]$.

First, note that if $f(z,w)=\qt(z) h(z,w)\in \qt(z) \sPt_{2n_0-1,2M+1;\cL_{\pt}}^{1}[z,w]$ and $g(x,y)=p_M(y;x)r(x)\in p_M(y;x) \Rset_{N-n_0}[x]$, then by \eqref{6.3} we have
\begin{subequations}
\begin{align}
\langle g(x,y), \cS(f(z,w))\rangle_{\cL} =& \left\langle \frac{(z^2-1)z^{N-n_0}r(x)w^{2M+1}z^{n_0}p(x,1/w)}{\qt(z)} ,h(z,w)  \right\rangle_{\cL_{\pt}} \label{6.6a} \\
&-  \left\langle \frac{z^{N}(z^2-1)r(x)p(x,w)}{w\, \qt(z)} ,h(z,w)  \right\rangle_{\cL_{\pt}}. \label{6.6b} 
\end{align}
\end{subequations}
The right-hand side of \eqref{6.6a} is equal to $0$ by \eqref{5.4b}, and it is easy to see that the right-hand side of \eqref{6.6b} is equal to $0$ by computing the $w$-integral. This shows that the left-hand side of \eqref{6.6a} is also equal to $0$, and therefore 
\begin{equation*}
\cS:\qt(z) \sPt_{2n_0-1,2M+1;\cL_{\pt}}^{1}[z,w] \to \sPt_{N,M;\cL}[x,y] \ominus (p_M(y;x) \Rset_{N-n_0}[x]).
\end{equation*}
Note that both spaces in the last equation have the same dimension $n_0$, and we will prove next that $\cS$ is an isomorphism. However, in general, $\cS$ will not be an isometry if we use the natural inner product $\langle , \rangle_{\cL_{p_c}}$ on $\qt(z) \sPt_{2n_0-1,2M+1;\cL_{\pt}}^{1}[z,w]$.

\begin{Lemma}\label{le6.1}
The map $\cS:\qt(z) \sPt_{2n_0-1,2M+1;\cL_{\pt}}^{1}[z,w] \to \sPt_{N,M;\cL}[x,y]$ is injective. Moreover, for $f,g\in \qt(z) \sPt_{2n_0-1,2M+1;\cL_{\pt}}^{1}[z,w]$ we have
\begin{equation}\label{6.7}
\langle \cS(f(z,w)), \cS(g(z,w))\rangle_{\cL}=\langle  f(z,w), g(z,w)\rangle_{\cL_{p_c}}-\langle w^{2M} f(z,1/w), g(z,w)\rangle_{\cL_{p_c}}.
\end{equation}
\end{Lemma}

\begin{proof}
Suppose first that $f(z,w)\in \qt(z) \sPt_{2n_0-1,2M+1;\cL_{\pt}}^{1}[z,w]$ is a nonzero element. We will show that the coefficient of $y^{M}$ in $\cS(f(z,w))$ is a nonzero polynomial, thus proving that $\cS$ is injective. If $h(z)$ is the stable Fej\'er--Riesz factor of $p(x,0)$, then by \eqref{5.5} we have
$$f(z,w)= \qt(z)h(z)\gamma(z)w^{2M+1}\mod \Rset_{2n-1,2M}[z,w], \text{ where }0\neq \gamma(z)\in\Rset_{n_0-1}[z].$$
Since $S_{w,M-1}(\Rset_{2M}[w])\subset \Rset_{M-1}[y]$, we see that 
\begin{align*}
\cS(f(z,w))= S_{z,N}(\qt(z)h(z)\gamma(z))S_{w,M-1}(w^{2M+1}) \mod \Rset_{N,M-1}[x,y],
\end{align*}
i.e. up to a nonzero constant factor, the coefficient of $y^{M}$ is $S_{z,N}(\qt(z)h(z)\gamma(z))$. For simplicity, we set $\alpha(z)=\qt(z)h(z)\in\Rset_{n_0+2n_1}[z]$ which has no zeros in $\oDset$. If we assume that $S_{z,N}(\alpha(z)\gamma(z))=0$, then 
\begin{equation}\label{6.8}
\alpha(z)\gamma(z)=z^{2N+2}\alpha(1/z)\gamma(1/z)=(z^{2N+3-n_0}\alpha(1/z))(z^{n_0-1}\gamma(1/z)).
\end{equation}
The polynomials $\alpha(z)$ and $(z^{2N+3-n_0}\alpha(1/z))$ are relatively prime in $\Rset[z]$, because the first one 
has no zeros in $\oDset$, while all the zeros of the second one are in $\Dset$. From \eqref{6.8} we see that $\alpha(z)$ must divide $z^{n_0-1}\gamma(1/z)$ and therefore we can rewrite this equation as 
\begin{equation*}
\gamma(z)=(z^{2N+3-n_0}\alpha(1/z))\, \frac{z^{n_0-1}\gamma(1/z)}{\alpha(z)},
\end{equation*}
where the right-hand side is a polynomial of degree at least $2N+3-n_0>n_0$ which is impossible. The contradiction shows that the coefficient of $y^M$ in $\cS(f(z,w))$ is nonzero and therefore $\cS$ is injective.

For the second part, note that if we replace $g(x,y)$ in \eqref{6.3} by $\cS(g(z,w))$ we obtain 
\begin{subequations}\label{6.9}
\begin{align}
&\langle \cS(g(z,w)), \cS(f(z,w))\rangle_{\cL}\nonumber \\
&\qquad=\langle  g(z,w), f(z,w)\rangle_{\cL_{p_c}}-\langle w^{2M} g(z,1/w), f(z,w)\rangle_{\cL_{p_c}}\label{6.9a}\\
&\qquad\quad+\langle  z^{2N+2}w^{2M}g(1/z,1/w), f(z,w)\rangle_{\cL_{p_c}}\label{6.9b}\\
&\qquad\quad-\langle  z^{2N+2}g(1/z,w), f(z,w)\rangle_{\cL_{p_c}}.\label{6.9c}
\end{align}
\end{subequations}
If $f,g\in \sPt_{2n-1,2M+1;\cL_{p_c}}^{1}[z,w] $, then the term in \eqref{6.9b} is $0$ by \leref{le5.2}. If we take $ f(z,w)=\qt(z)f_1(z,w)$ and $ g(z,w)=\qt(z)g_1(z,w)$ to be two elements in  $\qt(z)\sPt_{2n_0-1,2M+1;\cL_{\pt}}^{1}[z,w]$, then the term in \eqref{6.9c} can be rewritten as follows
\begin{equation}\label{6.10}
\langle  z^{2N+2}g(1/z,w), f(z,w)\rangle_{\cL_{p_c}}= \langle  z^{2N+2}g_1(1/z,w), f_1(z,w)\rangle_{\cL_{\pt}}.
\end{equation}
Since $z^{2n_0}\pt(1/z,w)=\pt(z,w)$, \prref{pr5.3} tells us that 
$$z^{2n_0-1}g_1(1/z,w)\in \sPt_{2n_0-1,2M+1;\cL_{\pt}}^{2}[z,w].$$
This shows that $z^{2N+2}g_1(1/z,w) \in z^{2(N-n_0)+3}\sPt_{2n_0-1,2M+1;\cL_{\pt}}^{2}[z,w] $
which combined with \eqref{5.4e} proves that the term in \eqref{6.10} is $0$. Since the terms in \eqref{6.9b}-\eqref{6.9c} are zero, equation \eqref{6.9a} reduces to \eqref{6.7}.
\end{proof}
For $f(z,w)\in\Rset[z,w]$ we denote by $\cM_{f(z,w)}:\Rset[z,w]\to \Rset[z,w]$ the multiplication by $f(z,w)$, i.e. 
$$\cM_{f(z,w)}(g(z,w))=f(z,w)g(z,w).$$
Note that $\cM_{\qt(z)}: \sPt_{2n_0-1,2M+1;\cL_{\pt}}^{1}[z,w] \to \sPt_{2n-1,2M+1;\cL_{p_c}}^{1}[z,w]$ is an isometry.
Summarizing the statement so far we obtain the following theorem.
\begin{Theorem}\label{th6.2}
Let $\cL$ be the functional defined by \eqref{4.3} where $q(x)\in\Rset_{2n_1}[x]$, $p(x,w)\in \Rset_{n_0,2m}[x,w]$ are such that $q(x)>0$, $p(x,w)\neq 0$ when $(x,w)\in[-1,1]\times \oDset$ and let $\pt(z,w)=z^{n_0}p(x,w)$. Let $N\geq n_0+n_1$ and $M\geq m$. 
For  $f,g\in \sPt_{2n_0-1,2M+1;\cL_{\pt}}^{1}[z,w]$, the bilinear form 
\begin{equation}\label{6.11}
\langle f, g \rangle:=\langle  f(z,w), g(z,w)\rangle_{\cL_{\pt}}-\langle w^{2M} f(z,1/w), g(z,w)\rangle_{\cL_{\pt}},
\end{equation}
defines an inner product on  $ \sPt_{2n_0-1,2M+1;\cL_{\pt}}^{1}[z,w]$, and with this inner product the map 
$$\cS_{N,M-1}\circ \cM_{\qt(z)}: \sPt_{2n_0-1,2M+1;\cL_{\pt}}^{1}[z,w]  \to  \sPt_{N,M;\cL}[x,y], $$
is an isometry. Moreover, if $\{U^q_j(x)\}_{j=0}^{N-n_0}$ are orthonormal polynomials with respect to the measure $\frac{2}{\pi}\frac{\sqrt{1-x^2}}{q(x)}\chi_{(-1,1)}(x)dx$ on $\Rset$, and if we set 
$$p_M(y;x)=\frac{w^{M+1} p (x,1/w)- w^{-M-1} p (x,w)}{w-1/w},$$
then 
\begin{equation}\label{6.12}
\sPt_{N,M;\cL}[x,y] = \cS_{N,M-1}(\qt(z) \sPt_{2n_0-1,2M+1;\cL_{\pt}}^{1}[z,w]) \bigoplus_{k=0}^{N-n_0} \Span_{\Rset}\{p_M(y;x)U^q_k(x) \}.
\end{equation}
\end{Theorem}

\begin{Remark}\label{re6.3}
(i) We can now relax the conditions on $q(x)$ and show that the decomposition \eqref{6.12} holds even if $q(x)$ has simple zeros at $x=\pm 1$. Indeed, suppose that $q(x)=(1-x)^{\nu_1}(1+x)^{\nu_2}q_0(x)$, where $\nu_1,\nu_2\in \{0,1\}$ and $q_0(x)$ is nonzero for $x\in[-1,1]$. Then, for $\epsilon>0$, we can consider \eqref{6.12} with $q(x)$ replaced by $q(x;\epsilon)=(1+\epsilon-x)^{\nu_1}(1+\epsilon+x)^{\nu_2}q_0(x)$. Note that $\epsilon$ will not appear in the inner product \eqref{6.11}, and the only terms depending on $\epsilon$ on the right-hand side of \eqref{6.12} are $\qt(z)$ and $U^q_k(x)$, so taking the limit $\epsilon\to 0$ is straightforward.

(ii) From \thref{th3.6} with $l=1$ we know that the recurrence coefficients $a_{k+1}$ and $b_{k}$ of the polynomials $U^q_k(x)$ will be equal to $\frac{1}{2}$ and $0$, respectively, for $k\geq n_1$. Clearly, the orthonormal  elements $p_M(y;x)U^q_k(x)$ of $\sPt_{N,M;\cL}[x,y]$ will satisfy the same recurrence relation. Note that the inner product in \eqref{6.11} depends only on $M$, but not on $N$. From \eqref{6.12} it follows that the orthonormal set $\{p_M(y;x)U^q_k(x)\}$ can be extended to an orthonormal basis of $\sPt_{N,M;\cL}[x,y]$ by adding elements which also satisfy a Chebyshev relation. The precise statement is given below.
\end{Remark}

\begin{Corollary}\label{co6.4}
With the notations in \thref{th6.2}, let $\{\phi_j^M(z,w)\}_{j=1}^{n_0}$ be an orthonormal basis of $\sPt_{2n_0-1,2M+1;\cL_{\pt}}^{1}[z,w]$ with respect to the inner product in \eqref{6.11}, and let
$$r_{j,N}^M(x,y)=\cS_{N,M-1}(\qt(z)\phi_j^M(z,w)).$$
The set 
$$ \{ p_M(y;x)U^q_k(x): k=0,\dots,N-n_0\} \cup \{r_{j,N}^M(x,y):j=1,\dots,n_0\}$$
is an orthonormal basis of $\sPt_{N,M;\cL}[x,y]$ and the elements $r_{j,N}^M(x,y)$ satisfy the Chebyshev relation 
\begin{equation}
r_{j,N}^M(x,y)+r_{j,N+2}^M(x,y)=2xr_{j,N+1}^M(x,y), \quad\text{ for } j=1,\dots,n_0.
\end{equation}
\end{Corollary}

\begin{Remark}\label{re6.5} 
In general, the last term in \eqref{6.11} need not be zero, and therefore, $\cS$ is not an isometry, see for instance the example in \seref{ss7.2}, and the discussion in Subsection~\ref{ss7.2.1}.
\end{Remark}

\subsection{Szeg\H{o} mapping for the Bernstein--Szeg\H{o} weight}
In this subsection, we construct {\em explicit orthonormal bases} of the spaces $\sPt_{N,M;\cL}$ and $\sP_{N,M;\cL}$ for all $N$ and $M$ sufficiently large, when the Bernstein--Szeg\H{o} weight satisfies certain nonvanishing conditions. 
This can be regarded as a far-reaching extension of formula \eqref{1.4} on the real line, where $q$ is replaced by finitely elements in fixed orthonormal bases of polynomials on $\Tset^{2}$ determined from the weight. Here we see explicitly how the  extra orthogonality conditions replace the Cauchy--Goursat  theorem.

Throughout this subsection we assume that \\ 
$\bullet$ $\omega(z,w)\in\Rset_{n_0,m_0}[z,w]$ is nonzero for $(z,w)\in \oDset^2$, \\
$\bullet$ $q_1(x)\in \Rset_{2n_1}[x]$ is positive for $x\in [-1,1]$,  \\
$\bullet$  $q_2(y)\in \Rset_{2m_1}[y]$ is positive for $y\in [-1,1]$,\\
and we set $n=n_0+n_1$, $m=m_0+m_1$. The conditions on $q_1(x)$ and $q_2(y)$ will be relaxed later, see \reref{re6.7}(i).
Let $\qt_1(z)\in\Rset_{2n_1}[z]$ and $\qt_2(w)\in\Rset_{2m_1}[w]$ 
be the stable Fej\'er--Riesz factors of $q_1$ and $q_2$, respectively.
We can apply the constructions in the previous subsection with $q(x)$ replaced by $q_1(x)$ and 
\begin{equation}\label{6.14}
\pt(z,w)=\omega(z,w) \left( \qt_2(w) z^{n_0}\omega(1/z,w)\right).
\end{equation}
By \thref{th5.5}, we have
\begin{equation}\label{6.15}
\sPt_{2n_0-1,2m;\cL_{\pt}}^{1}[z,w]=w^{2m_1+m_0} \qt_2(1/w) \omega(z,1/w) \sPt_{n_0-1,m_0;\cL_{\omega}}[z,w],
\end{equation}
and for $M\geq m$
\begin{equation}\label{6.16}
\sPt_{2n_0-1,2M+1;\cL_{\pt}}^{1}[z,w]=w^{2M+1-m_0} \qt_2(1/w) \omega(z,1/w) \sPt_{n_0-1,m_0;\cL_{\omega}}[z,w].
\end{equation}
Using the last equation we can show that the last term in \eqref{6.11} vanishes, and therefore we do not need to modify the inner product on the space $\sPt_{2n_0-1,2M+1;\cL_{\pt}}^{1}$. Indeed, if $f,g\in \sPt_{2n_0-1,2M+1;\cL_{\pt}}^{1}$, then from \eqref{6.16} we have
\begin{align*}
f(z,w)&=w^{2M+1-m_0} \qt_2(1/w) \omega(z,1/w) f_1(z,w)\\
g(z,w)&=w^{2M+1-m_0} \qt_2(1/w) \omega(z,1/w) g_1(z,w),
\end{align*}
where $f_1,g_1\in  \sPt_{n_0-1,m_0;\cL_{\omega}}[z,w]\subset \Rset_{n_0-1,m_0}[z,w]$. Substituting these formulas in $\langle g(z,w),w^{2M} f(z,1/w)\rangle_{\cL_{\pt}}$ and computing just the $w$-integral, up to a constant factor, we obtain
$$\oint_{\Tset}\frac{w^{2(M-m)+2} g_1(z,w) (w^{2m_1} \qt_2(1/w)) f_1(1/z,w) }{ \qt_2(w) \omega(1/z,w)  \omega(z,w)} \frac{dw}{w}.$$
Since $M\geq m$, the numerator is a polynomial of $w$ while the denominator is nonzero for $w\in\oDset$, and therefore the integral is $0$ by the Cauchy--Goursat  theorem.

If we use the involution $\cR_w^{m_0}$  on $\Cset[z^{\pm 1},w^{\pm 1}]$ defined by 
$$\cR_w^{m_0}(g(z,w))=w^{m_0}g(z,1/w),$$
then a computation similar to the one in \eqref{6.5} shows that
\begin{equation*}
S_{w,M-1}\circ \cM_{w^{2M+1-m_0}g(z,1/w)}=-S_{w,M}\circ \cM_{g(z,w)}\circ \cR_w^{m_0}.
\end{equation*}

Combining the above with formula \eqref{6.16} and \thref{th6.2} we obtain the following theorem for the Bernstein--Szeg\H{o} weight defined in \eqref{4.29} and \eqref{4.31}.

\begin{Theorem}\label{th6.6}
Let $\cL$ be the functional 
\begin{equation}\label{new3}
\cL(f)=\frac{4}{\pi^2}\iint_{(-1,1)^2}\frac{f(x,y)\sqrt{1-x^2}\sqrt{1-y^2}}{q_1(x)q_2(y)\omega(z,w)\omega(1/z,w)\omega(z,1/w)\omega(1/z,1/w)}\,dx\,dy, 
\end{equation}
where\\
$\bullet$ $\omega(z,w)\in\Rset_{n_0,m_0}[z,w]$ is nonzero for $(z,w)\in \oDset^2$, \\
$\bullet$ $q_1(x)\in \Rset_{2n_1}[x]$ is positive for $x\in [-1,1]$,  \\
$\bullet$  $q_2(y)\in \Rset_{2m_1}[y]$ is positive for $y\in [-1,1]$,\\
and let $\qt_1(z)$ and $\qt_2(w)$  be the stable Fej\'er--Riesz factors of $q_1(x)$ and $q_2(y)$, respectively. 
For $N\geq n_0+n_1$ and $M\geq m_0+m_1$, define
\begin{subequations}
\begin{align}
\cTt_{N,M}&=\cS_{N,M}\circ \cM_{\qt_1 (z)\qt_2 (w)\omega(z,w)}\circ \cR_w^{m_0},\\
\cT_{N,M}&=\cS_{N,M}\circ \cM_{\qt_1 (z)\qt_2 (w)\omega(z,w)}\circ \cR_z^{n_0}.
\end{align}
\end{subequations}
Then
\begin{subequations}\label{6.18}
\begin{align}
&\cTt_{N,M}:  \sPt_{n_0-1,m_0;\cL_{\omega}}[z,w]\to  \sPt_{N,M;\cL}[x,y], \\
&\cT_{N,M}:  \sP_{n_0,m_0-1;\cL_{\omega}}[z,w]\to  \sP_{N,M;\cL}[x,y],
\end{align}
\end{subequations}
are isometries. Moreover, if \\
$\bullet$ $\{U^{q_1}_j(x)\}_{j=0}^{N-n_0}$ are orthonormal polynomials with respect to 
$\frac{2}{\pi}\frac{\sqrt{1-x^2}}{q_1(x)}\chi_{(-1,1)}(x)dx$, \\ 
$\bullet$ $\{U^{q_2}_j(y)\}_{j=0}^{M-m_0}$ are orthonormal polynomials with respect to 
$\frac{2}{\pi}\frac{\sqrt{1-y^2}}{q_2(y)}\chi_{(-1,1)}(y)dy$,\\ 
and if we set 
\begin{subequations}\label{6.19}
\begin{align}
&p (x,w)=\qt_2(w)\omega(z,w)\omega(1/z,w), && p_M(y;x)=\frac{w^{M+1} p (x,1/w)- w^{-M-1} p (x,w)}{w-1/w}, \label{6.19a}\\
&\ph (z,y)=\qt_1(z)\omega(z,w)\omega(z,1/w), && \ph_N(x;y)=\frac{z^{N+1} \ph (1/z,y)- z^{-N-1} \ph (z,y)}{z-1/z},\label{6.19b}
\end{align}
\end{subequations}
then 
\begin{subequations}\label{6.20}
\begin{align}
\sPt_{N,M;\cL}[x,y] &= \cTt_{N,M}(\sPt_{n_0-1,m_0;\cL_{\omega}}[z,w]) \bigoplus_{k=0}^{N-n_0} \Span_{\Rset}\{p_M(y;x)U^{q_1}_k(x) \},\label{6.20a}\\
\sP_{N,M;\cL}[x,y] &= \cT_{N,M}(\sP_{n_0,m_0-1;\cL_{\omega}}[z,w]) \bigoplus_{k=0}^{M-m_0} \Span_{\Rset}\{\ph_N(x;y)U^{q_2}_k(y) \}.\label{6.20b}
\end{align}
\end{subequations}
\end{Theorem}

\begin{Remark}\label{re6.7}
(i) Similarly to \reref{re6.3}(i), we can relax the conditions on $q_1(x)$ and $q_2(y)$ and show that the equations \eqref{6.18}-\eqref{6.20} hold even when $q_1(x)$ and $q_2(y)$ have simple zeros at $x=\pm 1$ and $y=\pm 1$. 

(ii) Note that the spaces $ \sPt_{n_0-1,m_0;\cL_{\omega}}[z,w]$ and $\sP_{n_0,m_0-1;\cL_{\omega}}[z,w]$ in  \eqref{6.18} and \eqref{6.20} are independent now of $N$ and $M$. This means that, besides the polynomials $q_{1},q_{2}$ and $\omega$ in the weight in \eqref{new3}, we can take $n_{0}$ polynomials which form an orthonormal basis of $ \sPt_{n_0-1,m_0;\cL_{\omega}}[z,w]$, $m_{0}$ polynomials which form an orthonormal basis of  $\sP_{n_0,m_0-1;\cL_{\omega}}[z,w]$ and we can use them to construct orthonormal bases of $\sPt_{N,M;\cL}[x,y]$ and $\sP_{N,M;\cL}[x,y] $ using the decompositions in equations \eqref{6.20}. Moreover, the multiplications by $x$ and $y$ with respect to these bases will be represented by Chebyshev relations. In particular, this provides a new proof of equations \eqref{4.32} in \thref{th4.11} in the case when $\omega(z,w)\neq 0$ for $(z,w)\in\oDset^2$.
\end{Remark}

\section{Explicit examples}\label{se7}

In this section we include examples which illustrate different aspects of the main results. In both examples, we start with a positive functional $\cL:\Rset_{2n,2m}\to\Rset$ by displaying its moment matrix $\fM=[h_{k,l}]$ where $h_{k,l}=\cL(x^ky^l)$ for $0\leq k\leq 2n$ and $0\leq l\leq 2m$.

\subsection{Example 1: Bernstein--Szeg\H{o} weight}\label{ss7.1}
Let $\cL:\Rset_{4,2}\to\Rset$ be the functional with moment matrix
\begin{equation}\label{7.1}
\fM=\left[
\begin{array}{ccc}
 1 & \frac{c_1 c_2}{2} & \frac{c_1^2 c_2^2+1 }{4}\\
 \frac{c_2}{2} & \frac{c_1 (c_2^2+1)}{4} & \frac{c_2 (c_1^2 c_2^2 +c_1^2+1)}{8}  \\
 \frac{ c_2^2+1}{4} & \frac{c_1 c_2 (c_2^2+2)}{8}& \frac{ (c_2^2+1) (c_1^2 c_2^2 +c_1^2+1)}{16} \\
 \frac{c_2 (c_2^2+2)}{8} & \frac{c_1 (c_2^2+1) (c_2^2+2)}{16} & \frac{c_2 (c_1^2 c_2^4+3 c_1^2 c_2^2+3 c_1^2+c_2^2+2)}{32} \\
 \frac{(c_2^2+1) (c_2^2+2)}{16} & \frac{c_1 c_2 (c_2^4+4 c_2^2+5)}{32} & \frac{(c_2^2+1) (c_1^2 c_2^4+3 c_1^2 c_2^2+3 c_1^2+c_2^2+2)}{64} \\
\end{array}
\right],
\end{equation}
where $c_1,c_2$ are free real parameters. This example was considered in \cite[Section 4]{GeI1}, using the parameters $s_{i,j}$ in \cite{DGIM} related to $c_1$ and $c_2$ above as follows: $c_1=s_{1,1}$, $c_2=2s_{1,0}$. In particular, the results in these papers imply that if 
\begin{equation}\label{7.2}
c_1\in(-1,1)\qquad \text{ and }\qquad c_2\in\Rset,
\end{equation}
then the functional $\cL$ 
extends to a positive linear functional on the space $\Rset[x,y]$ of all polynomials and equations \eqref{4.32} hold with $n=2$ and $m=1$. By abuse of notation, we will denote the extension also by $\cL$. 
The case $c_1=c_2=0$ corresponds to the Chebyshev measure on $\Rset^2$ in \eqref{1.8}.
Using the standard basis $\cB_1=(1,y)$ of $\Rset_1[y]$ we compute the vector orthonormal polynomials and we obtain
\begin{subequations}\label{7.3}
\begin{align}
P_{0,1}(x,y)&= 
\left[
\begin{array}{c}
1\\
2y-c_1c_2  \\
\end{array}
\right],\label{7.3a}\\
P_{1,1}(x,y)&= 
\left[
\begin{array}{c}
 \frac{2 x-2 c_1 y+c_1^2c_2-c_2}{\sqrt{1-c_1^2}} \\
4 x y -2 c_2 y-c_1 \\
\end{array}
\right],\label{7.3b}\\
P_{2,1}(x,y)&=\left[
\begin{array}{c}
 \frac{4 x^2-4 c_1 x y-2 c_2 x+2 c_1 c_2 y+c_1^2-1}{\sqrt{1-c_1^2}} \\
8 x^2 y -4 c_2 x y-2 c_1 x-2 y+c_1 c_2 \\
\end{array}
\right],\label{7.3c}\\
\intertext{and therefore}
A_{2,1}&=\frac{1}{2}I_{2}.\label{7.3d}
\end{align}
\end{subequations}
Substituting the above formulas into the left-hand side of \eqref{4.8} and following the steps outlined in Remarks \ref{re2.7} and \ref{re4.3}, we see that \eqref{4.8} holds when $n=2$, $m=1$ with
\begin{subequations}\label{7.4}
\begin{equation}\label{7.4a}
q(x)=1+c_2^2-2c_2x=(1-c_2z)(1-c_2/z)
\end{equation}
and 
\begin{equation}\label{7.4b}
p_1(y;x)=\frac{2}{\sqrt{1-c_1^2}}(y-c_1x), \qquad p_0(y;x)=\sqrt{1-c_1^2}.
\end{equation}
\end{subequations}
Equivalently, this means that \eqref{4.4} holds if we set $p(x,w)=w(p_1(y;x)-wp_0(y;x))$, or explicitly 
\begin{equation}\label{7.5}
p(x,w)=\frac{1}{\sqrt{1-c_1^2}}(c_1^2w^2-2c_1xw+1)=\frac{1}{\sqrt{1-c_1^2}}(c_1w-z)(c_1w-1/z).
\end{equation}
Equation \eqref{7.2} shows that $p(x,w)\neq 0$ for $(x,w)\in (-1,1)^2$, i.e. condition (a) in \thref{th4.1}(II) holds. For (b), we compute $\hat{\Psi}^1_{2}(z)$ and we find
\begin{equation}\label{7.6}
\hat{\Psi}^1_{2}(z)=(1-c_2z) \left[
\begin{array}{cc}
 \frac{c_1^2 z^2+1}{\sqrt{1-c_1^2}} & -\frac{2 c_1 z}{\sqrt{1-c_1^2}} \\
 -c_1 z & 2 \\
\end{array}
\right],  \text{ hence }\det (\hat{\Psi}^1_{2}(z)) = \frac{2(1-c_2z)^2 }{\sqrt{1-c_1^2}}.
\end{equation}
Therefore, $\hat{\Psi}^1_{2}(z)$ will be invertible for $z\in (-1,1)$ if and only if $c_2\in [-1,1]$. We continue by dividing this subsection into 2 parts. In the first part, we consider the case when $c_2\in [-1,1]$ and we illustrate Theorems \ref{th4.1} and \ref{th6.6}. In the second part, we describe the measure for the extension of  $\cL$ when $c_2\not\in [-1,1]$, which completes, extends and corrects the results in \cite[Theorem 4.1]{GeI1} where one range of the parameters was overlooked.

\subsubsection{The case $c_2\in [-1,1]$}\label{ss7.1.1} Throughout Subsection~\ref{ss7.1.1} we assume that 
$$(c_1,c_2)\in (-1,1)\times [-1,1].$$
Since both conditions in  \thref{th4.1}(II) are satisfied, it follows that \eqref{4.3} holds where $q(x)$ and $p(x,w)$ are given in equations \eqref{7.4a} and \eqref{7.5}, respectively. Thus $\cL$ can be extended to a positive functional on $\Rset[x,y]$, which we will denote also by $\cL$. If we set
 \begin{equation}\label{7.7}
\omega(z,w)=\frac{1-c_1zw}{\sqrt[4]{1-c_1^2}},
\end{equation}
then we can apply \thref{th4.11} with $q_1(x)=q(x)$ given in \eqref{7.4a} and $q_2(y)=1$ to conclude that equations \eqref{4.32} hold with $n=2$ and $m=1$. Moreover, we can use \thref{th6.6} to construct simple explicit bases of $\sP_{N,M;\cL}[x,y]$ and $\sPt_{N,M;\cL}[x,y]$ for all $N\geq 2$, $M\geq1$, using $p(x,w)$ in \eqref{7.5} and 
\begin{equation}\label{7.8}
\ph(z,y)=\frac{(1-c_2z)(1-2c_1zy+c_1^2z^2)}{\sqrt{1-c_1^2}}.
\end{equation}
Since $p(x,0)=\frac{1}{\sqrt{1-c_1^2}}$ and $\ph(0,y)=\frac{1}{\sqrt{1-c_1^2}}$ are positive constants, we can obtain orthonormal bases of $\sP_{N,M;\cL}[x,y]$ and $\sPt_{N,M;\cL}[x,y]$ corresponding to the standard bases of $\Rset_M[y]$ and $\Rset_N[x]$, see \reref{re4.7}(ii). We illustrate this for the space $\sP_{N,M;\cL}[x,y]$ which also shows how to obtain \eqref{7.3c} from \thref{th6.6}. From \eqref{7.8} and \eqref{6.19b} it follows that
\begin{equation}
\ph_N(x;y)=\frac{1}{\sqrt{1-c_1^2}}(U_N(x)-(c_2+2c_1y)U_{N-1}(x)+c_1(c_1 + 2 c_2 y)U_{N-2}(x)-c_1^2 c_2U_{N-3}(x)),
\end{equation}
where $U_k(x)=\frac{z^{k+1}-z^{-k-1}}{z-z^{-1}}$ are the Chebyshev polynomials of the second kind, with the convention that $U_{-1}(x)=0$ and $U_{k}(x)=-U_{-k-2}(x)$ if $k<-1$.
Note next that 
$$\cL_{\omega}(z^kw^l)=\frac{\sqrt{1-c_1^2}}{(2\pi)^2}\iint_{\Tset^2}z^kw^l\, \frac{|dz|\,|dw|}{(1-c_1zw)(1-c_1/(zw))},$$
and the space $\sP_{1,0;\cL_{\omega}}[z,w]$ is spanned by the unit element $\phi(z)=\sqrt[4]{1-c_1^2}\,z$. Since  $\qt(z)=1-c_2z$ is the stable Fej\'er--Riesz factor of $q(x)$, we have
\begin{align*}
&\cT_{N,M}(\phi(z))=\cS_{N,M}((1-c_2z)(1-c_1zw))\\
&\quad=U_{N}(x)U_{M}(y)-c_2 U_{N-1}(x)U_{M}(y) - c_1U_{N-1}(x)U_{M-1}(y) + c_1 c_2 U_{N-2}(x)U_{M-1}(y).
\end{align*}
From the above computations and equation \eqref{6.20b} it follows that for $N\geq 2$ and $M\geq 1$ the entries of the vector polynomials $P_{N,M}(x,y)$ with respect to the standard basis $\cB_M=(1,y,\dots,y^M)$ of $\Rset_M[y]$ are given by 
\begin{subequations}\label{7.10}
\begin{equation}
p_{N,M}^{k}(x,y)=\frac{(U_N(x)-(c_2+2c_1y)U_{N-1}(x)+c_1(c_1 + 2 c_2 y)U_{N-2}(x)-c_1^2 c_2U_{N-3}(x))U_k(y)}{\sqrt{1-c_1^2}},
\end{equation}
for $k=0,1,\dots,M-1$, and 
\begin{equation}
\begin{split}
p_{N,M}^{M}(x,y)&=U_{N}(x)U_{M}(y)-c_2 U_{N-1}(x)U_{M}(y) \\
&\qquad - c_1U_{N-1}(x)U_{M-1}(y) + c_1 c_2 U_{N-2}(x)U_{M-1}(y).
\end{split}
\end{equation}
\end{subequations}
It is easy to see that  equations \eqref{7.10} with $N=2$ and $M=1$ lead to \eqref{7.3c}. Note that the formulas in equations \eqref{7.10} work also when $N=M=1$ and give \eqref{7.3b}, which explains \eqref{7.3d} (cf. with the last part of \thref{th3.4}).
Similar formulas for the orthonormal basis $\{\pt_{N,M}^{j}(x,y)\}$ of  $\sPt_{N,M;\cL}[x,y]$ stem from \eqref{6.20a}.

\subsubsection{The case $c_2\not\in [-1,1]$} In this subsection we analyze the possible cases when $c_1\in(-1,1)$ while $c_2\not\in [-1,1]$. We set 
$$x_0=\frac{1}{2}\left(c_2+\frac{1}{c_2}\right),\quad  y_0=\frac{1}{2}\left(c_1c_2+\frac{1}{c_1c_2}\right), \text{ when }c_1\neq 0,$$ 
and we show that
\begin{enumerate}[(i)]
\item if $|c_1c_2|\leq 1<|c_2|$, then 
\begin{align}
\cL(f)&=\frac{4}{\pi^2}\iint_{(-1,1)^2}\frac{f(x,y)\,\sqrt{1-x^2}\sqrt{1-y^2}}{q(x)p(x,w)p(x,1/w)}dx\,dy\nonumber\\
&\quad
+\frac{2(c_2^2-1)}{\pi c_2^2}\int_{(-1,1)}\frac{ f(x_0,y)}{p(x_0,w)p(x_0,1/w)}\sqrt{1-y^2}dy,\label{7.11}
\end{align}
\item if $|c_1|<1<|c_1c_2|$, then 
\begin{align}
\cL(f)&=\frac{4}{\pi^2}\iint_{(-1,1)^2}\frac{f(x,y)\,\sqrt{1-x^2}\sqrt{1-y^2}}{q(x)p(x,w)p(x,1/w)}dx\,dy\nonumber\\
&\quad
+\frac{2(c_2^2-1)}{\pi c_2^2}\int_{(-1,1)}\frac{ f(x_0,y)}{p(x_0,w)p(x_0,1/w)}\sqrt{1-y^2}dy
+\frac{c_1^2c_2^2-1}{c_1^2c_2^2}f(x_0,y_0).\label{7.12}
\end{align}
\end{enumerate}
For $c_1,c_2$ satisfying \eqref{7.2} we define the functional $\cL_0:\Rset[x,y]\to \Rset$ by 
\begin{equation*}
\cL_0(f)=\frac{4}{\pi^2}\iint_{(-1,1)^2}f(x,y)\frac{\sqrt{1-x^2}\sqrt{1-y^2}}{q(x)p(x,w)p(x,1/w)}\,dx\,dy,
\end{equation*}
which coincides with $\cL$ when $(c_1,c_2)\in (-1,1)\times [-1,1]$ by the results in Subsection~\ref{ss7.1.1}. 
Write $z=e^{i\theta}$, $w=e^{i\phi}$, or equivalently $x=\cos\theta$, $y=\cos\phi$, and use the invariance of the integrand of $\cL_0$ under the transformations $z\to1/z$ and $w\to1/w$ to obtain
\begin{equation*}
\cL_0(f)=\frac{1}{\pi^2}\iint_{(-\pi,\pi)^2}\frac{f(\cos\theta,\cos\phi)\sin^2\theta\sin^2\phi} {q(x)p(x,w)p(x,1/w)} \,d\theta\, d\phi.
\end{equation*}

If $c_1=0$, equation \eqref{7.11} follows easily from the one-dimensional theory, so we will assume below that $c_1\neq 0$.
We begin with the region $0<|c_1|<|c_2|< 1$ and evaluate the above integral by residues. Thus we write
\begin{equation}\label{7.13}
\cL_0(f)=\frac{1}{\pi}\int_{-\pi}^{\pi}\cI(f,w)\sin^2\phi d\phi 
\end{equation}
where
\begin{equation}\label{7.14}
\cI(f,w)=\frac{1}{\pi}\int_{-\pi}^{\pi}\frac{f(\cos\theta,\cos\phi)\sin^2\theta}{q(x)p(x,w)p(x,1/w)}\,d\theta 
=-\frac{1-c_1^2}{4\pi i}\oint_{\Tset}\frac{f(x,y)w^2(1-z^2)^2}{\tau(z,w)}dz, 
\end{equation}
and
$$
\tau(z,w)=(c_1zw - 1)(c_1z - w)(zw - c_1)(c_1w - z)(c_2z - 1)(z -c_2).
$$
The residues of the above contour integral are at $z=0$, $z=c_1w$, $z=c_1/w$, and $z=c_2$. Hence
$$
\cI(f,w)=R_0(w)+R_1(w)+R_2(w)+R_3(w)
$$
where 
\begin{subequations}
\begin{align}
R_0(w)&=-\frac{1-c_1^2}{2}\text{res}_{z=0}\frac{f(x,y)w^2(1-z^2)^2}{\tau(z,w)},\label{7.15a}\\
R_1(w)&=-\frac{1}{2c_1}\frac{f(\frac{1}{2}(c_1w+\frac{1}{c_1w}),y)w(w^2c_1^2-1)}{(w^2 - 1)(c_1w - c_2)(c_1c_2w-1)},\label{7.15b}\\
R_2(w)&=R_1(1/w),\\
\intertext{and}
R_3(w)&=-\frac{(1-c_1^2)(1-c_2^2 )f(x_0,y)w^2}{2(c_1c_2w - 1)(w-c_1c_2)(c_2w - c_1)(c_1w - c_2)}.\label{7.15d}
\end{align}
\end{subequations}
Thus $\cL(f)=\cL_0(f)$ when $0<|c_1|<|c_2|< 1$ and
\begin{align}
\cL(f)&=\frac{1}{\pi}\int_{-\pi}^{\pi}(R_0(w)+2R_1(w)+R_3(w))\sin^2\phi d\phi \nonumber \\
&=-\frac{1}{4\pi i}\oint_{\Tset}(R_0(w)+2R_1(w)+R_3(w))\frac{(1-w^2)^2}{w^3} dw.\label{7.16}
\end{align}
Applying the Cauchy's residue theorem we obtain
\begin{align}
&\cL(f)=- \frac{1}{2}\res_{w=0}  \left( R_0(w)\frac{(1-w^2)^2}{w^3} \right)- \res_{w=0}  \left( R_1(w)\frac{(1-w^2)^2}{w^3} \right) \nonumber  \\
&\qquad-\frac{1}{2}\res_{w=0}  \left( R_3(w)\frac{(1-w^2)^2}{w^3} \right)- \frac{1}{2}\sum_{j=0}^{1}\res_{w=w_j}  \left( R_3(w)\frac{(1-w^2)^2}{w^3} \right),\label{7.17}
\end{align}
where $w_0=c_1c_2$, $w_1=c_1/c_2$. Equation~\eqref{7.17} provides an explicit formal extension of the functional $\cL$ for all $c_1,c_2$ satisfying \eqref{7.2}. In the remaining part, we explain how this formula coincides with the integral formulas given in \eqref{7.11}-\eqref{7.12}, depending on the values of the parameters. 

We consider now the region $|c_1c_2|\leq 1<|c_2|$ in (i). In this case the residue at $z=c_2$ in the evaluation of \eqref{7.14} is replaced by the residue at $z=\frac{1}{c_2}$ which gives
\begin{equation}\label{7.18}
\hat R_3(w)=\frac{(1-c_1^2)(1-c_2^2 )f(x_0,y)w^2}{2(c_1c_2w - 1)(w-c_1c_2)(c_2w - c_1)(c_1w - c_2)}=-R_3(w).
\end{equation}
Since, 
$$-\res_{w=0}  \left( R_3(w)\frac{(1-w^2)^2}{w^3} \right)-\sum_{j=0}^{1}\res_{w=w_j}  \left( R_3(w)\frac{(1-w^2)^2}{w^3} \right)=\frac{c_2^2-1}{c_2^2} f(x_0,y),$$
equations \eqref{7.13}, \eqref{7.14}, \eqref{7.17} and \eqref{7.18} yield
$$\cL(f)-\cL_0(f)=\frac{2(c_2^2-1)}{\pi c_2^2}\int_{(-1,1)}\frac{ f(x_0,y)}{p(x_0,w)p(x_0,1/w)}\sqrt{1-y^2}dy,$$
proving \eqref{7.11}.

We finally consider the case when $1<|c_1c_2|$. We denote by $\cL_1$ the functional in \eqref{7.11}, which coincides with $\cL$ and \eqref{7.16} holds for $\cL$ and $\cL_1$ when $|c_1c_2|\leq 1<|c_2|$. Note that if we apply the residue theorem for $\cL_1$ in \eqref{7.16} when $1<|c_1c_2|$, we will have to compute also the residues of $R_1(w)$ at $w=1/w_0$ and replace the residue of $R_3(w)$ with the residue at $w=1/w_0$. 
Since
\begin{align*}
&\res_{w=1/w_0}  \left( R_1(w)\frac{(1-w^2)^2}{w^3} \right)=\frac{c_1^2c_2^2-1}{2c_1^2c_2^2}f(x_0,y_0)\\
&\res_{w=1/w_0}  \left( R_3(w)\frac{(1-w^2)^2}{w^3} \right)=-\res_{w=w_0}  \left( R_3(w)\frac{(1-w^2)^2}{w^3} \right)= \frac{c_1^2c_2^2-1}{2c_1^2c_2^2}f(x_0,y_0),
\end{align*}
we see that
\begin{align*}
\cL(f)-\cL_1(f)&=-\frac{1}{2}\res_{w=w_0}  \left( R_3(w)\frac{(1-w^2)^2}{w^3} \right)+\frac{1}{2}\res_{w=1/w_0}  \left( R_3(w)\frac{(1-w^2)^2}{w^3}\right) \\
&\qquad +\res_{w=1/w_0}  \left( R_1(w)\frac{(1-w^2)^2}{w^3} \right)=\frac{c_1^2c_2^2-1}{c_1^2c_2^2}f(x_0,y_0),
\end{align*}
completing the proof of \eqref{7.12}.

\subsection{Example 2: one-sided factorization}\label{ss7.2}
Let $\cL:\Rset_{2,2}\to\Rset$ be the functional with moment matrix
\begin{equation}\label{7.19}
\fM=\left[
\begin{array}{ccc}
 1 & -\frac{c_0}{6}-\frac{173 c_1}{2688} & \frac{13 c_0^2}{432}+\frac{731 c_1 c_0}{24192}+\frac{38509 c_1^2}{3096576}+\frac{1}{3} \\
 \frac{113}{448} & -\frac{173 c_0}{2688}-\frac{70429 c_1}{1204224} & \frac{731 c_0^2}{48384}+\frac{38509 c_1 c_0}{1548288}+\frac{85970699 c_1^2}{9710862336}+\frac{2}{21} \\
 \frac{59809}{200704} & -\frac{70429 c_0}{1204224}-\frac{19352717 c_1}{539492352} & \frac{38509 c_0^2}{3096576}+\frac{85970699 c_1 c_0}{4855431168}+\frac{33713900827 c_1^2}{4350466326528}+\frac{61}{588} \\
\end{array}
\right],
\end{equation}
where $c_0,c_1$ are free real parameters. 
To simplify the notation, we set 
$$d_0=\sqrt{\left(28 c_0+53 c_1\right){}^2+112896}\quad\text{ and }\quad d_1=\sqrt{\left(28 c_0+53 c_1\right){}^2+1382976}.$$

Using the standard basis $\cB_1=(1,y)$ of $\Rset_1[y]$ we compute the vector orthonormal polynomials and we obtain
\begin{subequations}\label{7.20}
\begin{align}
P_{0,1}(x,y)&= 
\left[
\begin{array}{c}
1\\
\frac{\sqrt{3}}{8d_0}(2688 y +448 c_0+173 c_1 )\\
\end{array}
\right],\label{7.20a}\\
P_{1,1}(x,y)&=\left[
\begin{array}{c}
 \frac{112 d_0^2x+30240(28 c_0+53 c_1)y+(28 c_0+53 c_1) (4249 c_0+449 c_1)-3189312}{4 \sqrt{15}\, d_0 d_1} \\
 \frac{896 d_1^2xy+45(9408(448 c_0+173 c_1)-5 c_1 d_0^2)x-32 (49 d_0^2+4 d_1^2)y-45 (5 c_0 d_0^2+2688(14 c_0-311 c_1))}{224 \sqrt{15}\, d_0 d_1}\\  
 \end{array}
\right],\label{7.20b}
\intertext{and therefore}
A_{1,1}&=\left[
\begin{array}{cc}
 \frac{\sqrt{15} d_1}{28 d_0} & 0 \\
 \frac{3 \sqrt{5} (28 c_0+53 c_1)}{14 d_1} & \frac{252 \sqrt{5}}{d_1} \\
\end{array}
\right].\label{7.20c}
\end{align}
\end{subequations}
In particular, these equations show that $\cL$ is a positive linear functional. 
Substituting the above formulas into the left-hand side of \eqref{4.8}, we see that \eqref{4.8} holds when $n=m=1$ with $q(x)=1$ and 
\begin{equation}\label{7.21}
p_1(y;x)=\frac{32 (53-28 x) y+225(c_1 x+c_0)}{224 \sqrt{15}}, \qquad p_0(y;x)=\frac{1073-448 x}{224 \sqrt{15}}.
\end{equation}
Equivalently, this means that \eqref{4.4} holds if we take $p(x,w)=w(p_1(y;x)-wp_0(y;x))$, or explicitly 
\begin{equation}\label{7.22}
p(x,w)=\frac{16 (53-28 x)+225(c_1 x+c_0)w-225w^2}{224 \sqrt{15}}.
\end{equation}
Note that for $x\in(-1,1)$, the product of the roots $w_1,w_2$ of the equation $p(x,w)=0$ is 
\begin{equation}\label{7.23}
w_1w_2=-\frac{16 (53-28 x)}{225}<-\frac{16}{9}.
\end{equation}
This means that for $x\in(-1,1)$, the roots are real, have opposite signs and $|w_1w_2|>1$, i.e. at most one of them is in $\oDset$. Since
\begin{equation}\label{7.24}
p(x,w)=-\frac{15\sqrt{15}}{224}(w-w_1)(w-w_2),
\end{equation}
it follows that condition (a) in \thref{th4.1}(II) holds if and only if $p(x,1)\geq 0$ and $p(x,-1)\geq 0$ for $x\in(-1,1)$. These conditions can be rewritten as 
$$|c_0+c_1x|\leq \frac{7(89-64x)}{225}\quad \text{for all }\quad  x\in(-1,1).$$
It is easy to see that the last condition holds if and only if the inequality is true at $x=\pm 1$. This shows that $p(x,w)\neq 0$ for $(x,w)\in(-1,1)^2$ if and only if the parameters $c_0$ and $c_1$ satisfy the following conditions:
\begin{equation}\label{7.25}
|c_0+c_1|\leq \frac{7}{9}\quad \text{ and }\quad |c_0-c_1|\leq \frac{119}{25}. 
\end{equation}
For (b), we compute $\hat{\Psi}^1_{1}(z)$ and we find
\begin{equation}\label{7.26}
\det (\hat{\Psi}^1_{1}(z)) = \frac{1}{105} \left(\frac{7}{2}-z\right)^3 \left(\frac{32}{7}-z\right).
\end{equation}
Therefore, $\hat{\Psi}^1_{1}(z)$ is invertible for all $z\in (-1,1)$. Summarizing, we see that  the conditions in  \thref{th4.1}(II) hold with $q(x)=1$ and $p(x,w)$ given in equation \eqref{7.22} if and only if the parameters $c_0,c_1$ satisfy \eqref{7.25}. 
In particular, this means that if equation~\eqref{7.25} holds, then $\cL$ extends to a positive linear functional on $\Rset[x,y]$ and the recurrence coefficients satisfy \eqref{4.5} with $n=m=1$. Note that if we have equality in one or both inequalities in \eqref{7.25}, then $p(x,w)$ will vanish at some of the points $(\pm 1,\pm 1)$ on the boundary of $(-1,1)\times\Dset$ and we obtain an example similar to the ones discussed in \reref{re4.4}.

If $M\geq 1$, we can use \eqref{4.19} to construct $N$ orthonormal elements in $\sPt_{N,M;\cL}[x,y]$ for every $N\geq 1$. For instance, when $M=1$ we can take the polynomial $p_1(y;x)$ in \eqref{7.21} and the polynomials $\{p_1(y;x)U_j(x)\}_{j=0}^{N-1}$ will form an orthonormal set in $\sPt_{N,1;\cL}[x,y]$, which can be completed to a basis of  $\sPt_{N,1;\cL}[x,y]$ by adding just one element. If $p(x,w)\neq 0$ for $(x,w)\neq [-1,1]\times\oDset$ we can construct this element using \thref{th6.2}. We illustrate this below in the case $M=1$. A direct computation shows that when $N=1$ the polynomial 
\begin{equation}\label{7.27}
\pt_{1,1}^{1}(x,y)=\frac{\sqrt{3}\left(18816 xy+7(448 c_0+173 c_1)x -2688 y+2(311 c_1-14 c_0)\right)}{28 d_0}
\end{equation}
is orthogonal to $\pt_{1,1}^{0}(x,y)=p_1(y;x)$ and has norm $1$. Therefore, $(\pt_{1,1}^{0}(x,y), \pt_{1,1}^{1}(x,y))$ is an orthonormal basis of $\sPt_{1,1;\cL}[x,y]$ with vector polynomial
$$\Pt_{1,1}(x,y)= 
\left[
\begin{array}{c}
\pt_{1,1}^{0}(x,y)\\
\pt_{1,1}^{1}(x,y)
\end{array}
\right],$$
corresponding to the basis $\cBt_1=(53 - 28x,x)$ of $\Rset_1[x]$ by \reref{re4.7}(ii). Note that $p(x,0)$ is not a constant, and thus the polynomials are not associated with the standard basis $(1,x)$ of $\Rset_1[x]$. 
If we define the orthogonal matrix
$$\tilde{R}_{1,1}=\frac{1}{d_1}\left[\begin{array}{cc}
504 \sqrt{5} & -d_0\\
d_0& 504 \sqrt{5}
\end{array}\right],$$
then $\tilde{R}_{1,1}^t\Pt_{1,1}(x,y)$ yields the orthonormal basis of $\sPt_{1,1;\cL}[x,y]$ with respect to the standard basis $(1,x)$ of $\Rset_1[x]$.

\subsubsection{Illustration of \thref{th6.2}}\label{ss7.2.1}
In this subsection we illustrate  \thref{th6.2} when $M=1$. For simplicity, we consider first the case
\begin{equation}\label{7.28}
c_1=c_2=0
\end{equation}
which will be sufficient to explain some of the subtle points very explicitly, and in particular, the need to modify the inner product on $\sPt^{1}_{1,3;\cL_{p_c}}[z,w]$. Then we apply the constructions for arbitrary parameters $c_1,c_2$ satisfying the strict inequalities in \eqref{7.25}, but we omit some of the details since the explicit formulas for the polynomials on the torus $\Tset^2$ become very involved. \\

\noindent
{\em \underline{Special case: $c_1=c_2=0$}.} Note that for $c_1=c_2=0$ the polynomial $p(x,w)$ in \eqref{7.22} reduces to 
$$p(x,w)=\frac{16 (53-28 x)-225w^2}{224 \sqrt{15}},$$
and therefore
$$p_c(z,w)=zp(x,w)=-\frac{16 (2 z-7) (7 z-2)+225zw^2}{224 \sqrt{15}}.$$
With the above polynomial, we define the positive linear functional 
$$\cL_{p_c}(f)=\frac{1}{(2\pi)^2}\iint_{\Tset^2}\, \frac{f(z,w)\, |dz|\,|dw|}{|p_c(z,w)|^2},$$
on $\Rset[z^{\pm 1},w^{\pm 1}]$. Note that here $\qt(z)=1$ and therefore $\pt(z,w)=p_c(z,w)$ in the notations of \thref{th6.2}. A straightforward computation shows that the space $\sPt^{1}_{1,3;\cL_{p_c}}[z,w]$ is spanned by the element 
\begin{equation}\label{7.29}
\tilde{\phi}(z,w)=\frac{\sqrt{3} w \left(64 w^2 (2 z-7)+(7 \sqrt{17}-33) z+32 (9-\sqrt{17})\right)}{896}.
\end{equation}
The polynomial $\tilde{\phi}(z,w)$ is normalized so that 
\begin{subequations}\label{7.30}
\begin{align}
\langle  \tilde{\phi}(z,w), \tilde{\phi}(z,w)\rangle_{\cL_{p_c}}&=\frac{3}{56} (23-\sqrt{17})\label{7.30a}\\
\langle w^{2} \tilde{\phi}(z,1/w), \tilde{\phi}(z,w)\rangle_{\cL_{p_c}}&=\frac{1}{56} (13-3 \sqrt{17}).\label{7.30b}
\end{align}
\end{subequations}
This means that for the inner product defined in \eqref{6.11} we have
\begin{align*}
\langle  \tilde{\phi}(z,w), \tilde{\phi}(z,w)\rangle =\langle  \tilde{\phi}(z,w), \tilde{\phi}(z,w)\rangle_{\cL_{p_c}}-\langle w^{2} \tilde{\phi}(z,1/w), \tilde{\phi}(z,w)\rangle_{\cL_{p_c}}=1.
\end{align*}
Note, in particular that the term in \eqref{7.30b} is nonzero, and therefore this term cannot be omitted. 
By \thref{th6.2}, for every $N\geq 1$, the element 
\begin{align*}
\pt_{N,1}^{N}(x,y)=\cS_{N,0}( \tilde{\phi}(z,w))&=\frac{\sqrt{3}}{2}\,\cS_{N,0}\left(w^3 \left(-1+\frac{2}{7}z\right)\right)=\sqrt{3} \left(U_N(x)-\frac{2}{7}U_{N-1}(x)\right)y,
\end{align*}
completes the set $(p_1(y;x),p_1(y;x)U_1(y),\dots,p_1(y;x)U_{N-1}(y))$ to an orthonormal basis of $\sPt_{N,1;\cL}[x,y]$. It is straightforward to check that for $N=1$ the above formula coincides with $\pt_{1,1}^{1}(x,y)$ in \eqref{7.27} when $c_1=c_2=0$.\\

\noindent
{\em \underline{General case.}} Suppose now that $c_0$ and $c_1$ are real constants, such that
\begin{equation}\label{7.31}
|c_0+c_1|< \frac{7}{9}\quad \text{ and }\quad |c_0-c_1|<\frac{119}{25}. 
\end{equation}
We set $\pt(z,w)=p_c(z,w)=zp(x,w)$ where $p(x,w)$ is the polynomial in \eqref{7.22} and we consider the positive linear functional 
$$\cL_{p_c}(f)=\frac{1}{(2\pi)^2}\iint_{\Tset^2}\, \frac{f(z,w)\, |dz|\,|dw|}{|p_c(z,w)|^2},$$
on $\Rset[z^{\pm 1},w^{\pm 1}]$. A straightforward computation shows that the space $\sPt^{1}_{1,3;\cL_{p_c}}[z,w]$ endowed with the inner product in \eqref{6.11} is spanned by a unit element of the form
\begin{align}
\tilde{\phi}(z,w)&=(\tilde{\phi}_{0,0}+\tilde{\phi}_{1,0}z)(1+w^2)+(\tilde{\phi}_{0,1}+\tilde{\phi}_{1,1}z)w+\tilde{\phi}_0(z,w)\nonumber\\
\intertext{where}
\tilde{\phi}_0(z,w)&=\frac{\sqrt{3}}{56d_0}\left(4(14 c_0-311 c_1)z-7(448 c_0+173 c_1)\right)w^2+\frac{24 \sqrt{3} (2 z-7)w^3}{d_0},\label{7.32}
\end{align}
with $\tilde{\phi}_{0,0},\tilde{\phi}_{1,0},\tilde{\phi}_{0,1},\tilde{\phi}_{1,1}\in\Rset$. Since $S_{w,0}(w)=S_{w,0}(1+w^2)=0$, we have
\begin{align}
&\pt_{N,1}^N(x,y)=\cS_{N,0}( \tilde{\phi}(z,w))=\cS_{N,0}( \tilde{\phi}_0(z,w))=\frac{\sqrt{3}}{56d_0}\Big( 7(448 c_0+173 c_1)U_{N}(x) \nonumber\\
&\qquad-4(14 c_0-311 c_1)U_{N-1}(x)+2688 (7U_{N}(x)-2U_{N-1}(x))y  \Big).\label{7.33}
\end{align}
It is easy to see that when $N=1$ the above formula  agrees with \eqref{7.27}.
Recall that $p_1(y;x)$ is given in \eqref{7.21}. With these formulas, \thref{th6.2} tells us that if  $N\geq 1$ and we consider the basis 
 $$\cBt_N=(53 - 28x,(53 - 28x)x,(53 - 28x)x^2,\dots,(53 - 28x)x^{N-1},x^N)\text{ of } \Rset_N[x],$$
then 
$$\Pt_{N,1}(x,y)= 
\left[
p_1(y;x), \, p_1(y;x)U_1(x), \, \cdots, \, 
p_1(y;x)U_{N-1}(x), \, \pt_{N,1}^N(x,y)
\right]^t.$$

\subsubsection{Extension}\label{ss7.2.2} As in the previous example, there are extensions to regions where $p(x,w)$ can vanish when $(x,w)\in(-1,1)^2$. We will consider one such extension, 
\begin{equation}\label{7.34}
c_0+c_1\geq  \frac{7}{9}\quad \text{ and }\quad c_0-c_1\geq \frac{119}{25},
\end{equation}
and we will assume that at least one of these inequalities is strict, so that \eqref{7.25} does not hold. For fixed $x\in\Rset$, we denote by $w_1=w_1(x)$ and $w_2=w_2(x)$ the roots of $p(x,w)=0$
\begin{align*}
w_1&=\frac{1}{2}\left(c_0 +c_1 x -\sqrt{\frac{64}{225}(53 - 28 x) + (c_0 + c_1 x)^2}\right)\\
w_2&=\frac{1}{2}\left(c_0 +c_1 x +\sqrt{\frac{64}{225} (53 - 28 x) + (c_0 + c_1 x)^2}\right),\\
\intertext{and we set }
y_1(x)&=\frac{1}{2}\left(w_1(x)+\frac{1}{w_1(x)}\right).
\end{align*}
With these notations, we show below that if \eqref{7.34} holds, then
\begin{equation}\label{7.35}
-1<w_1(x)<0<w_2(x)\qquad \text{ for }\quad x\in(-1,1),
\end{equation}
and
\begin{align}
\cL(f)&=\frac{4}{\pi^2}\iint_{(-1,1)^2}\frac{f(x,y)\,\sqrt{1-x^2}\sqrt{1-y^2}}{p(x,w)p(x,1/w)}dx\,dy\nonumber\\
&\quad+\frac{6272}{15\pi}\int_{-1}^1 \frac{f(x,y_1(x))(1-w_1(x)^2)w_2(x)\,\sqrt{1-x^2}}
{(53-28x)(1-w_1(x)w_2(x))(w_2(x)-w_1(x))}\,dx \label{7.36}
\end{align}
extends  $\cL$ from $\Rset_{2,2}[x,y]$ with moment matrix in \eqref{7.19} to a positive linear functional on $\Rset[x,y]$. Note that $y_1(x)$ in the last line of the above equation depends on $x$ unlike the formulas in the previous example.

Since $c_0+c_1x>0$ for $x\in(-1,1)$, it follows that $w_1<0<w_2$ when $x\in(-1,1)$.  Furthermore $|w_1w_2|>1$ for all real $c_0, c_1$ when $x\in(-1,1)$ by \eqref{7.23}, so that if one root is in magnitude less than one for all $x\in(-1,1)$ the other must be in magnitude greater than one. We look to see if $-1<w_1$ for $x\in(-1,1)$. 
This leads to the inequality
$$\frac{623-448x}{225}<c_0+c_1x.$$
This will be true if and only if $\frac{623-448x}{225}\leq c_0+c_1x$ for $x=\pm 1$ and at least one of these inequalities is strict. It is straightforward  to see that these two inequalities coincide with the ones in \eqref{7.34}, thus completing the proof of \eqref{7.35}.

For arbitrary $c_0,c_1\in\Rset$ and with $p(x,w)$ in \eqref{7.22} we define the functional $\cL_0:\Rset[x,y]\to \Rset$ by 
\begin{equation*}
\cL_0(f)=\frac{4}{\pi^2}\iint_{(-1,1)^2}f(x,y)\frac{\sqrt{1-x^2}\sqrt{1-y^2}}{p(x,w)p(x,1/w)}\,dx\,dy=\frac{2}{\pi}\int_{-1}^{1}\cI(f,x)\sqrt{1-x^2} dx 
\end{equation*}
where
\begin{align*}
\cI(f,x)&=\frac{1}{\pi}\int_{-\pi}^{\pi}\frac{f(x,\cos\phi)\sin^2\phi}{p(x,w)p(x,1/w)}\,d\phi \\
&=\frac{784}{15(53-28x) \pi i}\oint_{\Tset}\frac{f(x,\frac{1}{2}(w+\frac{1}{w}))(w-1/w)^2w}{\tau(x,w)}dw, 
\end{align*}
and
$$\tau(x,w)=(w-w_1)(w - w_2)(w - 1/w_1)(w -1/w_2).$$
In the region where equation~\eqref{7.25} holds $\cL=\cL_0$ and the residues of the above contour integral are at $w=0$, $w=1/w_1$, and $w=1/w_2$.  Hence
$$\cI(f,x)=\frac{784}{15(53-28x)}(R_0(x)+R_1(x)+R_2(x))$$
where 
\begin{align*}
R_0(x)&=2\text{res}_{w=0}\frac{f(x,\frac{1}{2}(w+\frac{1}{w})(w-1/w)^2w}{\tau(z,w)},\quad 
R_1(x)=2\frac{f(x,\frac{1}{2}(w_1+\frac{1}{w_1}))(1-w^2_1)}{(1-w_1w_2)(1-\frac{w_1}{w_2})},\\
R_2(x)&=2\frac{f(x,\frac{1}{2}(w_2+\frac{1}{w_2}))(1-w^2_2)}{(1 -w_1w_2)(1-\frac{w_2}{w_1})}.
\end{align*}

We now consider the  region where \eqref{7.34} holds.
In this case $|w_1|<1<|w_2|$ and the residue at $w=1/w_1$ is replaced by the residue at $w=w_1$ which gives
\begin{equation*}
\hat R_1(x)=-2\frac{f(x,\frac{1}{2}(w_1+\frac{1}{w_1}))(1-w^2_1)}{(1-w_1w_2)(1-\frac{w_1}{w_2})}=-R_1(x).
\end{equation*}
Therefore
$$\cL(f)-\cL_0(f)=\frac{6272}{15\pi}\int_{-1}^1\frac{f(x,y_1(x))(1-w_1(x)^2)\sqrt{1-x^2}}{(53-28x)(1-w_1(x)w_2(x))(1-\frac{w_1(x)}{w_2(x})}dx,$$
establishing \eqref{7.36}.

\section*{Acknowledgements}
The authors thank an anonymous referee for thoughtful suggestions to improve the presentation of the paper.


\begin{thebibliography}{xx}


\bibitem{Ber} Ju.~M.~Berezans\cprime ki\u{\i}, {\em Expansions in eigenfunctions of selfadjoint operators}, Translations of Mathematical Monographs, Vol. 17, American Mathematical Society, Providence, R.I., 1968.

\bibitem{BCGM}  E.~Berriochoa, A. Cachafeiro, J. M. Garc\'{\i}a-Amor and F. Marcell\'an, {\em New quadrature rules for Bernstein measures on the interval $[-1,1]$}, Electron. Trans. Numer. Anal. 30 (2008), 278--290. 

\bibitem{DPS} D.~Damanik, A.~Pushnitski and B.~Simon, {\em The analytic theory of matrix orthogonal polynomials},  Surv. Approx. Theory 4 (2008), 1--85.

\bibitem{DS} D.~Damanik and B.~Simon, {\em Jost functions and Jost solutions for Jacobi matrices. II. Decay and analyticity}, Int. Math. Res. Not. 2006, Art. ID 19396, 32 pp. 

\bibitem{DGIM} A. Delgado, J. Geronimo, P. Iliev and F. Marcell\'an, {\em Two variable orthogonal polynomials and structured matrices}, SIAM J. Matr. Anal. Appl. 28 (2006), no. 1, 118--147.

\bibitem{DGIX} A. Delgado, J. Geronimo, P. Iliev and Y. Xu,  {\em On a two variable class of Bernstein--Szeg\H{o} measures}, Constr. Approx. 30 (2009), no. 1, 71--91.

\bibitem{DGK} P.~Delsarte,  Y.~Genin and  Y.~Kamp, {\em Planar least squares inverse polynomials: Part I -- Algebraic properties}, 
IEEE Trans. Circuits and Systems 26 (1979), no. 1, 59--66.

\bibitem{DTV} H. Dette, D. Tomecki and M. Venker, {\em Random moment problems under constraints},
Ann. Probab. 48 (2020), no. 2, 672--713. 

\bibitem{vD} J. F.  van Diejen, {\em Spectral analysis of an open $q$-difference Toda chain with two-sided boundary interactions on the finite integer lattice}, J. Spectr. Theory 13 (2023), no. 4, 1261--1280.

\bibitem{vDE1} J. F.  van Diejen and E. Emsiz, {\em Cubature rules for unitary Jacobi ensembles}, Constr. Approx. 54 (2021), no. 1, 145--156.

\bibitem{vDE2} J. F.  van Diejen and E. Emsiz, {\em Exact cubature rules for symmetric functions}, Math. Comp. 88 (2019), no. 317, 1229--1249. 

\bibitem{DN} J. Dombrowski and P. Nevai, {\em Orthogonal polynomials, measures and recurrence relations}, SIAM J. Math. Anal. 17 (1986), no. 3, 752--759.

\bibitem{DX} C. F. Dunkl and Y. Xu, {\em Orthogonal polynomials of several variables}, 2nd edition, Encyclopedia of Mathematics and its Applications 155, Cambridge University Press, 2014.

\bibitem{Fejer} L.~Fej\'{e}r, {\em \"{U}ber trigonometrische Polynome}, J. Reine Angew. Math. 146 (1916), 53--82.

\bibitem{GK} M.~I.~Gekhtman and A.~A.~Kalyuzhny, {\em On the orthogonal polynomials in several variables}, Integral Equations Operator Theory 19 (1994), no. 4, 404--418. 

\bibitem{Ger} J. S.~Geronimo, {\em Scattering theory and matrix orthogonal polynomials on the real line}, Circuits Systems and Signal Processing 1 (1982), 471--495.

\bibitem{GC} J. S.~Geronimo and K. M.~Case, {\em Scattering theory and polynomials orthogonal on the real line}, Trans. Amer. Math. Soc. 258 (1980), no. 2, 467--494.

\bibitem{GeI1} J.~S.~Geronimo and P.~Iliev, Two variable deformations of the Chebyshev measure, in: {\em ``Integrable Systems and Random Matrices: In Honor of Percy Deift''}, pp. 197--213, Contemp. Math. 458, Amer. Math. Soc., Providence RI, 2008.

\bibitem{GeI2} J. S.~Geronimo and P.~Iliev, {\em Fej\'er--Riesz factorizations and the structure of bivariate polynomials orthogonal on the bi-circle},  J. Eur. Math. Soc. (JEMS) 16 (2014), 1849--1880.

\bibitem{GeI3}  J. S.~Geronimo and P.~Iliev, {\em Bernstein-Szeg\H{o} measures, Banach algebras, and scattering theory}, Trans. Amer. Math. Soc. 369 (2017), no. 8, 5581--5600. 

\bibitem{GeIK} J. S.~Geronimo, P.~Iliev and G.~Knese, {\em Polynomials with no zeros on a face of the bidisk}, J. Funct. Anal. 270 (2016), 3505--3558.

\bibitem{GeWo} J. S.~Geronimo and H.~Woerdeman, {\em Positive extensions, Fej\'er-Riesz factorization and autoregressive filters in two variables}, Ann. of Math. (2) 160 (2004), 839--906.

\bibitem{GeWo2} J.~Geronimo and H.~J.~Woerdeman, {\em Two variable orthogonal polynomials on the bicircle and structured matrices}, SIAM J. Matrix Anal. Appl. 29 (2007), no. 3, 796--825.

\bibitem{GS}  L. Ya.~Geronimus and G.~Szeg\H{o}, {\em Two papers on special functions}, American Mathematical Society Translations, Ser. 2, Vol. 108, 1977, ii+130 pp. 

\bibitem{Jackson} D.~Jackson, {\em Formal properties of orthogonal polynomials in two variables}, 
Duke Math. J. 2 (1936), 423--434. 

\bibitem{Kozhan} R.~Kozhan, {\em  Jost asymptotics for matrix orthogonal polynomials on the real line}, Constr. Approx. 36 (2012), 267--309.

\bibitem{Kozhane}  R.~Kozhan, {\em Correction to: Jost asymptotics for matrix orthogonal polynomials on the real line}, Constr. Approx. 58 (2023), 545--549.

\bibitem{Lander} F. I. Lander, {\em The Bezout\^{u}iant' and the inversion of Hankel and Toeplitz matrices}, Mat. Issled. 9 (1974), 69--87, 249--250, in Russian.

\bibitem{LL1}  E. Levin and D. S. Lubinsky, {\em Orthogonal polynomials for exponential weights}, CMS Books in Mathematics/Ouvrages de Math\'ematiques de la SMC 4, Springer-Verlag, New York, 2001.

\bibitem{LL2}  E. Levin and D. S. Lubinsky, {\em Bounds and asymptotics for orthogonal polynomials for varying weights},  SpringerBriefs in Mathematics, Springer, 2018.

\bibitem{NJ} R. G.~Newton and R.~Jost, {\em The construction of potentials from the S-matrix for systems of differential equations}, 
Nuovo Cimento (10) 1 (1955), 590--622. 

\bibitem{Peherstorfer} F. Peherstorfer, {\em On the remainder of Gaussian quadrature formulas for Bernstein-Szeg\H{o} weight functions}, Math. Comp. 60 (1993), no. 201, 317--325. 

\bibitem{Riesz} F.~Riesz, {\em \"{U}ber ein Problem des Herrn Carath\'{e}odory},  J. Reine Angew. Math. 146  (1916), 83--87.

\bibitem{Strintzis} M.~G.~Strintzis, {\em Tests of stability of multidimensional filters}, IEEE Trans. Circuits and Systems CAS-2 (1977), no. 8, 432--437.

\bibitem{Suetin} P. K. Suetin, {\em Orthogonal polynomials in two variables}, translated from the 1988 Russian original by E. V. Pankratiev, Analytical Methods and Special Functions, vol. 3, Gordon and Breach Science Publishers, 1999.

\bibitem{Szego} G.~Szeg\H{o}, {\em Orthogonal polynomials}, 4th ed., Amer. Math. Soc. Coll. Publ. Vol. 23, Providence, RI, 1975.

\bibitem{VanAssche} W. Van Assche, {\em Asymptotics for orthogonal polynomials}, Lecture Notes in Mathematics, 1265. Springer-Verlag, Berlin, 1987.



\bibitem{Xu} Y. Xu, Orthogonal polynomials of several variables. {\em Encyclopedia of special functions: the Askey-Bateman project. Vol. 2. Multivariable special functions}, 19--78, Cambridge Univ. Press, Cambridge, 2021. 

\bibitem{YK}  D.~C.~Youla and N.~N.~Kazanjian, {\em Bauer-type factorization of positive matrices and the theory of matrix polynomials orthogonal on the unit circle}, IEEE Trans. Circuits and Systems CAS-2 (1978), no. 2, 57--69. 

\end{thebibliography}
\end{document}